\documentclass[reqno]{amsart}
\usepackage{amsfonts}
\usepackage{amsmath}
\usepackage{amsthm, amscd}
\usepackage{mathtools}
\usepackage{amssymb}
\usepackage{mathrsfs}
\usepackage{overpic}

\usepackage[alphabetic]{amsrefs}
\usepackage{etex}
\usepackage{hyperref}
\hypersetup{
  colorlinks = true,
  linkcolor  = black
}
\usepackage{color}

\usepackage{tikz}
\usepackage{xypic}
\usepackage{overpic}

\allowdisplaybreaks

\theoremstyle{plain}\newtheorem{Theorem}{Theorem}[section]
\theoremstyle{plain}\newtheorem{Corollary}[Theorem]{Corollary}
\theoremstyle{plain}\newtheorem{Lemma}[Theorem]{Lemma}
\theoremstyle{plain}\newtheorem{Definition}[Theorem]{Definition}
\theoremstyle{plain}\newtheorem{Proposition}[Theorem]{Proposition}
\theoremstyle{plain}
\theoremstyle{plain}
\theoremstyle{plain}\newtheorem*{Theorem*}{Theorem}

\makeatletter
\newtheorem*{rep@theorem}{\rep@title}
\newcommand{\newreptheorem}[2]{%
\newenvironment{rep#1}[1]{%
 \def\rep@title{#2 \ref{##1}}%
 \begin{rep@theorem}}%
 {\end{rep@theorem}}}
\makeatother
\theoremstyle{plain}\newreptheorem{theorem}{Theorem}

\theoremstyle{remark}\newtheorem{Addendum}[Theorem]{Addendum}
\theoremstyle{remark}\newtheorem{remark}[Theorem]{Remark}
\theoremstyle{remark}\newtheorem{example}[Theorem]{Example}

\numberwithin{equation}{section}

\DeclareMathOperator{\ima}{Im}

\DeclareMathOperator{\SU}{SU}
\DeclareMathOperator{\SO}{SO}
\DeclareMathOperator{\Spin}{Spin}
\DeclareMathOperator{\Stab}{Stab}
\DeclareMathOperator{\Orb}{Orb}

\DeclareMathOperator{\Crit}{Crit}

\DeclareMathOperator{\Bif}{Bif}
\DeclareMathOperator{\Ad}{Ad}

\DeclareMathOperator{\Hom}{Hom}
\DeclareMathOperator{\Hess}{Hess}

\DeclareMathOperator{\id}{id}
\DeclareMathOperator{\ind}{ind}
 
\DeclareMathOperator{\CS}{CS} 
\DeclareMathOperator{\hol}{Hol} 
\DeclareMathOperator{\sym}{Sym}
\DeclareMathOperator{\U}{U}
\DeclareMathOperator{\s}{S}
\DeclareMathOperator{\diag}{Diag}
\DeclareMathOperator{\mat}{Mat}
\DeclareMathOperator{\Conj}{Conj}
\DeclareMathOperator{\supp}{supp}
\DeclareMathOperator{\grad}{grad}
\DeclareMathOperator{\mult}{mult}

\newcommand{\bC}{\mathbb{C}}
\newcommand{\bH}{\mathbb{H}}
\newcommand{\bK}{\mathbb{K}}

\newcommand{\bR}{\mathbb{R}}
\newcommand{\bZ}{\mathbb{Z}}

\newcommand{\clA}{\mathcal{A}}

\newcommand{\clC}{\mathcal{C}}
\newcommand{\clH}{\mathcal{H}}
\newcommand{\clG}{\mathcal{G}}

\newcommand{\clM}{\mathcal{M}}
\newcommand{\clP}{\mathcal{P}}
\newcommand{\clR}{\mathcal{R}}
\newcommand{\clS}{\mathcal{S}}
\newcommand{\clT}{\mathcal{T}}

\newcommand{\clF}{\mathcal{F}}

\newcommand{\frg}{\mathfrak{g}}
\newcommand{\frp}{\mathfrak{p}}

\newcommand{\fru}{\mathfrak{u}}
\newcommand{\frv}{\mathfrak{v}}

\begin{document}

\author{Shaoyun Bai}
\address{Department of Mathematics, MIT, Boston, MA, 02139, USA}
\email{shaoyunb@mit.edu}
\author{Boyu Zhang}
\address{Department of Mathematics, The University of Maryland at College Park, Maryland 20742, USA}
\email{bzh@umd.edu}
\title{Equivariant Cerf theory and perturbative $\SU(n)$ Casson invariants}
\begin{abstract}
We develop an equivariant Cerf theory for Morse functions on finite-dimensional manifolds with group actions, and adapt the technique to the infinite-dimensional setting to study the moduli space of perturbed flat $\SU(n)$--connections. As a consequence, we prove the existence of perturbative $\SU(n)$ Casson invariants on integer homology spheres for  all $n\ge 3$, and  write down an explicit formula when $n=4$. This generalizes the previous works of Boden-Herald \cite{boden1998the} and Herald \cite{herald2006transversality}.
\end{abstract}

\maketitle

\setcounter{tocdepth}{1}

\section{Introduction}

The Casson invariant is an invariant for oriented integer homology 3-spheres introduced by Casson in 1985 (see \cite{akbulut1990casson} or \cite{marin1988nouvel}).  Taubes \cite{taubes1990casson} proved that the Casson invariant is equal to half of the number of points (counted with signs) of the moduli space of irreducible critical points of the perturbed Chern-Simons functionals with $\SU(2)$--gauge. Boden and Herald \cite{boden1998the} studied the case when the gauge group is $\SU(3)$ and constructed a perturbative $\SU(3)$ Casson invariant for integer homology spheres. Variations of the $\SU(3)$ Casson invariant were later given by Boden-Herald-Kirk \cite{BHK2001} and Cappell-Lee-Miller \cite{cappell2002perturbative}. However, the construction of perturbative $\SU(n)$ Casson invariants for $n\ge 4$ has remained open since then.

A different approach of generalizing the Casson invariant was introduced by Boyer-Nicas \cite{boyer1990varieties} and Walker \cite{walker1992extension}, where one studies the intersection of representation varieties of handlebodies. Cappell-Lee-Miller \cite{cappell1990symplectic} announced a program of extending the Casson invariant to all oriented closed 3-manifolds and all semi-simple Lie groups using Bierstone transversality. The program was carried out in detail for $\SO(3)$, $\U(2)$, $\Spin(4)$, and $\SO(4)$ by Curtis \cite{curtis1994generalized}.

The main difficulty of constructing $\SU(n)$ Casson invariants using the Chern-Simons functional is that even if one could perturb it so that the moduli space of irreducible critical points is cut out transversely, the signed count of irreducible critical points depend on the perturbation. Therefore, one has to study the moduli space of both reducible and irreducible critical points, and understand the difference of the critical sets between different choices of perturbations, in order to write down a counting of critical points that is independent of the perturbation.

In the $\SU(3)$ case, Boden and Herald \cite{boden1998the} showed that the moduli space of critical points over a generic 1-parameter family of perturbations can only admit a particular type of bifurcation, and that one can write down a weighted counting of  critical points using spectral flow so that the counting does not change under the bifurcation, therefore the $\SU(3)$ Casson invariant is constructed. The constructions in \cite{BHK2001} and \cite{cappell2002perturbative} were  based on the same line of argument but used different  weight functions to make the resulting invariants behave better. Later, Herald \cite{herald2006transversality} studied the  bifurcations of moduli spaces for general gauge groups and characterized the possible bifurcations for $\SU(4)$, but the relation between the bifurcations and the spectral flows was not given, and the $\SU(4)$ Casson invariant was not discussed. The argument of \cite{herald2006transversality} relied on a technical property of $\SU(4)$--connections called ``sphere transitivity'', which is not satisfied by $\SU(n)$--connections in general.

In this paper, we give a complete description of the possible changes of the moduli space of critical points with different perturbations when the gauge group is $\SU(n)$, for all $n\ge 3$. We also compute the corresponding changes of the spectral flows. As a consequence, we prove that perturbative $\SU(n)$ Casson invariants exist for all $n\ge 3$ on integer homology spheres. We also write down an explicit formula of Casson invariant when $n=4$. Most of the arguments work for general three-manifolds and for arbitrary simple compact gauge groups, and we will state the results in the more general setting whenever possible. In fact, the only place that requires the manifold to be an integer homology sphere and that the gauge group to be $\SU(n)$ is in the construction of the equivariant index in Section \ref{subsec_gauge_inv_equi_ind}.

\begin{figure}
\label{figure_bif_SU5}
  \begin{overpic}[width=0.7\textwidth]{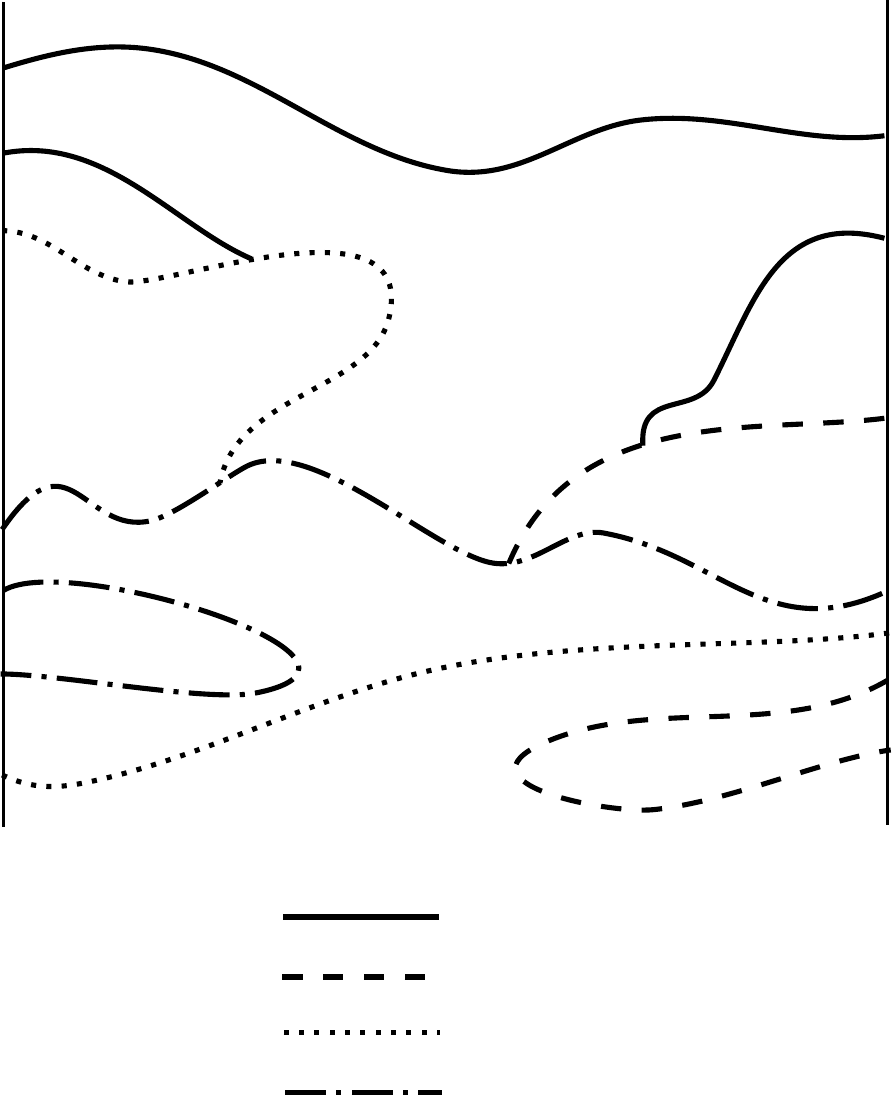}
  	\put(-3,20){$t=0$}
    \put(78,20){$t=1$}
    \put(45,15){$\mathbb{C}^5$}
    \put(45,10){$\mathbb{C}^4\oplus \mathbb{C}$}
    \put(45,5){$\mathbb{C}^3\oplus \mathbb{C}^2$}
    \put(45,-1){$(\mathbb{C}^2)^{\oplus 2}\oplus \mathbb{C}$}
  \end{overpic}
\caption{A possible bifurcation diagram for $\SU(5)$}
\end{figure}

In order to define the $\SU(n)$ Casson invariant, one needs to study the reducible connections with different possible stabilizers simultaneously. 
Figure \ref{figure_bif_SU5} illustrates some possible bifurcations of the moduli space over a 1-parameter family of perturbations when the gauge group is $\SU(5)$. Here, the notation $\bC^5$ means that the corresponding connection is irreducible; $\bC^4\oplus \bC$ means  it is given by the direct sum of an irreducible connection on the trivial $\bC^4$--bundle and a connection on the trivial $\bC$--bundle; $\bC^3\oplus \bC^2$ means  it is given by the direct sum of an irreducible connection on the trivial $\bC^3$--bundle and an irreducible connection on the trivial $\bC^2$--bundle; and $(\bC^2)^{\oplus 2}\oplus \bC$ means  the connection is given by $B_2\oplus B_2\oplus B_1$, where $B_2$ is an irreducible connection on the trivial $\bC^2$--bundle, and $B_1$ is a connection on the trivial $\bC$--bundle.

In the definition of the $\SU(3)$--Casson invariants \cite{boden1998the,BHK2001,cappell2002perturbative}, the weights on the reducible connections are given by the spectral flow, which assigns an integer to each critical point.
However, notice that in Figure \ref{figure_bif_SU5}, the moduli space of reducible connections of the form $(\bC^2)^{\oplus 2}\oplus \bC$ has two possible bifurcations: it can either bifurcate a reducible connection of the form $\bC^4\oplus \bC$, or a reducible connection of the form $\bC^3\oplus \bC^2$. Therefore, in order to keep track of this, we need to use a refinement of the spectral flow that can take values in higher-dimensional spaces. In Section \ref{subsec_equiv_spec_flow}, we will define an \emph{equivariant spectral flow}, which assigns to each critical point an element in the representation ring of the stabilizer. Similar to the classical spectral flow, the  equivariant spectral flow is not gauge invariant. In Section \ref{subsec_gauge_inv_equi_ind}, we add the equivariant spectral flow by another term given by the Chern-Simons functional to cancel the gauge ambiguity. As a result, we associate an \emph{equivariant index} to each critical orbit. The equivariant index takes value in $\widetilde{\clR}_{\SU(n)}$, which is a space defined by Definition \ref{def_tilde_R_G} using the representation rings of the subgroups of $\SU(n)$. The equivariant index of the orbit of $B$ will be denoted by $\ind B$. 

We say that a critical point of the perturbed Chern-Simons functional is \emph{non-degenerate}, if the Hessian of the perturbed Chern-Simons functional at the point has the minimum possible kernel. The precise definition will be given by Definition \ref{def_non-deg_perturbed_connection}. The next result will be proved as an immediate consequence of Theorem \ref{thm_change_total_index_gauge}:

\begin{Theorem}
\label{thm_existence_casson_intro}
For every $n\ge 3$, there exists a function 
	$$
	w: \widetilde{\clR}_{\SU(n)} \to \bC
	$$
	with the following property. 
	Suppose $Y$ is an integer homology sphere, let 
	$$P=\SU(n)\times Y$$ 
	be the trivial $\SU(n)$--bundle over $Y$, let $\theta$ be the trivial connection of $P$.
	Then for a generic holonomy perturbation $\pi$, the critical set of the perturbed Chern-Simons functional consists of finitely many non-degenerate orbits. 
	Let $\clM_\pi$ be the moduli space of critical points of the Chern-Simons functional perturbed by $\pi$, and decompose $\clM_\pi$ as 
	$$\clM_\pi=\clM_\pi^*\sqcup \clM_\pi^r,$$
	where $\clM_\pi^*$ consists of irreducible critical orbits, and $\clM_\pi^r$ consists of reducible critical orbits. Then for $\pi$ sufficiently small, the sum
	\begin{equation}
	\label{eqn_formula_lambda_intro}
			\lambda_{w}:= \sum_{[B]\in \clM^*} (-1)^{Sf(B,\pi)} + \sum_{[B]\in \clM^r} e^{\pi i\cdot \CS(\hat B)/(\pi^2)}\cdot w(\ind B)
	\end{equation}
	is independent of $\pi$, where $Sf(B,\pi)\in \bZ$ is the (classical) spectral flow from the extended Hessian (see \eqref{eqn_def_K}) of the $\pi$-perturbed Chern--Simons functional at the critical point $B$ to the extended Hessian of the unperturbed Chern--Simons functional at the trivial connection, and $\hat B$ is a flat connection close to $B$.
	 \end{Theorem}
The term $\pi^2$ in \eqref{eqn_formula_lambda_intro} comes from the normalization convention of the definition of the Chern-Simons functional in Equation \eqref{eqn_def_chern_simons}. 

When $n=3$, the $\SU(3)$--Casson invariant of Boden-Herald \cite{boden1998the} can be arranged into the form of \eqref{eqn_formula_lambda_intro}.

The function $w$ in Theorem \ref{thm_existence_casson_intro} is constructed by induction and hence it is possible to write down the formula for any given value of $n$. 
We will write down an explicit formula when $n=4$. \\

The proof of the main result is organized as follows: In Section \ref{sec_G_Morse} and Section \ref{sec_equivariant_Cerf}, we prove an analogous result on finite-dimensional manifolds. In Section \ref{sec_holonomy_perturbation}, we develop the necessary transversality properties for the holonomy perturbations. In Section \ref{sec_su(n)_casson}, we apply a Kuranishi reduction argument to prove the main results by reducing to the finite-dimensional case. In  Section \ref{sec_examples}, we characterize all the possible bifurcations of moduli space in $\SU(n)$--gauge and write down an explicit formula for the $\SU(4)$ Casson invariant.\\

The finite-dimensional results established in Section \ref{sec_G_Morse} and Section \ref{sec_equivariant_Cerf} can be thought of as an equivariant version of Cerf theory \cite{cerf1970stratification} and may be of independent interest. We briefly summarize the main result here. Suppose $G$ is a compact Lie group acting on a smooth closed oriented manifold $M$. From \cite{wasserman1969equivariant}, a smooth $G$--invariant function $f: M \to \bR$ is called \emph{$G$--Morse} if for all $p \in M$ being a critical point of $f$, the kernel of the Hessian of $f$ at $p$ is equal to the tangent space of the $G$--orbit passing through $p$. 
Suppose $p$ is a critical point of a $G$--Morse function,  we will define the \emph{equivariant index} of $p$ in Definition \ref{def_equi_index_finite_dim}, which is given by the subspace of $T_pM$ spanned by the negative eigenvectors of $\Hess_p f$, as a representation of the stabilizer of $p$. We use $\ind p\in\clR_G$ to denote the equivariant index of $p$, where $\clR_G$ is the set given by Definition \ref{def_clR_G} that consists of isomorphism classes of representations. 
It will be shown in Section \ref{sec_G_Morse} that the equivariant index is constant on the orbit of $p$. Let $\Conj(G)$ be the set of conjugation classes of closed subgroups of $G$, and let $\bZ\Conj(G)$ be the free abelian group generated by $\Conj(G)$. The set $\Conj(G)$ embeds canonically in $\clR_G$ by taking the zero representations, and we will regard $\Conj(G)$ as a subset of $\clR_G$.  Theorem \ref{thm_equivairant_cerf_finite_dim} will give a complete description of the possible differences of critical sets of different $G$--Morse functions on $M$. We will then prove the following result using Theorem \ref{thm_equivairant_cerf_finite_dim}.

\begin{Theorem}
\label{thm_finite_dim_Cerf_into}
Given $G$, there exists a unique map $\eta:\clR_G\to \bZ\Conj(G)$ with $\eta (\sigma)=\sigma$ for all $\sigma\in \Conj(G)$, such that the following holds. For every closed $G$--manifold $M$, let $f$ be a $G$--Morse function on $M$, let $\Crit(f)$ be the set of critical orbits of $f$, then the sum
$$
\sum_{[p]\in\Crit(f)}\eta(\ind p)
$$
is independent of the function $f$.
\end{Theorem}

\begin{remark}
After the completion of the first version of this paper, we learned that the invariant given in Theorem \ref{thm_finite_dim_Cerf_into} is closely related to the notion of universal equivariant Euler characteristic from \cite[Chapter IV, 1]{tom-Dieck}, whose definition does not use Morse theory. See Addendum \ref{add:universal} for a more detailed discussion.
\end{remark}

\begin{remark}
	Let $\eta$ be given by Theorem \ref{thm_finite_dim_Cerf_into}.
For every homomorphism $\varphi$ from $\bZ\Conj(G)$ to $\bC$, the composition $\varphi\circ \eta$ defines  $\bC$--valued weight functions. Although the $\bZ\Conj(G)$--valued weight function $\eta$ is unique by Theorem \ref{thm_finite_dim_Cerf_into}, we see that $\bC$--valued weight functions are not unique.

On the other hand, Theorem \ref{thm_existence_casson_intro} is only stated with respect to $\bC$--valued weights, and it claims the existence but not the uniqueness of the weight functions. The reason that our Casson invariant doesn't lift to $\bZ\Conj(G)$ is because the spectral flow to the trivial connection is not gauge invariant, and we will use an additional term involving the Chern--Simons functional to cancel the gauge dependency. This leads to some additional algebraic difficulties that is discussed in  Section \ref{subsec_extend_ind_over_R}. As a consequence, the $\bC$--valued weights defined in  Theorem \ref{thm_existence_casson_intro} do not in general lift to  $\bZ\Conj(G)$. It may be possible to use the methods from \cite{BHK2001} to define $\bZ$--valued or $\bZ\Conj(G)$--valued Casson invariants.

%
\end{remark}

We finish the introduction with several additional remarks.

For the proofs of both Theorem \ref{thm_existence_casson_intro} and Theorem \ref{thm_finite_dim_Cerf_into}, it is crucial to establish certain transversality results under group actions (Lemmas \ref{lem_one_degeneracy_codim}, \ref{lem_two_degeneracy_codim}, \ref{lem_gauge_one_degeneracy_codim}, and \ref{lem_gauge_two_degeneracy_codim}). Our argument is adapted from an equivariant transversality argument of Wendl \cite[Theorem D]{wendl2016transversality}, which can be further dated back to Taubes \cite{taubes1996counting}.
	This argument seems to have simplified an earlier argument of Herald \cite{herald2006transversality}: roughly speaking, Herald's transversality argument would require considerations of higher order derivatives of the Chern--Simons functional (cf. \cite[Definition 20]{herald2006transversality}), while the transversality argument we use only considers derivatives up to the second order.

The sum over the reducible connections in \eqref{eqn_formula_lambda_intro} can be thought of as a ``correction term'' for the counting of irreducible connections. 
Although the construction of correction terms has its root in many papers on gauge theory, for example in \cite{boden1998the} and \cite{mrowka2011seiberg}, this paper shows that it is possible to construct correction terms for reducible connections on all strata simultaneously.
  Similar correction terms are also being sought for in other fields, for example, in the construction of enumerative invariants of calibrated 3-manifolds in $G_2$-manifolds \cite{doan2017counting}, and in the construction of integer-valued refinements of Gromov-Witten invariants.  We hope this paper can provide some insight into those questions as well.
  
  Nakajima \cite[Section 1(iv)]{nakajima2016towards} conjectured that $\SU(n)$ Casson invariants are related to the counting of solutions to the generalized Seiberg-Witten equations. In particular, the bifurcation phenomenon of the moduli space of perturbed flat $\SU(n)$--connections is conjecturally related to the non-compactness of the moduli space of generalized Seiberg-Witten equations. By this conjecture, the correction term in \eqref{eqn_formula_lambda_intro} could potentially be related to certain (conjectural) correction terms in the generalized Seiberg-Witten theory.

 It is also a natural question to ask about the \emph{properties} of the  $\SU(n)$ Casson invariants. The follow-up work of this paper \cite{bai2021casson} shows that the $\SU(n)$ Casson invariants can be understood as a version of equivariant intersection number of character varieties. Such an interpretation could potentially open up better structural understandings. In particular, it would be interesting to see if the invariants admit any surgery formulas. One can also ask about the asymptotic behavior of the $\SU(n)$ Casson invariant when $n\to\infty$. These questions will not be discussed in the current paper.

\subsection*{Acknowledgements} The first author thanks his Ph.D. advisor John Pardon for constant support and encouragements. He would also like to thank Mohan Swaminathan for several interesting discussions. This project was inspired by a joint work of the first author and Mohan Swaminathan on defining integer-valued refinements of Gromov-Witten invariants for Calabi-Yau $3$--folds. 

\section{$G$--Morse functions}
\label{sec_G_Morse}

Suppose $M$ is a closed smooth manifold, let $f_0$, $f_1$ be two (classical) Morse functions on $M$. Cerf's theorem \cite{cerf1970stratification} states that a generic 1-parameter family from $f_0$ to $f_1$ has finitely many degeneracies, and each degeneracy corresponds to a birth-death transition on the critical set. More precisely, suppose 
$$F:[0,1]\times M\to \bR$$ 
is a generic smooth map such that  $F(0,x)=f_0(x)$ and $F(1,x)=f_1(x)$, then $F(t,\cdot)$ is Morse for all but finitely many values of $t$; for every $t_0$ such that $F(t_0,\cdot)$ is not Morse, there is exactly one degeneracy point where $F$ is locally conjugate to 
\begin{equation}
\label{eqn_Cerf_normal_form}
c+x_1^3+\epsilon_1(t-t_0)x_1 + \epsilon_2 x_2^2+\cdots \epsilon _n x_n^2,
\end{equation}
with $\epsilon_i\in \{-1,1\}$ and $c\in \bR$ being constants.

Suppose $p$ is a critical point of a Morse function $f$, recall that the \emph{index} of $f$ at $p$ is defined to be the number of negative eigenvalues of the Hessian of $f$ at $p$ counted with multiplicities. By \eqref{eqn_Cerf_normal_form}, each time $F$ goes through a degeneracy point, it creates or cancels a pair of critical points with consecutive indices. Let $n_k(f)$ be the number of critical points of a Morse function $f$ with index $k$. Then by Cerf's theorem, for any two Morse functions $f_0$ and $f_1$ on $M$, we have
$$
\sum_{k=0}^{\dim M} (-1)^k n_k(f_0) = \sum_{k=0}^{\dim M} (-1)^k n_k(f_1). 
$$
As a consequence, the value of 
\begin{equation}
\label{eqn_alternating_sum_Morse}
	\sum_{k=0}^{\dim M} (-1)^k n_k(f)
\end{equation}
 does not depend on the Morse function $f$ and hence is an invariant of $M$.
It is well-known that \eqref{eqn_alternating_sum_Morse} is equal to the Euler number of $M$. 

The purpose of Section \ref{sec_G_Morse} and Section \ref{sec_equivariant_Cerf} is to generalize the results above and establish a Cerf theory for manifolds with group actions. Section \ref{sec_G_Morse} will introduce the necessary terminologies, and Section \ref{sec_equivariant_Cerf} will state and prove the main results.

\subsection{The equivariant topology of $M$}
For the rest of Section \ref{sec_G_Morse} and Section \ref{sec_equivariant_Cerf},
$G$ will denote a compact Lie group (which can be disconnected), and $M$ will denote a compact smooth manifold possibly with boundary. We also fix a smooth $G$--action and a smooth $G$--invariant Riemannian metric on $M$. By the invariance of domain, if $\partial M\neq \emptyset$, then $\partial M$ is preserved by the $G$--action.

This subsection establishes the basic topological properties of the $G$--action on $M$. Most of the results are standard, and the reader may refer to, for example, \cite{wasserman1969equivariant}, for a more complete discussion.

\begin{Lemma}
\label{lem_tbl_nbhd_partial_M}
$\partial M$ has a neighborhood that is $G$--equivariantly diffeomorphic to $(-1,0]\times \partial M$, where $G$ acts on $(-1,0]$ trivially.
\end{Lemma}

\begin{proof}
	Let $\nu(\partial M)$ be the orthogonal complement of $T(\partial M)$ in $TM|_{\partial M}$. Since the metric on $M$ is $G$--invariant, the exponential map on $\nu(\partial M)$ gives a $G$--equivariant embedding from $(-\epsilon,0]\times \partial M$ to $M$ for $\epsilon$ sufficiently small.
\end{proof}

\begin{Definition}
\label{def_stab_orb_finite}
	For $p\in M$, define 
	\begin{align*}
		\Stab(p):=\{g\in G| g(p) = p\}, \\
		\Orb(p):=\{g(p) \in M| g\in G\}.
	\end{align*}
	Then $\Stab(p)$ is a closed subgroup of $p$, and $\Orb(p)$ is diffeomorphic to $G/\Stab(p)$. 
If $p\in M-\partial M$, then $\Orb(p)$ is a closed submanifold of $M$; if $p\in \partial M$, then $\Orb(p)$ is a closed submanifold of $\partial M$.
\end{Definition}

\begin{Definition}
	\label{defn_Sp}
	Suppose $p\in M-\partial M$, let $S_p\subset T_pM$ be the orthogonal complement of $T_p\Orb(p)$ in $T_pM$.
\end{Definition}

Since the metric on $M$ is $G$--invariant, $S_p$ is invariant under the action of $\Stab(p)$ and hence can be regarded as an orthogonal representation of $\Stab(p)$. Viewing $G$ as a principal $\Stab(p)$--bundle over $\Orb(p)$, then $G\times _{\Stab(p)} S_p$ is an associated vector bundle over $\Orb(p)$.
The next lemma is a well-known property of compact Lie group actions, and the reader may refer to, for example, \cite[Theorem 2.4.1]{duistermaat2012lie}, for a proof.

\begin{Lemma}
\label{lem_local_slice}
Suppose $p\in M-\partial M$, then $\Orb(p)$ has a $G$--invariant open tubular neighborhood that is $G$--equivariantly diffeomorphic to a $G$--invariant open neighborhood of the zero section of $G\times _{\Stab(p)} S_p$, where $\Orb(p)$ embeds as the zero section.
\qed
\end{Lemma}



We introduce the following definitions:

\begin{Definition}
\label{def_slice}
Suppose $p\in M-\partial M$. Let $D$ be an embedded closed disk in $M$ with dimension equal to the dimension of $S_p$.  Then $D$ is called a \emph{slice of $p$} if the following hold:
\begin{enumerate}
	\item $D$ intersects $\Orb(p)$ transversely at $p$, and $D$ is invariant under the $\Stab(p)$--action;
	\item $D$ is $\Stab(p)$--equivariantly diffeomorphic to a closed ball in $S_p$ where $p\in D$ is mapped to $0$;
	\item The map
		\begin{align*}
		 \varphi_D: G\times_{\Stab(p)} D & \to M \\
		[g,x] &\mapsto g(x)
	\end{align*}
	is a smooth embedding.
\end{enumerate}
\end{Definition}

By Lemma \ref{lem_local_slice}, every interior point of $M$ has a slice.

\begin{Definition}
\label{def_U_p(D)}
	Suppose $D$ is a slice of $p$, and suppose $\varphi_D$ is the diffeomorphism given by Definition \ref{def_slice}. Define $U_p(D)$ to be the image of $\varphi_D$ in $M$.
\end{Definition}

\begin{remark}
	\label{rem_local_G-invariant_function_slice}
 By definition, $U_p(D)$ is a closed $G$--invariant neighborhood of $\Orb(p)$. The set of $G$--invariant functions on $U_p(D)$ is in one-to-one correspondence with the set of $\Stab(p)$--invariant functions on $D$ via restrictions to $D$.
\end{remark}

\begin{Lemma}
\label{lem_conjugate_rep_same_nbhd}
	Suppose $H_1,H_2$ are closed subgroups of $G$, and suppose there exists $u\in G$ such that $$H_1=uH_2u^{-1}.$$ 
	Let
	 $$\rho_1:H_1\to \Hom(V,V)$$ 
	be a representation of $H_1$, and let $$\rho_2:H_2\to \Hom(V,V)$$ be the representation of $H_2$ defined by 
		$$\rho_2(h)=\rho_1(uhu^{-1}).$$
		Then $G\times_{\rho_1} V$ is $G$--equivariantly diffeomorphic to $G\times_{\rho_2} V$. 
\end{Lemma}

\begin{proof}
	Consider the map 
	\begin{align*}
	\varphi: G\times V &\to G\times_{\rho_2} V \\
	(g,v) &\mapsto [gu,v].
	\end{align*}
	For $h\in H_1$, we have
	\begin{align*}
	\varphi(gh,v) & = [ghu,v] = [gu\cdot u^{-1}hu,v]
	\\ & = [gu,\rho_2(u^{-1}hu)v] =[gu,\rho_1(h)v]= \varphi\big(g,\rho_1(h)v\big).
	\end{align*}
	Therefore $\varphi$ induces a map $\bar\varphi$ from $G\times_{\rho_1} V$ to $G\times_{\rho_2} V$. It is straightforward to verify that $\bar\varphi$ is a $G$--equivariant diffeomorphism.
\end{proof}

Lemma \ref{lem_local_slice} and
Lemma \ref{lem_conjugate_rep_same_nbhd} are the motivations of the following definition:

\begin{Definition}
\label{def_clR_G}
Let $\clR_G$ be the set of isomorphism classes of $(H,V,\rho)$, where $H$ is a closed subgroup of $G$, and $\rho:H\to\Hom(V,V)$ is a finite-dimensional representation of $H$. We say that $(H,V,\rho)$ is isomorphic to $(H',V',\rho')$, if there exists $g\in G$ and an isomorphism $\varphi:V\to V'$, such that
$$
	H' = g H g^{-1},
$$
and 
$$
\rho'(ghg^{-1}) =\varphi\circ \rho(h)\circ \varphi^{-1}.
$$
\end{Definition}

We also introduce the following notations for later reference:

\begin{Definition}
	Let $\Conj(G)$ be the set of conjugation classes of closed subgroups of $G$. Suppose $H$ is a closed subgroup of $G$, we will use $[H]\in \Conj(G)$ to denote the conjugation class of $H$.
\end{Definition}

\begin{Definition}
	Suppose $[H] \in \Conj(G)$. 
	Define $\clR_G([H])$ to be the subset of $\clR_G$ consisting of elements represented by the representations of $H$.
\end{Definition}

Let $\sigma_1,\sigma_2\in \clR_G([H])$, and suppose $\sigma_i$ is represented by $(H,V_i,\rho_i)$ for $i=1,2$. Then the direct sum of $\sigma_1$ and $\sigma_2$ is in general not well-defined, because there may exist an element $g\in G$ with $H=gHg^{-1}$, such that $h\mapsto \rho_1(h)\oplus \rho_2(h)$ and $h\mapsto \rho_1(h)\oplus \rho_2(ghg^{-1})$ are non-isomorphic representations of $H$ on $V_1\oplus V_2$. However, the following statement holds nonetheless. 

\begin{Lemma}
	\label{lem_defn_direct_sum}
Suppose $\sigma \in (H,V_1,\rho_1)$ represents an element in $\clR_G([H])$. 
Suppose $\tau\in \clR_G([G])$ is given by $(G,V_2,\rho_2)$. Then the element
\begin{equation}
	\label{eqn_direct_sum_defn}
\sigma\oplus \tau := [H,V_1\oplus V_2,\rho_1\oplus (\rho_2|_{H})]\in \clR_G([H])
\end{equation}
does not depend on the choice of the representatives. 
\end{Lemma}
\begin{proof}
Suppose $\tau$ is represented by another triple $(G, V_2', \rho_2')$ such that there exists $u \in G$ and an isomorphism $\varphi: V_2 \to V_2'$ with $\rho_2'(u g u^{-1}) = \varphi \circ \rho_2(g) \circ \varphi^{-1}$ for all $g \in G$. Then $(G,V_2,\rho_2)$ and $(G, V_2', \rho_2')$ are isomorphic as $G$--representations. Therefore \eqref{eqn_direct_sum_defn} does not depend on the choice of the representative of $G$. 

Now suppose $\sigma$ is represented by another triple $(H', V_1', \rho_1')$. Then there exists $u\in G$ and an isomorphism $\varphi:V_1\to V_1'$ such that $H' = uHu^{-1}$, and  $\rho_1'(u h u^{-1}) = \varphi \circ \rho_1(h) \circ \varphi^{-1}$ for all $h \in H$. Let $\psi = \varphi \oplus \rho_2(u): V_1\oplus V_2 \to V_1'\oplus V_2$. Then the pair $(u,\psi)$ defines an isomorphism between $[H,V_1\oplus V_2, \rho_1\oplus (\rho_2|_{H})]$ and $[H,V_1' \oplus V_2, \rho_1'\oplus (\rho_2|_{H})]$.
\end{proof}


\begin{Definition}
	\label{defn_dir_sum_clR}
Let $\bZ\clR_G$ denote the free abelian group generated by $\clR_G$. Define
\begin{equation}
\label{eqn_dir_sum_clR}
	\oplus : \bZ\clR_G\times \clR_G([G])\to \bZ\clR_G.
\end{equation}
to be the linear extension of the operator given by \eqref{eqn_direct_sum_defn}.
\end{Definition}

\begin{Definition}
\label{def_i^H_G}
	Suppose $H$ is a closed subgroup of $G$. Define 
	$$i^H_G:\clR_H\to\clR_G$$
	to be the tautological map by viewing subgroups of $H$ as subgroups of $G$. Then $i^H_G$ induces a homomorphism from $\bZ\clR_H$ to $\bZ\clR_G$, which we also denote by $i^H_G$.
\end{Definition}

\begin{Definition}
	Suppose $H$ is a compact Lie group, and let $\rho:H\to\Hom(V,V)$ be a finite-dimensional real representation of $H$. 
	\begin{enumerate}
		\item We say that $\rho$ is \emph{trivial}, if $\rho(h)= \id$ for all $h\in H$.
		\item  We say that $\rho$ \emph{has no trivial component}, if the isotypic decomposition of $(V,\rho)$ has no trivial component, or equivalently, if $\rho$ does not have non-zero fixed point.
	\end{enumerate}
\end{Definition}

We now return to the discussion of the topology of $M$.

\begin{Definition}
\label{def_stratify_M_representation}
For each $\sigma\in \clR_G$, define $M_\sigma$ to be the set of $p\in M-\partial M$ such that the $\Stab(p)$--representation $S_p$ represents the isomorphism class $\sigma$.
\end{Definition}

The following lemma is another standard property of compact Lie group actions, and the proof is essentially the same as \cite[Theorem 2.7.4]{duistermaat2012lie}.

\begin{Lemma}
\label{lem_M_sigma_basic_properties}
The decomposition $$M-\partial M=\bigcup_{\sigma\in \clR_G} M_\sigma$$ has the following properties:
\begin{enumerate}
	\item For every $\sigma$, the set $M_\sigma$ is a (not necessarily closed) $G$--invariant submanifold of $M-\partial M$.
	\item Suppose $p\in M_\sigma$. Then the action of $\Stab(p)$ on $T_p M_\sigma/ T_p\Orb(p)$ is trivial, and the action of $\Stab(p)$ on $T_p M/ T_p M_\sigma$ has no trivial component.
	\item Suppose $p\in M_\sigma$, let $[p]$ be the image of $p$ in $M_\sigma/G$. Let $D$ be a slice of $p$, let $D^0\subset D$ be the fixed-point subset of the $\Stab(p)$--action. Then $M_\sigma/G$ is a manifold, and $D^0$ maps diffeomorphically to a closed neighborhood of $[p]$ in $M_\sigma/G$ by the quotient map.
	\item There are only finitely many $\sigma$ such that $M_\sigma\neq \emptyset$.
\end{enumerate}
\end{Lemma}

\begin{proof}	
	Let $p\in M-\partial M$ and $g\in G$, then $\Stab(gp)=g\Stab(p)g^{-1}$, and $S_{gp}=(Tg)(S_p)$, where $Tg$ is the tangent map of the action by $g$. Therefore $S_p$ and $S_{gp}$ represent the same element in $\clR_G$, and hence $M_\sigma$ is $G$--invariant.
	
	Let $D$ be a slice of $p$, and let $U_p(D)$ be the neighborhood given by Definition \ref{def_U_p(D)}. Let $D^0$ be the fixed-point subset of $D$ with respect to the $\Stab(p)$--action. Then $M_\sigma\cap U_p(D)$ is given by $G\times _{\Stab(p)} D^0$. Therefore $M_\sigma$ is a submanifold of $M$, and Parts (1), (2), (3)  of the lemma are proved.

To prove Part (4) of the lemma, we apply induction on $\dim M$. The statement is obvious if $\dim M=0$. Now suppose the statement is true for $\dim M<k$, we show that it also holds for $\dim M=k$. For $p\in M$, let $U_p$ be as above, let $M'$ be the unit sphere of $S_p$. By the induction hypothesis on $(\Stab(p),M')$, we conclude that there are only finitely many $\sigma\in \clR_G$ such that 
	$$M_\sigma\cap (U_p-\Orb(p))\neq \emptyset.$$ 
	Therefore, there are only finitely many $\sigma\in \clR_G$ such that 
	$M_\sigma\cap U_p\neq \emptyset.$
	 The statement then follows from the compactness of $M$ and Lemma \ref{lem_tbl_nbhd_partial_M}.
\end{proof}

\begin{Definition}
Let $\nu(M_\sigma)$ be the orthogonal complement of $TM_\sigma$ in $TM|_{M_\sigma}$.
\end{Definition}

By definition, $\nu(M_\sigma)$ is a $G$--equivariant vector bundle over $M_\sigma$. By Part (2) of Lemma \ref{lem_M_sigma_basic_properties}, for each $p\in M_\sigma$, the action of $\Stab(p)$ on $\nu(M_\sigma)|_p$ has no trivial component.

\begin{remark}
By standard linear algebra, orthogonal complements are canonically isomorphic to quotient spaces. For example, 
 $\nu(M_\sigma)$ is canonically isomorphic to $TM|_{M_\sigma} / TM_\sigma$. In the following, we will frequently identify orthogonal complements with quotient spaces without further comments.
\end{remark}

\subsection{$G$--Morse functions}

Suppose $f$ is a $G$--invariant $C^2$ function on $M$, then the critical set of $f$ is invariant under the $G$--action. An orbit $\Orb(p)\subset M$ is called a \emph{critical orbit} if it consists of critical points. 
Let $p$ be a critical point of $f$. We use $\Hess_pf:T_pM\to T_pM$ to denote the Hessian of $f$ at $p$. Then $\Hess_pf$ is a $\Stab(p)$--equivariant self-adjoint map, and 
\begin{align*}
&\Hess_pf\,\big(T_p\Orb(p)\big)=\{0\},\\
&\Hess_pf\,(S_p)\subset S_p.
\end{align*}

We can now introduce the definition of \emph{$G$--Morse functions}. Our definition of $G$--Morse functions is equivalent to the definition of \emph{Morse functions} in \cite{wasserman1969equivariant}. 
In the following definition, $N$ denotes a smooth manifold possibly with boundary, and it is endowed with a smooth $G$--action and a smooth $G$--invariant Riemannian metric. The manifold $N$ is allowed to be non-compact.

\begin{Definition} 
	Let $f$ be a $G$--invariant $C^2$ function on $N$. We say that $f$ is \emph{$G$--Morse}, if 
	\begin{enumerate}
		\item $\nabla f\neq 0$ everywhere on $\partial N$, 
		\item $\ker  \Hess_p f=T_p\Orb(p)$ for all critical points $p$ of $f$.
	\end{enumerate} 
	
\end{Definition}

\begin{remark}
Suppose $f$ is a $G$--Morse function on $N$, then the critical orbits are discrete, in the sense that their image in the quotient space $N/G$ forms a discrete set. If $N$ is compact, then $f$ has only finitely many critical orbits.
\end{remark}
 
\begin{remark}
Let $C^\infty_G(M)$ be the space of $G$--invariant $C^\infty$ functions on $M$, and recall that $M$ is compact.
If $\partial M = \emptyset$, by \cite[Lemma 4.8]{wasserman1969equivariant}, the set of smooth $G$--Morse functions are dense in $C^\infty_G(M)$, therefore $G$--Morse functions exist. If $\partial M\neq \emptyset$, one can construct a $G$--Morse function on $M$ by taking a $G$--Morse function $f$ on the double of $M$ such that $\nabla f \neq 0$ on $\partial M$. In conclusion, $G$--Morse functions always exist on $M$.
\end{remark}

\begin{remark}
\label{rmk_G_Morse_open}
	Let $C^2_G(M)$ be the Banach space of $G$--invariant $C^2$ functions on $M$. Then the set of $G$--Morse functions is open in $C^2_G(M)$.
\end{remark}

	

\begin{remark}
	\label{rem_local_slice_Morse}
	Suppose $p\in M-\partial M$, let $D$ be a slice of $p$, let $U_p(D)$ be the neighborhood of $p$ defined by Definition  \ref{def_U_p(D)}. Then the set of $G$--Morse functions on $U_p(D)$ is in one-to-one correspondence with the set of $\Stab(p)$--Morse functions on $D$ via restrictions to $D$.
\end{remark}

For $\sigma\in \clR_G$, recall that $\nu(M_\sigma)$ denotes the orthogonal complement of $TM_\sigma$ in $TM|_{M_\sigma}$.
	Suppose $f$ is a $G$--invariant function on $M$, and suppose $p\in M_\sigma$. Since $\Hess_p f$ restricts to a $\Stab(p)$--equivariant map on $S_p$, by Schur's lemma and Part (2) of Lemma \ref{lem_M_sigma_basic_properties}, we have 
	$$
	(\Hess_pf)\, \nu(M_\sigma)|_p\subset \nu(M_\sigma)|_p.
	$$
	
\begin{Definition}
\label{def_Hess_sigma_f}
	Suppose $M_\sigma\neq \emptyset$, let $f$ be a $G$--invariant $C^2$ function on $M$. Define 
	$$\Hess_\sigma f: \nu(M_\sigma) \to \nu(M_\sigma)$$ to be the bundle map  given by the Hessian of $f$. 
\end{Definition}
	
By definition, $\Hess_\sigma f$ is self-adjoint.

\begin{Lemma}
\label{lem_characterize_Morse_on_M_sigma}
	A $G$--invariant function $f$ is $G$--Morse if and only if the following holds: 
	\begin{enumerate}
		\item $\nabla f\neq 0 $ on $\partial M$,
		\item for every $\sigma\in \clR_G$ such that $M_\sigma\neq \emptyset$, the function $f|_{M_\sigma}$ reduces to a (classical) Morse function on $M_\sigma/G$.
		\item for every $\sigma\in \clR_G$ such that $M_\sigma\neq \emptyset$, $\Hess_\sigma f$ is non-degenerate at all critical points of $f|_{M_\sigma}$. 
	\end{enumerate}
\end{Lemma}

\begin{proof}
	Let $f$ be a $G$--invariant function, and take $p\in M_\sigma$.
	 By Lemma \ref{lem_local_slice}, $T_pM$ decomposes as a $\Stab(p)$--representation into
	$$
		T_pM=T_p\Orb(p)\oplus T_p M_\sigma/T_p\Orb(p) \oplus T_pM/T_p M_\sigma,
	$$
	where $T_p M_\sigma/T_p\Orb(p)$ is the trivial representation, and $T_pM/T_p M_\sigma$ does not have trivial components. Since $\nabla_p f$ is $\Stab(p)$--invariant, it must be tangent to $T_pM_\sigma$, therefore $p$ is a critical point of $f$ if and only it is a critical point of $f|_{M_\sigma}$. Besides, $\Hess_pf$ is injective on $T_pM/T_p\Orb(p)$ if and only if it is injective on both $T_pM/T_p M_\sigma$ and $T_p M_\sigma/T_p\Orb(p)$. Therefore the lemma is proved.
\end{proof}

The next two lemmas will be used in the proofs of Lemma \ref{lem_one_degeneracy_codim} and Lemma \ref{lem_two_degeneracy_codim} when we establish the equivariant transversality properties. The proofs of these lemmas are essentially contained in \cite{wasserman1969equivariant}.

\begin{Lemma}
\label{lem_surj_jet_M_sigma}
	Suppose $M_\sigma\neq \emptyset$, and suppose $F\subset M$ is a $G$--invariant compact subset disjoint from $M_\sigma$. Let $f_\sigma$ be a $G$--invariant smooth function on $M_\sigma$, and let 
	$$H_\sigma: \nu(M_\sigma)\to \nu(M_\sigma)$$
	 be a $G$--equivariant, smooth, self-adjoint bundle map. Suppose  $f_\sigma$ and $H_\sigma$ are compactly supported. Then there exists a $G$--invariant smooth function $f$ on $M$, such that 
	 \begin{enumerate}
	 	\item 	 $f|_{M_\sigma}=f_\sigma$, $\Hess_\sigma f=H_\sigma$,
	 	\item $f=0$ on a neighborhood of $F$.
	 \end{enumerate}

\end{Lemma}

\begin{proof}
Let $A,B$ be $G$--invariant open subsets of $M_\sigma$, such that
$$ \supp f_\sigma\cup\supp H_\sigma\subset A\subset \overline A\subset B \subset \overline B\subset M_\sigma.$$

	Let $\exp:\nu(M_\sigma)|_{B}\to M$ be the exponential map, then $\exp$ is a diffeomorphism near the zero section. Let $U\subset \nu(M_\sigma)|_{B}$ be a $G$--invariant open neighborhood of the zero section such that $\exp$ is a diffeomorphism on $U$. We may choose $U$ sufficiently small such that $U$ is disjoint from $F$. Let $\chi$ be a $G$--invariant smooth cut-off function that are supported in $U$ and is equal to $1$ on $A$. 
	
	Let $\pi:\nu(M_\sigma) \to M_\sigma$ be the projection map, let 
	$$\tilde{f}(x):= f_\sigma\Big(\pi\big((\exp|_{U})^{-1} (x)\big)\Big)$$ 
	be the pull-back of $f_\sigma$ to $U$. Endow $U$ with the pull-back metric, define a function $\hat f$ on $U$ by
	$$\hat f(x) : = \tilde{f}(x) - \langle (\Hess_\sigma \tilde{f})\, x,x\rangle + \langle H_\sigma x, x\rangle.$$
	Extending the function $(\chi\cdot \hat f)\circ (\exp|_{U})^{-1}$ to $M$ by zero yields the desired $f$.
\end{proof}

\begin{Lemma}
\label{lem_surj_jet_M_sigma_2pts}
	Let $\sigma_1,\sigma_2\in \clR_G$. For $i=1,2$, let $A_i\subset M_{\sigma_i}$ be $G$--invariant open sets such that $\overline{A_i}\subset M_{\sigma_i}$. If $\sigma_1=\sigma_2$, we further assume that $\overline{A_1}\cap \overline{A_2}=\emptyset$. Let $f_i$ be a $G$--invariant smooth function on $M_{\sigma_i}$, and let 
	$$H_i: \nu(M_{\sigma_i})\to \nu(M_{\sigma_i})$$
	 be a $G$--equivariant smooth self-adjoint bundle map. Suppose  $\supp f_i\cup \supp H_i\subset A_i$. Then there exists a $G$--invariant smooth function $f$ on $M$, such that $f|_{A_i}=f_i$, and $\Hess_\sigma f|_{A_i}=H_i$ for $i=1,2$.
\end{Lemma} 

\begin{proof}
	If $\sigma_1=\sigma_2$, then the statement follows from Lemma \ref{lem_surj_jet_M_sigma}. If $\sigma_1\neq \sigma_2$, then by Lemma \ref{lem_surj_jet_M_sigma}, there exist smooth $G$--invariant functions $f^{(1)}$ and $f^{(2)}$ on $M$, such that
	\begin{enumerate}
		\item $f^{(1)}|_{A_1}=f_1$, $\Hess_\sigma f^{(1)}|_{A_1}=H_1$,
		\item $f^{(2)}|_{A_2}=f_1$, $\Hess_\sigma f^{(2)}|_{A_2}=H_2$,
		\item $f^{(1)}=0$ on a neighborhood of $\overline{A_2}$,
		\item $f^{(2)}=0$ on a neighborhood of $\overline{A_1}$.
	\end{enumerate} 
	Therefore, the function $f=f^{(1)}+f^{(2)}$ satisfies the desired conditions.
\end{proof}

\subsection{Equivariant indices of $G$--Morse functions}

We introduce the notion of equivariant index, which will play a crucial role in the equivariant Cerf theory in Section \ref{sec_equivariant_Cerf}.

\begin{Definition}
\label{def_equi_index_finite_dim}
	Suppose $f$ is a $G$--Morse function on $M$, suppose $O$ is a critical orbit of $f$.  Let $p\in O$, let $\Hess^-_pf$ be the subspace of $T_pM$ spanned by the negative eigenvectors of $\Hess_pf$, then $\Hess^-_pf$ is $\Stab(p)$--invariant. Define the \emph{equivariant index} of $O=\Orb(p)$ to be the element in $\clR_G$ represented by $\Hess^-_pf$ as a $\Stab(p)$--representation. 
\end{Definition}

It is straightforward to verify that 
the equivariant index does not depend on the choice of the point $p\in O$.

\begin{Definition}
\label{def_n_f}
Recall that $\bZ\clR_G$ denotes the free abelian group generated by $\clR_G$. 
	Suppose $f$ is a $G$--Morse function on $M$. Define the \emph{total index} of $f$ to be the element 
	$$
	\sum_{\sigma\in\clR_G} n_\sigma(f) \cdot \sigma \in \bZ\clR_G,
	$$
	 where $n_\sigma(f)$ denotes the number of critical orbits of $f$ with equivariant index $\sigma$.
\end{Definition}

We will use $\ind f$ to denote the total index of a $G$--Morse function $f$.

\begin{Lemma}
\label{lem_n_f_invariant_under_Morse_deformation}
	Suppose $f_t:M\to \bR$  is a smooth 1-parameter family of $G$--Morse functions on $M$, then $\ind f_t$ is constant with respect to $t$.
\end{Lemma}

\begin{proof}
	If $G$ is trivial, then the result is a standard property of (classical) Morse functions. In the equivariant case, by Lemma \ref{lem_characterize_Morse_on_M_sigma}, for every $\sigma \in \clR_G$, the restriction of $f_t$ to $M_\sigma$ reduces to a (classical) Morse function $f_{t,\sigma}$ on $M_\sigma/G$, and the critical orbits of $f_t$ on $M_\sigma$ is in bijection with critical points of $f_{t,\sigma}$ on $M_\sigma/G$ via the quotient map. By classical theory, there exist smooth, disjoint,  1-parameter families of points $p_i(t)\in M_\sigma/G$ such that the critical points of $f_t$ on $M_\sigma/G$ are given by $\{p_i(t)\}$. Let $\bar{p}_i(t)$ be a lift of $p_i(t)$ in $M_\sigma$; we may choose $\bar{p}_i(t)$ so that is varies smoothly with respect to $t$.
	
	At each $p = \bar{p}_i(t)$, we have the decomposition 
		$$
	T_pM=T_p\Orb(p)\oplus T_p M_\sigma/T_p\Orb(p) \oplus T_pM/T_p M_\sigma,
	$$
	and $\Hess^-_p f_t$ decomposes as a $\Stab(p)$--representation into the direct sum of a subspace of $T_p M_\sigma/T_p\Orb(p)$ and a subspace of $T_pM/T_p M_\sigma$, which we denote by $V_1(i,t)$ and $V_2(i,t)$ respectively. By Lemma \ref{lem_M_sigma_basic_properties}, $V_1(i,t)$ is a trivial representation of $\Stab(p)$ and $V_2(i,t)$ has no trivial components. Moreover, the dimension of $V_1(i,t)$ is equal to the (classical) Morse index of $p_i(t)$ as a critical point of $f_{t,\sigma}$ on $M_\sigma/G$. Therefore, the dimension of $V_1$ is independent of $t$. By Lemma \ref{lem_characterize_Morse_on_M_sigma} Part (3), the isomorphism class of $V_2(i,t)$ in $\clR_G$ is independent of $t$. Therefore, the equivariant index of the critical orbit of $f$ at $[\bar{p}_i(t)]$ is independent of $t$. This proves the desired result.
\end{proof}

\section{Equivariant Cerf theory}
\label{sec_equivariant_Cerf}

This section establishes the equivariant Cerf theory. 
All representations will be finite-dimensional in this section unless otherwise specified.

\subsection{Local bifurcation models}
\label{subsec_local_bif_model_finite_dim}
If $G$ is the trivial group, then by Cerf's theorem, the critical points of two different $G$--Morse functions are related to each other by a sequence of isotopies and birth-death transitions. 

 When $G$ is non-trivial, it is possible to have more complicated bifurcations. 
This subsection constructs several examples of bifurcations that will serve as local models later in Section \ref{subsec_equivariant_cerf}.

The following is a simple example of a bifurcation that is different from birth-death transitions.
 \begin{example}
 \label{example_2_element_group_bif}
 	Let $G=\bZ \slash 2$, consider the action of $G$  on $M=[-1,1]$ such that the generator of $G$ acts by $x \mapsto -x$. Define a 1-parameter family of functions on $M$ by 
 	$$f_t(x)=tx^2-x^4,\qquad t\in [-1,1].$$
 	The function $f_t$ is $G$--Morse when $t\neq 0$. When $t<0$, the function $f_t$ has one critical orbit with stabilizer $\bZ/2$; when $t>0$, the function $f_t$ has two critical orbits, one with stabilizer $\bZ/2$ and the other with stabilizer $\{1\}$.
 \end{example}

A schematic diagram for the critical orbits of $f_t$ in Example \ref{example_2_element_group_bif} is shown in Figure \ref{fig:pitchfork}.

\begin{figure}
	\begin{overpic}[width=0.4\textwidth]{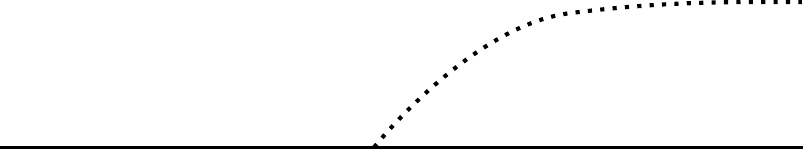}
		\put(-5,-10){$t=-1$}
		\put(42,-10){$t=0$}
		\put(95,-10){$t=1$}
		\put(102,0){$x=0$}
		\put(102,15){$x = \pm \sqrt{t/2}$}
	\end{overpic}
\vspace{\baselineskip}
	\caption{Critical orbits of $f_t$}\label{fig:pitchfork}
\end{figure}

The next example generalizes Example \ref{example_2_element_group_bif} to higher dimensional representations. Suppose $H$ is a compact Lie group that acts orthogonally on a finite-dimensional Euclidean space $V$, we consider the function 
$$f_t(x)=t\|x\|^2-\|x\|^4$$ 
on $V$. When $t<0$, the function $f_t(\cdot)$ is an $H$--Morse function with one critical orbit at $x=0$. When $t>0$, all points on the sphere $\|x\| = \sqrt{t/2}$ are critical points. Since critical orbits of $H$--Morse functions must be discrete, the function $f_t(\cdot)$ $(t>0)$ is not $H$--Morse if the $H$--action is not transitive on the sphere. In the next example, we further perturb the function $f_t$ near the critical sphere to obtain an $H$--Morse function when $t>0$.

\begin{example}
\label{example_bifurcation_model_finite_dim_irred}
Let $H$ be an arbitrary compact Lie group, let $V$ be a finite-dimensional Euclidean space, and let $\rho_V:H\to\Hom(V,V)$ be an orthogonal representation of $H$. We construct a 1-parameter family of $H$--invariant functions as follows.  
Let $g$ be a non-negative smooth $H$--Morse function on the unit sphere of $V$.
Let $\chi:\bR\to\bR$ be a smooth non-negative function supported in $[1/2,2]$ such that $\chi(1)=1$, $\chi^{\prime}(1) = 0$, $\chi'(x)\ge 0$ when $x\le 1$, and $\chi'(x)\le 0$ when $x\ge 1$. Let $B_V(r)$ be the closed ball in $V$ centered at $0$ with radius $r$. Define
$$
F_V:[-1,1]\times B_V(2) \to \bR
$$
by
\begin{equation}
\label{eqn_bifurcation_model_finite_dim_irred}
	F_V(t,x) =
\begin{cases}
t\|x\|^2-\|x\|^4 & \mbox{if } t\le 0,\\
t\|x\|^2 - \|x\|^4 + e^{-1/t}\cdot \chi\Big(\|x\|/\sqrt{\frac{t}{2}}\Big)\cdot g(x/\|x\|) & \mbox{if } t> 0 \mbox{ and } x \neq 0,\\
0 & \mbox{if } t>0 \mbox{ and } x = 0.
\end{cases}
\end{equation}
We show that the function $F_V(t,\cdot)$ is $H$--Morse when $t\neq 0$.

When $t<0$, we have 
$$\grad F_V(t,\cdot) (x) = (2t - 4 \|x\|^2){x},$$ so the function $F_V(t,\cdot)$ has exactly one critical orbit at $0$ and it is $H$--Morse.

When $t>0$ and $\|x\|<\frac12\sqrt{t/2}$, we have $F_V(t,x) = t\|x\|^2-\|x\|^4$, and it is straightforward to verify that the only critical orbit of $F_V(t,\cdot)$ with $\|x\|<\frac12\sqrt{t/2}$ is at $x=0$, and the critical orbit is non-degenerate. 

When $t>0$ and $\|x\|\ge \frac12\sqrt{t/2}$, we have 
\begin{multline*}
\grad F_V(t,\cdot) (x) = (2t - 4 \|x\|^2){x} + e^{-1/t}\sqrt{\frac{2}{t}}\cdot \frac{1}{\|x\|}\cdot \chi'\Big(\|x\|/\sqrt{\frac{t}{2}}\Big)\cdot g(x/\|x\|)\cdot x
\\
 + e^{-1/t}\cdot \chi\Big(\|x\|/\sqrt{\frac{t}{2}}\Big)\cdot\frac{1}{\|x\|}\cdot (\grad g)(x/\|x\|).
\end{multline*}
Since $(\grad g)(x/\|x\|)$ is tangent to the unit sphere, it is orthogonal to $x$, therefore the critical points of $F_V(t,\cdot)$ are given by the equations
\begin{align}
(2t - 4 \|x\|^2)+ e^{-1/t}\sqrt{\frac{2}{t}}\cdot \frac{1}{\|x\|}\cdot \chi'\Big(\|x\|/\sqrt{\frac{t}{2}}\Big)\cdot g(x/\|x\|) & = 0
\label{eqn_critical_points_FV_1}\\
(\grad g)(x/\|x\|) & = 0 
\end{align}
It is straightforward to verify that the left-hand side of \eqref{eqn_critical_points_FV_1} is positive when $\|x\| < \sqrt{t/2}$, zero when $\|x\|= \sqrt{t/2}$, and negative when $\|x\|>\sqrt{t/2}$. Therefore, all critical points are on the sphere $\partial B_V\Big(\sqrt{\frac{t}{2}}\Big)$, and $x\in \partial B_V\Big(\sqrt{\frac{t}{2}}\Big)$ is a critical point of $F_V(t,-)$ if and only if $x/\|x\|$ is a critical point of $g$. Since $g$ is $H$--Morse on the unit sphere, it is straightforward to verify that all critical orbits on $\partial B_V\Big(\sqrt{\frac{t}{2}}\Big)$ are non-degenerate.

In conclusion, when $t>0$, 
the function $F_V(t,\cdot)$ has one critical orbit at $0$ and a finite set of critical orbits on $\partial B_V\Big(\sqrt{\frac{t}{2}}\Big)$ that are in bijection with the critical orbits of $g$.
\end{example}

Example \ref{example_bifurcation_model_finite_dim_irred} can be further generalized as follows.

\begin{example}
\label{example_bifurcation_model_finite_dim_sum}
	Let $H,V,g,F_V$ be as in Example \ref{example_bifurcation_model_finite_dim_irred}.
	Let $V'$ be another orthogonal representation of $H$. Let $B_{V\oplus V'}(r)$ be the closed ball in $V\oplus V'$ centered at 0 with radius $r$, and define $B_V(r)$ and $B_{V'}(r)$ similarly. We define a 1-parameter family of $H$--invariant functions on $B_{V\oplus V'}(2)$ as follows. Let $h:V'\to \bR$ be an $H$--Morse function on $B_{V'}(2)$ such that $0$ is the unique critical point of $h$. Define
\begin{align}
F_{V\oplus V'} :[-1,1]\times B_{V\oplus V'}(2) & \to \bR 
\nonumber
\\
(t,(x,y)) & \mapsto F_V(t,x) + h(y),
\label{eqn_F_V_V'_local_model}
\end{align}
where $x\in V$, $y\in V'$. The function $F_{V\oplus V'}(t,\cdot)$ is $H$--Morse for $t\neq 0$, and $(x,y)$ is a critical point of  $F_{V\oplus V'}(t,\cdot)$ if and only if $y=0$ and $x$ is a critical point of $F_V(t,\cdot)$. 
\end{example}

\begin{remark}
	In the above example, the point $0$ is not necessarily a local minimum or a local maximum of $h$. In fact, $h$ can have an arbitrary equivariant index at $0$. This flexibility will be important for the later discussions.
\end{remark}

We now extend Example \ref{example_bifurcation_model_finite_dim_sum} to general $G$--manifolds using slices.

\begin{example}
\label{example_bifurcation_model_finite_dim_manifold}
Suppose $M$ is a $G$--manifold and $p\in M$. Let $H= \Stab(p)$, let $D$ be a slice of $p$, and recall that $U_p(D)$ is the closed $G$--invariant tubular neighborhood of $\Orb(p)$ defined by Definition \ref{def_U_p(D)}. By definition, $U_p(D)$ is equivariantly diffeomorphic to $G\times_H D$. Recall that by Remark \ref{rem_local_G-invariant_function_slice}, $G$--invariant functions on $U_p(D)$ are in one-to-one correspondence with $H$--invariant functions on $D$ via the restriction to $D$. By Remark \ref{rem_local_slice_Morse}, $G$--Morse functions on $U_p(D)$ correspond to $H$--Morse functions on $D$. 
Since $D$ is $H$--equivariantly diffeomorphic to the unit disk in an orthogonal $H$--representation,
Example \ref{example_bifurcation_model_finite_dim_sum} gives rise to 1-parameter families of $H$--invariant functions on $D$, and hence it defines 1-parameter families of $G$--invariant functions on $U_p(D)$. 
\end{example}

\begin{Definition}
\label{def_irred_bifurcation}
Suppose $f_t$ is a smooth 1-parameter family of $G$--invariant functions on $M$. We say that $f_t$ has an \emph{irreducible bifurcation} at $t=t_0$, if there exists $p\in M-\partial M$ such that the following holds:
\begin{enumerate}
	\item There exists $\epsilon>0$, such that $f_t$ is $G$--Morse for 
	$$t\in [t_0-\epsilon,t_0)\cup(t_0,t_0+\epsilon].$$
	\item 
	The function $f_{t_0}$ has exactly one degenerate critical orbit at $\Orb(p)$, and $\nabla f_{t_0}\neq 0$ on $\partial M$.
	\item  $S_p$ decomposes as $V\oplus V'$ as an orthogonal $\Stab(p)$--representation, where $V$ is non-trivial and irreducible.
	\item There exists a slice $D$ at $p$, a non-negative $\Stab(p)$--Morse function $g$ on the unit sphere of $V$, and a smooth $\Stab(p)$--Morse function $h$ on $B_{V'}(2)$ with $0$ being the only critical point of $h$, such that $f_{t-t_0}$ or $f_{t_0-t}$ is locally given by the family $F_{V\oplus V'}$ defined by \eqref{eqn_F_V_V'_local_model} on $D$. 
\end{enumerate}
\end{Definition}

We now compute the change of total indices under an irreducible bifurcation. 

Suppose $f_t$ is a 1-parameter family of $G$--invariant functions on $M$, such that 
\begin{enumerate}
	\item $f_t$ is $G$--Morse for $t\neq 0$,
	\item $f_t$ has an irreducible bifurcation at $t=0$ on the orbit $\Orb(p)$. 
\end{enumerate}
Without loss of generality, assume there is a slice $D$ of $p$ such that $f_t$ (instead of $f_{-t}$) is locally given by $F_{V\oplus V'}$ on $D$.

 Let $H=\Stab(p)$, let $S_p\cong V\oplus V'$ be the  decomposition of $S_p$ in  Part (2) of Definition \ref{def_irred_bifurcation}, and let $h$ and $g$ be as in Part (4) of Definition \ref{def_irred_bifurcation}. Let $$\rho_V:H\to\Hom(V,V)$$ be the representation of $H$ on $V$. 
 
Let $\Hess^- h\subset V'$ be the subspace spanned by the negative eigenvectors of the Hessian of $h$ at the origin. Let
  $$\rho_h:H\to\Hom(\Hess^-h,\Hess^-h)$$ 
 be the representation of $H$ on $\Hess^-h$. Suppose the total index of $g$ as an $H$--Morse function is given by 
 $$
 \ind g\in  \bZ\clR_H.
 $$

\begin{Lemma}
	\label{lem_change_of_index_under_irred_bif}
	For $t=-\epsilon$ with $\epsilon>0$ and sufficiently small, the total index of $f_t$ on $U_p(D)$ is given by 
	\begin{equation}
		\label{eqn_total_index_t=-epsilon}
		[H,\Hess^-h\oplus V,\rho_h\oplus \rho_V]\in \bZ\clR_G.
	\end{equation}
For $t=\epsilon$ with $\epsilon>0$ sufficiently small, the total index of $f_t$ on $U_p(D)$ is given by 
	\begin{equation}
		\label{eqn_total_index_t=+epsilon}
		[H, \Hess^-h,\rho_h] 
		+ i^H_G ( \ind g\oplus \bR\oplus [H, \Hess^-h,\rho_h]) \in \bZ\clR_G,
	\end{equation}
	where $\bR\in\clR_H([H])$ is given by the trivial representation of $H$ on $\bR$, the direct sum operator is taken in $\clR_H$ and is defined by \eqref{eqn_dir_sum_clR}, and the map $i^H_G$ is defined by Definition \ref{def_i^H_G}.

\end{Lemma}

\begin{proof}
	We only need to compute the total indices of $f_t$ as $H$--invariant functions on $D$, and then apply the map $i^H_G$ from Definition \ref{def_i^H_G}. By definition, the restriction of $f_t$ to $D$ is equal to $F_{V\oplus V'}(t,\cdot)$ from Example \ref{example_bifurcation_model_finite_dim_sum}. 
	
	When $t<0$, there is exactly one critical orbit at the origin, and its equivariant index is $[H,\Hess^-h\oplus V,\rho_h\oplus \rho_V]$. 
	
	When $t>0$, the critical orbit at $0$ has equivariant index $[H,\Hess^-h,\rho_h]$. Assume $x$ is a critical point of $F_{V\oplus V'}(t,\cdot)$ with $\|x\| = \sqrt{t/2}$. Then $x/\|x\|$ is a critical point of $g$, and $\Hess_x^- f_t$ is isomorphic (as an $H$--representation) to the direct sum of $\Hess^-_{x/\|x\|} g$, $\Hess^-h$, and a trivial 1-dimensional representation. The $1$--dimensional representation comes from the normal bundle of $\partial B(\sqrt{t/2})$ in $V$. Therefore the desired results are proved.
\end{proof}

Notice that the total indices given by \eqref{eqn_total_index_t=-epsilon} and \eqref{eqn_total_index_t=+epsilon} only depend on $H$,  $V$, $\rho_V$, $\Hess^-h$, $\rho_h$, and the $H$--Morse function $g$.
Therefore we make the following definition.
\begin{Definition}
\label{def_xi_+-}
 Suppose $H$ is a closed subgroup of $G$, and $V$ is a non-trivial irreducible orthogonal representation of $H$. Let $g$ be a non-negative $H$--Morse function on the unit sphere of $V$, and let $V'$ be a real representation of $H$. Suppose the total index of $g$ is given by $\ind g\in \clR_H$, and let $\rho_V$, $\rho_{V'}$ be the actions of $H$ on $V$, $V'$ respectively.
 Define 
 \begin{equation*}
  \xi_H^-(V,V',g):= [H,V\oplus V', \rho_V \oplus \rho_{V'}]\in \bZ\clR_H,
 \end{equation*}
 and 
$$
\xi_H^+(V,V',g):= [H, V',\rho_{V'}] + \ind g\oplus \bR \oplus [H, V',\rho_{V'}]
\in \bZ\clR_H,
$$
where $\bR$ denotes the trivial representation of $H$ on $\bR$. Define
$$
\xi_H(V,V',g):= \xi_H^-(V,V',g) - \xi_H^+(V,V',g).
$$
\end{Definition}

By \eqref{eqn_total_index_t=-epsilon} and \eqref{eqn_total_index_t=+epsilon}, an irreducible bifurcation changes the total index by
$$\pm i^H_G\big(\xi_H(V,V',g)\big),$$
 where $H$, $V$, $g$ are as in Definition \ref{def_irred_bifurcation}, and $V'$ is given by $(\Hess^-h,\rho_h)$.
 Also, by the definitions, we have
\begin{align}
\xi^{\pm}_H(V,V',g) &= \xi^{\pm}_H(V,0,g) \oplus [V'],
\label{eqn_dir_sum_xi_pm}
\\
\xi_H(V,V',g) &= \xi_H(V,0,g) \oplus [V'],
\label{eqn_dir_sum_xi}
\end{align}
where $0$ denotes the zero representation of $H$.

We now define the birth-death bifurcations with the presence of $G$ action.

\begin{Definition}
\label{def_birth_death_bif}
	Suppose $f_t$ is a smooth 1-parameter family of $G$--invariant functions on $M$. We say that $f_t$ has a \emph{birth-death bifurcation} at $t=t_0$, if there exists $p\in M$, such that following holds:
	\begin{enumerate}
	\item There exists $\epsilon>0$, such that $f_t$ is $G$--Morse for 
	$$t\in [t_0-\epsilon,t_0)\cup(t_0,t_0+\epsilon].$$
	\item  
	The function $f_{t_0}$ has exactly one degenerate critical orbit at $\Orb(p)$, and $\nabla f_{t_0}\neq 0$ on $\partial M$.
	\item  Suppose $p\in M_\sigma$, then $\Hess_\sigma f_{t_0}$ (from Definition \ref{def_Hess_sigma_f}) is non-degenerate at $p$.
	\item  The 1-parameter family of functions on $M_\sigma/G$ induced by $f_{t}|_{M_\sigma}$ has a birth-death singularity (in the classical sense) at $(t_0,[p])$, where $[p]$ is the image of $p$ in $M_\sigma/G$.  
\end{enumerate}
\end{Definition}

\begin{Lemma}
A birth-death bifurcation changes the total index by the addition of an element of the form $\pm (\sigma  +\sigma \oplus \bR)$, where $\bR\in \clR_G([G])$ is the element represented by the trivial representation of $G$ on $\bR$.
\end{Lemma}

\begin{proof}
Suppose $\tau\in \clR_G$, and $[q]\subset M_\tau$ is a non-degenerate critical orbit of $f_t$. Let $q$ be a point on the orbit. Recall that we have the decomposition
$$
T_qM=T_q\Orb(q)\oplus T_q M_\tau/T_q\Orb(q) \oplus T_qM/T_q M_\tau.
$$
So the equivariant index at $q$, which is represented by a $\Stab(q)$ representation, is decomposed as the direct sum of 
a component from $T_q M_\tau/T_q\Orb(q)$ and a component from $T_qM/T_q M_\tau$. We will call the first component the ``tangent'' component, and the second component the ``normal'' component. By Lemma \ref{lem_M_sigma_basic_properties}, the tangent component is given by a trivial representation, and the normal component is given by a representation without trivial components. Moreover, the dimension of the tangent component is equal to the (classical) Morse index of the induced function on $M_\tau/G$.

Since $\Hess_\tau f_{t_0}$ is non-degenerate at $p$, the two critical orbits created by a ``birth'' bifurcation or canceled out by a ``death" bifurcation have the same normal component in the equivariant indices. By the definition of the (classical) birth-death singularity, their  (classical) Morse indices for the induced functions on $M_\tau/G$ differ by $1$. Hence the equivariant indices of the two critical orbits have the form $\sigma$ and $\sigma\oplus \bR$.
\end{proof}

\begin{Definition}
	Let $\Bif_G\subset \bZ\clR_G$ be the subgroup generated by 
	$$
	i^H_G\big(\xi_H(V,V',g)\big) \quad\mbox{and } \quad \sigma  +\sigma \oplus \bR
	$$
for all possible choices of $H$, $V$, $V'$, $g$, and
for all $\sigma\in \clR_G$.
\end{Definition}

The following lemma follows immediately from the definitions and \eqref{eqn_dir_sum_xi}.
\begin{Lemma}
\label{lem_direct_sum_Bif_unfiltered}
Suppose $\sigma \in \clR_G([G])$, then 
$\Bif_G\oplus \, \sigma \subset \Bif_G.$ \qed
\end{Lemma}

\subsection{Statement of the equivariant Cerf theorem}
\label{subsec_equivariant_cerf}
The main result of this section is the following theorem, which states that irreducible bifurcations and birth-death bifurcations generate all the possible changes on the total index.

\begin{Theorem}
\label{thm_equivairant_cerf_finite_dim}
	Suppose $f_0$ and $f_1$ are $G$--Morse functions on $M$, and suppose there exists a smooth 1-parameter family $f_t$, $t\in[0,1]$ connecting $f_0$ and $f_1$ such that 
	$$\nabla f_t\neq 0 \mbox{ on }\partial M$$
	 for all $t$. Then $\ind f_0-\ind f_1\in \Bif_G$. 
\end{Theorem}

\begin{remark}
	\label{rem_path_not_generic_in_equivariant_Cerf}
	Unlike the classical Cerf theorem, we do not claim that a \emph{generic} 1-parameter family only contains irreducible and birth-death bifurcations. In fact, this claim is not true: let $G=\{u\in \bC| u^3=1\}$, and consider the action of $G$ on $\bC$ by multiplications. We will call this action the ``standard action'' of $G$ on $\bC\cong \bR^2$. Consider the 1-parameter family of $G$--invariant functions on $\bC$ given by 
	$$
		f_t(z) := t|z|^2 + \mbox{Re} (z^3).
	$$
	When $t\neq 0$, the critical orbits of $f_t$ are given by $(0,0)$ and the orbit of $(-2t/3, 0)$. The equivariant index at $(0,0)$ is represented by $\bC\cong \bR^2$ with the standard $G$ action when $t<0$, and is represented by the zero representation when $t>0$. At the orbit of $(-2t/3,0)$, the index remains the same for all $t\neq 0$ and is given by a $1$-dimensional trivial representation of the trivial group.
	
	Therefore, the change of the total index is given by the sum of $-\xi_G(V,V',g)$ and $(\sigma+\sigma\oplus \bR)$, where $V$ is $\bC\cong \bR^2$ endowed with the standard $G$ action, $V'$ is the zero representation, $g$ is a $G$--Morse function on the unit circle of $\bC$ that has two critical orbits, and $\sigma$ is the zero-dimensional space viewed as a representation of the trivial group. This sum cannot be realized as a single irreducible or birth-death bifurcation.
	
	 By a straightforward computation, one can show that $\{f_t\}$ as a path in the space of $G$--invariant functions is transverse to $\mathcal{F}_\sigma$ defined in Definition \ref{defn_F_sigma} below, where $\sigma$ is given by the standard action of $G$ on $\bR^2$. Hence the behavior of bifurcation persists after small perturbations of $\{f_t\}$.
\end{remark}

The proof of Theorem \ref{thm_equivairant_cerf_finite_dim} consists of two main parts. In the first part, we show that in an appropriate sense, the set of non-$G$--Morse functions can be written as a countable union of submanifolds with positive codimensions in the space of $G$--invariant functions. The proof of this statement is based on an equivariant transversality argument. As a result, we may take a generic perturbation of the path $\{f_t\}$ so that it intersects the locus of non-$G$--Morse functions transversely.

The  transversality condition does not immediately prove the theorem because of two reasons. First, our transversality result will only show that on a generic path $\{f_t\}$, there are at most \emph{countably} many values of $t$ such that $f_t$ is non-$G$--Morse (see Remark \ref{rem_countable_but_not_finite} for more discussions).  Since a countable and closed subset of $\bR$ is not necessarily finite, it is not straightforward to directly use this result to study the change of total index with respect to $t$.  

Second, as is shown in Remark \ref{rem_path_not_generic_in_equivariant_Cerf}, even if we assume that there is only one value of $t$ such that $f_t$ is not $G$--Morse, and assume that the path $\{f_t\}$ is transverse to the locus of non-$G$--Morse functions, the change of total index is still not necessarily given by an irreducible or birth-death bifurcation. From a technical point of view, this is because the transversality assumption only depends on the properties of the second order derivatives of $f_t(\cdot)$. But for $G$--invariant functions, the bifurcation behavior can depend on the higher order derivatives.  This difficulty is resolved by the following idea: when there is only one degeneracy in the path $\{f_t\}$ ($t\in[0,1]$) that is transverse to the locus of non-$G$--Morse functions, we first construct a pair of functions $\hat f_0$ and $\hat f_0$, so that $\ind f_i - \ind \hat f_i\in \Bif_G$ for $i=0,1$. The functions $\hat f_0$ and $\hat f_1$ are constructed so that there exists a path between them that only intersects the  locus of non-$G$--Morse functions in its ``lower'' strata, where the term ``lower'' is characterized by the dimension and the number of connected components of the stabilizer group at the degeneracy (see Definition \ref{def_m(H)}).  We can then use an induction argument to control the difference of total indices of $\hat f_0$ and $\hat f_1$. The core of the induction argument will be given by Lemma \ref{lem_special_case_linear_no_trivial_component}.

The proof of Theorem \ref{thm_equivairant_cerf_finite_dim}  is organized as follows.  In Section \ref{sebsec_linear_alg} , we will prove several technical lemmas in linear algebra. These results will be used to study the linearizations of various maps between Banach manifolds.  We will then establish the equivariant transversality result in Section \ref{subsec_transversality_finite_dim}, and finish the proof of Theorem \ref{thm_equivairant_cerf_finite_dim} in Section \ref{subsec_proof_thm_equi_cerf}.

\subsection{Preliminaries on linear algebra}
\label{sebsec_linear_alg}
Notice that if $H$ is a compact Lie group and $V$ is an irreducible representation of $H$, then $\Hom_H(V,V)$ is an associative division algebra over $\bR$, therefore $\Hom_H(V,V)\cong \bR$, $\bC$ or $\bH$. 

\begin{Definition}
	Suppose $H$ is a compact Lie group and $V$ is an irreducible representation of $H$.
	We say that $V$ is of \emph{type} $\bR$, $\bC$ or $\bH$, if 
	$$\Hom_H(V,V)\cong \bR, \bC \mbox{ or }\bH$$ respectively.
\end{Definition}

Recall that if $V$ is an Euclidean space over $\bR$, a linear map $f: V\to V$ is called \emph{symmetric} if for all $v, w \in V$, we have $\langle f(v), w \rangle = \langle v, f(w)\rangle.$ If $\bK\in \{\bR,\bC,\bH\}$, then there is a standard Euclidean structure on $\bK^r$, and a $\bK$--linear map $f: \bK^r \to \bK^r$ is called \emph{self-dual} if it is symmetric when viewed as an $\bR$--linear map.

\begin{Definition}
Suppose $H$ is a compact Lie group and $V$ is an orthogonal representation of $H$, define $\sym_H(V)$ to be the subspace of $\Hom_H(V,V)$ consisting of symmetric maps.
\end{Definition}

\begin{Definition}
Suppose $\bK\in \{\bR,\bC,\bH\}$. Let $r$ be a non-negative integer. Define $d_{\bK}(r)$ to be the (real) dimension of self-dual $\bK$--linear maps on $\bK^r$. Then we have
$$
d_{\bK}(r)=
\begin{cases}
\frac12 r(r+1) & \text{if $\bK=\bR$,} \\
r^2 & \text{if $\bK=\bC$,} \\
2r^2 - r & \text{if $\bK=\bH$}.
\end{cases}
$$
\end{Definition}

The following statement in a straightforward consequence of Schur's lemma.

\begin{Lemma}
\label{lem_dim_of_sym_isotypic_component}
	Suppose $H$ is a compact Lie group, $V$ is an orthogonal representation of $H$, and suppose the isotypic decomposition of $V$ is given by 
	$$
	V\cong V_1^{\oplus a_1}\oplus\cdots\oplus V_m^{\oplus a_m},
	$$
	where $V_i$ is of type $\bK_i$ for $\bK_i\in\{\bR,\bC,\bH\}$. Then
	$$
	\dim_\bR \sym_H(V) = \sum_{i=1}^m d_{\bK_i}(a_i).
	\phantom\qedhere\makeatletter\displaymath@qed
	$$
\end{Lemma}

\begin{Definition}
\label{def_d(sigma)}
	Suppose $\sigma\in\clR_G$ is represented by an orthogonal representation $(H,V)$. Define
	$$
	d(\sigma) := \dim_\bR \sym_H(V).
	$$
\end{Definition}

By Lemma \ref{lem_dim_of_sym_isotypic_component}, $d(\sigma)$ does not depend on the choice of $(H,V)$ or the Euclidean structure of $V$.

\begin{Definition}
	Let $H$ be a compact Lie group. Suppose  $V$ is an orthogonal $H$--representation. For $\sigma\in \clR_G([H])$, 
	let $\sym_{H,\sigma}(V)\subset \sym_H(V)$ be the subspace consisting of $s\in \sym_H(V)$ such that $\ker s$ represents $\sigma$.
\end{Definition}

\begin{Lemma}
\label{lem_sym_sigma_submanifold_finite_dim}
$\sym_{H,\sigma}(V)$ is a submanifold of $\sym_H(V)$ with codimension $d(\sigma)$. Moreover, suppose $s\in \sym_{H,\sigma}(V)$, and suppose $L\subset \sym_H(V)$ is a linear subspace, let $\Pi:V\to \ker s$ be the orthogonal projection onto $\ker s$, then $s+L$ is transverse to $\sym_{H,\sigma}(V)$ if and only if the map
$$
\varphi: L\to \sym_H(\ker s) 
$$
defined by 
$$\varphi(l)(x) := \Pi \big(l(x)\big)$$
is a surjection.
\end{Lemma}

\begin{proof}
Suppose $s\in \sym_{H,\sigma}(V)$. Let $V_1=\ker s$, and let $V_2$ be the orthogonal complement of $V_1$, then $s$ restricts to an invertible self-adjoint map on $V_2$. Suppose $s'\in \sym_H(V)$ is close to $s$, then under the decomposition $V=V_1\oplus V_2$, the map $s'$ is decomposed as 
$$
s' = \begin{pmatrix}
S_{11} & S_{12} \\
S_{21} & S_{22}
\end{pmatrix},
$$
where $S_{ij}:V_i\to V_j$ is an $H$--equivariant map. For $s'$ sufficiently close to $s$, the map $S_{22}$ is invertible, and we have
\begin{align}
	&\begin{pmatrix}
		\id & -S_{12} \circ S_{22}^{-1} \\
	0  & \id
	\end{pmatrix}
	\cdot 
	\begin{pmatrix}
		S_{11} & S_{12} \\
		S_{21} & S_{22}
	\end{pmatrix}
	\cdot
	\begin{pmatrix}
		\id & 0 \\
		-S_{22}^{-1}\circ S_{21} & \id
	\end{pmatrix}
	\nonumber
	\\
	= &
	\begin{pmatrix}
		S_{11}-S_{12}\circ S_{22}^{-1}\circ S_{21} & 0 \\
		0 & S_{22}
	\end{pmatrix}.
	\label{eqn_elementary_matrix_trans_S22_finite_dim}
	\end{align}
Hence $s'\in \sym_{H,\sigma}(V)$ if and only if 
$$
S_{11}-S_{12}\circ S_{22}^{-1}\circ S_{21} = 0.
$$
As a result, $\sym_{H,\sigma}$ is a manifold near $s$, and its tangent space at $s$ is given by $S_{11}=0$. Therefore the lemma is proved.
\end{proof}

\subsection{Transversality}
\label{subsec_transversality_finite_dim}
Let $f_0$ and $f_1$ be as in Theorem  \ref{thm_equivairant_cerf_finite_dim}, this subsection studies the property of a generic path from $f_0$ to $f_1$.

Let $C_G^\infty(M)$ be the space of $G$--invariant $C^\infty$ functions $f$ on $M$, then $C_G^\infty(M)$ is a closed subspace of $C^\infty(M)$. Endow  $C_G^\infty(M)$ with the standard $C^\infty$--topology.

Let $f_t$ be as in Theorem \ref{thm_equivairant_cerf_finite_dim}.
Since $\nabla f_t\neq 0$ on $\partial M$, there exists $$0=t_1<t_2<\cdots<t_m=1$$ such that 
$$u\,f_{t_i}+(1-u)\,f_{t_{i+1}}$$ 
have non-vanishing gradients on $\partial M$ for all $i=1,\cdots,m-1$ and $u\in[0,1]$. 
Let $\{g_0,g_1,\cdots \}$ be a countable dense subset of $C_G^\infty(M)$ that contains $f_{t_i}$ for all $i=1,\cdots,m$. For $m\in \bZ^+$, let
$$N_m := \sup \{\|g_0\|_{C^i},\cdots,\|g_m\|_{C^m}\}.$$ Let $\clF$ be the Banach space  defined by 
$$\clF:=\{(a_0, a_1, \cdots)|\sum_{m\ge 0} N_m|a_m|<+\infty\}.$$
 Then the map 
\begin{align*}
\iota:\clF &\to C_G^\infty(M) \\
(a_0,a_1,\cdots)&\mapsto \sum_{m\ge 0} a_m g_m
\end{align*}
is a continuous linear map with dense image. 

Let 
$$\clF':=\{p\in\clF|\nabla\iota(p)\neq 0 \mbox{ on } \partial M\},$$ then $\clF'$ is an open subset of $\clF$. Take $\frp_0,\frp_1\in\clF$ such that $\iota(\frp_0)=f_0$, $\iota(\frp_1)=f_1$, then $\frp_0$ and $\frp_1$ are in the same connected component of $\clF'$ by the construction of $\clF$.

\begin{Definition}
\label{def_C_infty_subvariety}
	Suppose $\clM$ is a Banach manifold and $d$ is a non-negative integer. A subset $\clS\subset \clM$ is called a \emph{$C^\infty$--subvariety with codimension at least $d$}, if 
$\clS$ can be covered by the image of countably many smooth Fredholm maps with index $-d$.
\end{Definition}

\begin{Definition}
	\label{defn_F_sigma}
Suppose $\sigma\in \clR_G$.
	Let 
	\begin{multline*}
		 \clF_\sigma:=\{\fru\in\clF'\,|\,\exists p\in M \mbox{ such that } \\
		 \nabla_p\iota(\fru)=0, \mbox{ and } \ker \Hess_p \iota(\fru)/T_p\Orb p \mbox{ represents }  \sigma\}.
	\end{multline*}
\end{Definition}

\begin{Lemma}
\label{lem_one_degeneracy_codim}
	$\clF_\sigma$ is a $C^\infty$--subvariety of  $\clF'$ with codimension at least $d(\sigma)$.
\end{Lemma}

\begin{proof}
	Let 
	$$\widetilde \clF_\sigma := \{(\fru,p)\in \clF'\times M \, | \, \nabla \iota(\fru)(p) = 0,\, \ker \Hess_p \iota(\fru)/T_p\Orb p \mbox{ represents }  \sigma\}.$$
	Suppose $(\fru,p)\in \widetilde \clF_\sigma$, let $D$ be a slice of $p$, we claim that there exists an open neighborhood $U$ of $(\fru,p)$ in $\clF'\times D$, such that 
	\begin{enumerate}
		\item $\widetilde \clF_\sigma\cap U$ is a Banach manifold,
		\item The projection of $\widetilde \clF_\sigma\cap U$ to $\clF'$ is Fredholm and has index $-d(\sigma)$.
	\end{enumerate}
	The result then follows from the above claim and the separability of $\clF'\times M$.
	
	To prove the claim, notice that changing the $G$--invariant metric of $M$ does not change the set $\widetilde{\clF}_\sigma$, therefore we may assume without loss of generality that $D$ is $\Stab(p)$--equivariantly diffeomorphic to a closed disk of $S_p$ (see Definition \ref{defn_Sp}) via an isometry, therefore $TD$ is canonically trivialized by $S_p$ via parallel translations. Let $D^0$ be the fixed-point subset of $D$ with respect to the $\Stab(p)$--action, let $S_p^0$ be the fixed-point subset of $S_p$. Then
$$\widetilde \clF_\sigma\cap (\clF'\times D)=\widetilde \clF_\sigma\cap (\clF'\times D^0),$$
and $\widetilde \clF_\sigma\cap (\clF'\times D^0)$ is given by the pre-image of $\{0\}\times \sym_{\Stab(p),\sigma}(S_p)$ of the map
\begin{align*}
\varphi:\clF'\times D^0 &\to (S_p^0)^* \times \sym_{\Stab(p)}(S_p) \\
(\frv,q)&\mapsto (\nabla_q\iota(\frv)|_{D^0} , \Hess \iota(\frv)|_{S_p}).
\end{align*}
By Lemma \ref{lem_surj_jet_M_sigma} and the density of $\ima \iota$ in $C^\infty_G(M)$, the image of the tangent map of $\varphi$ at $(\fru,p)$ is dense.  Since the co-domain of the tangent map is a finite-dimensional linear space and the image is closed, we conclude that the tangent map of $\varphi$ at $(\fru,p)$ is surjective. Therefore by Lemma \ref{lem_sym_sigma_submanifold_finite_dim}, 
$$\varphi^{-1}\big(\{0\}\times \sym_{\Stab(p),\sigma}(S_p)\big)$$ is a Banach manifold near $(\fru,p)$. 

The embedding of $\varphi^{-1}\big(\{0\}\times \sym_{\Stab(p),\sigma}(S_p)\big)$ in $\clF'\times D^0$ is Fredholm with index $-\dim D^0 - d(\sigma)$, and the projection of $\clF'\times D^0$ to $\clF'$ is Fredholm with index $\dim D^0$.
Therefore the projection of $\varphi^{-1}\big(\{0\}\times \sym_{\Stab(p),\sigma}(S_p)\big)$ to  $\clF'$ is Fredholm and has index $-d(\sigma)$.
\end{proof}

Notice that $d(\sigma)=0$ if and only if $\sigma$ is given by a zero representation.
Therefore by Lemma \ref{lem_one_degeneracy_codim}, the function $\iota(\fru)$ is $G$--Morse for a generic $\fru\in\clF$. Hence we have recovered the existence theorem of $G$--Morse functions by Wasserman \cite{wasserman1969equivariant}. 

The following discussion shows that for a generic 1-parameter family $f_t$, there is at most one degeneracy orbit at any given $t$.

\begin{Definition}
Suppose $\sigma_1,\sigma_2\in \clR_G$.
	Let $\clF_{\sigma_1,\sigma_2}$ to be the set of $\fru\in\clF'$, such that there exist $p,q\in M$ with the following properties:
	\begin{enumerate}
		\item $\Orb(p)\neq \Orb(q),$
		\item $\nabla_p\iota(\fru)=0,  \nabla_q\iota(\fru)=0, $
		\item $\ker \Hess_p \iota(\fru)/T_p\Orb p$  represents  $\sigma_1$,  
		\item 
		 $\ker \Hess_q \iota(\fru)/T_q\Orb q$ represents  $\sigma_2$.
	\end{enumerate}
\end{Definition}

\begin{Lemma}
\label{lem_two_degeneracy_codim}
	$\clF_{\sigma_1,\sigma_2}$ is a $C^\infty$--subvariety of  $\clF'$ with codimension at least $d(\sigma_1)+d(\sigma_2)$.
\end{Lemma}

\begin{proof}
The proof is essentially the same as Lemma \ref{lem_one_degeneracy_codim}.
Let 
	\begin{multline*}
		\widetilde \clF_{\sigma_1,\sigma_2} := \{(\fru,p,q)\in \clF'\times M \times M\, | \Orb(p)\neq \Orb(q),
		\\
		\, \nabla_p \iota(\fru) = 0,\, \ker \Hess_p \iota(\fru)/T_p\Orb p \mbox{ represents }  \sigma_1,
		\\
		\, \nabla_q \iota(\fru) = 0,\, \ker \Hess_q \iota(\fru)/T_p\Orb p \mbox{ represents }  \sigma_2
		\}.
	\end{multline*}

	Suppose $(\fru,p,q)\in \widetilde \clF_\sigma$, let $D_p$, $D_q$ be a slice of $p$ and $q$ respectively such that the images of $D_p$ and $D_q$ are disjoint in the quotient set $M/G$. we claim that there exists an open neighborhood $U$ of $(\fru,p,q)$ in $\clF'\times D_p\times D_q$, such that 
	\begin{enumerate}
		\item $\widetilde \clF_{\sigma_1,\sigma_2}\cap U$ is a Banach manifold,
		\item The projection of $\widetilde \clF_{\sigma_1,\sigma_2}\cap U$ to $\clF'$ is Fredholm and has index $-d(\sigma_1)-d(\sigma_2)$.
	\end{enumerate}
	The result then follows from the above claim and the separability of $\clF'\times M\times M$.
	
	To prove the claim, notice that changing the $G$--invariant metric of $M$ does not change the set $\widetilde{\clF}_{\sigma_1,\sigma_2}$, therefore we may assume without loss of generality that $D_p$ is  $\Stab(p)$--equivariantly diffeomorphic to a closed disk of $S_p$ via an isometry, and also  $D_q$ is $\Stab(q)$--equivariantly diffeomorphic to a closed disk of $S_q$ via an isometry. Therefore $TD_p$ and $TD_q$ are canonically trivialized by parallel translations. Let $D_p^0$ be the fixed-point subset of $D_p$ with respect to the $\Stab(p)$--action, let $S_p^0$ be the fixed-point subset of $S_p$, and define $D_q^0$, $S_q^0$ similarly. Then
$$\widetilde \clF_{\sigma_1,\sigma_2}\cap (\clF'\times D_p\times D_q)=\widetilde \clF_{\sigma_1,\sigma_2}\cap (\clF'\times D_p^0\times D_q^0),$$
and the intersection is given by the pre-image of 
$$\{0\}\times \sym_{\Stab(p),\sigma_1}(S_p)\times \sym_{\Stab(q),\sigma_2}(S_q)$$
 of the map
\begin{align*}
\varphi:\clF'\times D_p^0\times D_q^0 &\to (S_p^0)^* \times \sym_{\Stab(p)}(S_p)\times (S_q^0)^* \times \sym_{\Stab(q)}(S_q) \\
(\frv,s,t)&\mapsto (\nabla_s\iota(\frv)|_{D_p^0} , \Hess \iota(\frv)|_{S_p},\nabla_s\iota(\frv)|_{D_q^0} , \Hess \iota(\frv)|_{S_q}).
\end{align*}
By Lemma \ref{lem_surj_jet_M_sigma_2pts} and the density of $\ima \iota$ in $C^\infty_G(M)$, the tangent map of $\varphi$ is surjective at $(\fru,p,q)$, therefore  by Lemma \ref{lem_sym_sigma_submanifold_finite_dim}, 
\begin{equation}
\label{eqn_varphi_inverse_singular_set_finite_dim_two_pts}
\varphi^{-1}\big(\{0\}\times \sym_{\Stab(p),\sigma_1}(S_p)\times \sym_{\Stab(q),\sigma_2}(S_q)\big)
\end{equation}
is a Banach manifold near $(\fru,p,q)$.

The embedding of \eqref{eqn_varphi_inverse_singular_set_finite_dim_two_pts} to $\clF'\times D_p^0\times D_q^0 $ is Fredholm with index 
$$-\dim D_p^0 - \dim D_q^0 - d(\sigma_1)-d(\sigma_2),$$
and the projection from $\clF'\times D_p^0\times D_q^0 $ to $\clF'$ is Fredholm with index $\dim D_p^0 + \dim D_q^0$. Therefore the projection of \eqref{eqn_varphi_inverse_singular_set_finite_dim_two_pts} to $\clF'$ is Fredholm and has index $-d(\sigma_1)-d(\sigma_2)$.
\end{proof}

 Lemma \ref{lem_one_degeneracy_codim} and Lemma \ref{lem_two_degeneracy_codim} have the following immediate corollary.

\begin{Corollary}
\label{cor_wall_cross_finite_dim}
	Suppose $\frp_t$, $t\in[0,1]$ is a generic path from $\frp_0$ to $\frp_1$ in $\clF'$ that intersects all $\clF_\sigma$ and $\clF_{\sigma_1,\sigma_2}$ transversely, and let $f_t=\iota(\frp_t)$. Then there are at most countably many $t$ such that $f_t$ is not $G$--Morse; for every such $t$, there is exactly one critical orbit $\Orb(p)$ of $f_t$, and $\ker \Hess_p f_t/T_p\Orb(p)$ is an irreducible representation of $\Stab(p)$. \qed
\end{Corollary}

\begin{remark}
	\label{rem_countable_but_not_finite}
	The transversality argument does not immediately imply that the set of non-$G$--Morse $f_t$ is finite. This is because the sets $\mathcal{F}_\sigma$ and $\mathcal{F}_{\sigma_1,\sigma_2}$ are projection images of Banach submanifolds of $\clF'\times M$ with negative indices, and they are not necessarily submanifolds of $\clF'$. This issue will be resolved in the proof of Proposition \ref{prop_equivairant_cerf_finite_dim_induction} at the end of Section \ref{subsec_proof_thm_equi_cerf}.
\end{remark}

\subsection{Proof Theorem \ref{thm_equivairant_cerf_finite_dim}}
\label{subsec_proof_thm_equi_cerf}

This subsection finishes the proof of Theorem \ref{thm_equivairant_cerf_finite_dim} using the previous results and an induction argument. 

Define $\mathfrak{M}:=\bZ^{\ge 0}\times \bZ^{\ge 0}$, 
and let $\succeq$ be the lexicographical order on $\mathfrak{M}$. 
Namely, for $(a,b), (a',b')\in\mathfrak{M}$, we have $(a,b)\succeq (a',b')$ if and only if $a>a'$, or $a=a'$ and $b\ge b'$. We write $(a,b) \succ (a',b')$ if $(a,b) \succeq (a',b')$ and $(a,b) \neq (a',b')$. Define $\prec$ and $\preceq$ to be the reverses of $\succ$ and $\succeq$ respectively. 

\begin{Definition}
\label{def_m(H)}
	Suppose $H$ is a compact Lie group, define 
	$$m(H):=(\dim H, \# \pi_0(H))\in\mathfrak{M},$$
	where $\# \pi_0(H)$ denotes the number of connected components of $H$.
\end{Definition}

We define the following filtration on $\Bif_G$.

\begin{Definition}
	Suppose $k\in\mathfrak{M}$.
	Define $\Bif_G^{(k)}\subset \bZ\clR_G$ to be the subgroup generated by 
	$$
	i^H_G\big(\xi_H(V,V',g)\big) \quad\mbox{and }\quad \sigma  +\sigma \oplus \bR,
	$$
for all possible choices of $H$ such that $m(H)\preceq k$, and for all possible choices of $V$, $V'$, $g$, and
 $\sigma\in \clR_G([H])$.
\end{Definition}

Then Lemma \ref{lem_direct_sum_Bif_unfiltered} also holds for the filtrations: 
\begin{Lemma}
\label{lem_direct_sum_Bif}
Let $k\in \mathfrak{M}$.
Suppose $\sigma \in \clR_G([G])$, then 
$\Bif_G^{(k)}\oplus \, \sigma \subset \Bif_G^{(k)}.$ \qed
\end{Lemma}

Notice that $(\mathfrak{M},\succeq)$ is a totally ordered set, and every non-empty subset of $\mathfrak{M}$ has a minimum element. Therefore one can apply induction on $(\mathfrak{M},\succeq)$.
We prove the following stronger version of Theorem \ref{thm_equivairant_cerf_finite_dim} using induction on $k\in\mathfrak{M}$:

\begin{Proposition}
\label{prop_equivairant_cerf_finite_dim_induction}
	Suppose $f_0$ and $f_1$ are $G$--Morse functions on $M$, and suppose there exists a smooth 1-parameter family $f_t$, $t\in[0,1]$ connecting $f_0$ and $f_1$ such that 
	$$\nabla f_t\neq 0 \mbox{ on }\partial M$$
	 for all $t$. Let 
	 $$
	 k=\max_{p\in M}\, m\big(\Stab(p)\big),
	 $$
	 Then $\ind f_0-\ind f_1\in \Bif_G^{(k)}$. 
\end{Proposition}

Notice that $m(H)\succeq (0,1)$ for all compact Lie groups $H$. If
$$\max_{p\in M}\, m\big(\Stab(p)\big)=(0,1),$$
 then $\Stab(p)$ is trivial for all $p\in M$, thus $M/G$ is a manifold, and $G$--equivariant functions on $M$ are equivalent to smooth functions on $M/G$. Therefore Proposition \ref{prop_equivairant_cerf_finite_dim_induction} follows from the classical Cerf's theorem.

Now suppose $k\succ (0,1)$, and suppose Proposition \ref{prop_equivairant_cerf_finite_dim_induction} is true for all $(M,G)$ such that
$$\max_{p\in M}\, m\big(\Stab(p)\big)\prec k.$$ 
We prove Proposition \ref{prop_equivairant_cerf_finite_dim_induction} when 
$$\max_{p\in M}\, m\big(\Stab(p)\big)=k.$$

We start with several technical lemmas.

\begin{Lemma}
\label{lem_connect_local_positive}
	Let $H$ be a compact Lie group, and suppose $V$ is an orthogonal representation of $H$ without trivial components. Let $B(r)$ be the closed ball in $V$ centered at $0$ with radius $r$. Suppose $f_t$ with $t\in[0,1]$ is a smooth 1-parameter family of $H$--invariant functions on $B(1)$, such that 	
	$\Hess f_i$ is positive definite at $0$ for $i=0,1$.
	 Then there exists a family of $H$--invariant functions $\tilde f_t$, $t\in [0,1]$ on $B(1)$, such that 
	\begin{enumerate}
		\item $\tilde f_t = f_t$ for $t=0,1$,
		\item there exists a neighborhood $N(\partial B(1))$ of $\partial B(1)$, such that $\tilde f_t = f_t$ on $N(\partial B(1))$  for all $t$,
		\item $\Hess \tilde f_t$ is positive definite at $0$ for all $t\in[0,1]$.
	\end{enumerate}
\end{Lemma}

\begin{proof}
Let $\chi:B(1)\to\bR$ be a $G$--invariant cut-off function that equals $1$ near $0$ and equals $0$ near $\partial B(1)$. Then the family
	$$\tilde{f}_t := (1-\chi)\cdot f_t+ \chi \cdot \big((1-t)f_0+tf_1\big)$$
	satisfies the desired conditions.
\end{proof}

\begin{Lemma}
\label{lem_index_on_slice}
	Let $H$ be a closed subgroup of $G$, let $V$ be a finite-dimensional orthogonal $H$--representation, and let $B(1)$ be the closed unit ball of $V$. Let $M=G\times_H B(1)$, let $f_0$ and $f_1$ be $G$--Morse functions on $M$. Let 
	$$D=H\times_H B(1)\subset M.$$
	 Then $f_0|_{D}$ and $f_1|_{D}$ are $H$--Morse functions. Suppose $\ind f_0|_D-\ind f_1|_D\in \Bif_H^{(k)}$, then $\ind f_0-\ind f_1\in \Bif_G^{(k)}$.
\end{Lemma}

\begin{proof}
We have
	$$i^H_G \,(\Bif_H^{(k)})\subset \Bif_G^{(k)},$$
	 and 
	 $$i^H_G\,(\ind f_0|_D-\ind f_1|_D) = \ind f_0-\ind f_1.$$
	 Hence the lemma is proved.
\end{proof}

\begin{Lemma}
	\label{lem_flip_sign_of_Hess}
	Suppose $V$ is a finite-dimensional orthogonal $G$--representation without trivial components, and let $M=B(1)$ be the closed unit ball of $V$. Let $f$ be a $G$--invariant Morse function on $M$ such that the Hessian of $f$ at $0$ is not positive definite. Then there exists $\hat f$ such that 
\begin{enumerate}
	\item $\hat f=f$ on a neighborhood of $\partial M$,
	\item $\hat f$ is $G$--Morse,
	\item $\ind f - \ind \hat f$ has the form $\xi_G(V_1,V_2,g)$, where $V_1\oplus V_2$ is isomorphic to a direct summand  of $V$ as $G$--representations, and $V_1$ is irreducible,
	\item the number of positive eigenvalues of $\Hess \hat f|_{x=0}$ (counted with multiplicities) is strictly greater than the number of positive eigenvalues of $\Hess f|_{x=0}$ (counted with multiplicities).
\end{enumerate}
\end{Lemma}

\begin{remark}
By Lemma \ref{lem_change_of_index_under_irred_bif}, Condition (3) above states that the change of the total index from $f$ to $\hat f$ is the same as the change of the total index under an irreducible bifurcation.
\end{remark}

\begin{remark}
	Since $V$ has no trivial components as a $G$--representation, the origin $x=0$ is a critical point for every $G$--invariant smooth function on $M$. 
\end{remark}

\begin{proof}
		Let $\mathcal{S}$ be the set of $G$--Morse functions on $M$ such that the desired $\hat f$ exists. 
		We prove that if $f$ is $G$--Morse and $\Hess f$ is not positive definite at $0$, then $f\in \mathcal{S}$. 	The general idea is to first construct a family of elements in the set $\clS$ using Example \ref{example_bifurcation_model_finite_dim_sum}, and then extend the examples by gluing them to other $G$--Morse functions.
	We present the proof in several steps. 
	
	In the following, all eigenvalues are counted with multiplicities. 
	
	\vspace{.5\baselineskip}	
	\textbf{Step 1.}
	We first find a family of elements in the set $\clS$. Suppose $V$ is orthogonally decomposed as $V=V_1\oplus V_2$ as $G$--representations such that $V_1$ is an irreducible representation. Here, $V_2$ is allowed to be the zero space. Assume $h$ is a $G$--invariant quadratic function on $V_2$ such that $\Hess h$ has a full rank. Consider the function
	\begin{align*}
	f_a : V = V_1\oplus V_2 & \to \bR \\
	    x = (x_1,x_2) & \mapsto -a \|x_1\|^2 - \|x_1\|^4 + h(x_2).
	\end{align*}
	We show that there exists $\epsilon >0$, which may depend on $h$, such that $f_a\in \mathcal{S}$ for all $a\in (0,\epsilon)$. 
	
	Note that $f_a = F_{V_1\oplus V_2}(-a,\cdot)$, where $F_{V_1\oplus V_2}$ is the function defined by Example \ref{example_bifurcation_model_finite_dim_sum} with respect to the given quadratic function $h$. Let $\chi$ be a smooth  $G$--invariant cut-off function on $V$ such that $\chi = 1$ on $B(1/3)$ and $\chi = 0$ on $B(1) - B(2/3)$. Consider the function
	$$
	\hat f_a (x):= \chi\cdot F_{V_1\oplus V_2}(a,x) + (1-\chi) f_a.
	$$
	Then $\hat f_a = f_a$ on a neighborhood of $\partial M$, and the number of positive eigenvalues of $\Hess \hat f_ a|_{x=0}$ is strictly greater than the number of positive eigenvalues of $\Hess f_a|_{x=0}$. We also know that $\hat f_a = F_{V_1\oplus V_2}(a,\cdot)$ on $B(1/3)$. On $B(1)- B(1/3)$, we have 
	$$\lim_{a\to 0^+} \hat f_a = \chi\cdot F_{V_1\oplus V_2}(0,\cdot) + (1-\chi) f_0 = f_0 = -\|x_1\|^4 + h(x_2)
	$$
	in the $C^\infty$ topology, thus there exists $\epsilon >0$ such that for all $a\in(0,\epsilon)$, the function $\hat f_a$ has no critical point on $B(1)- B(1/3)$. Hence we conclude that $\hat f_a$ is $G$--Morse and 
	$$\ind f_a - \ind \hat f_a = \ind  F_{V_1\oplus V_2}(-a,\cdot) - \ind F_{V_1\oplus V_2}(a,\cdot)= \xi_G(V_1, V_2', g),$$
    where $V_2'$ is the span of negative eigenvectors of $h$, and $g$ is the function on the unit sphere of $V_1$ that appears in the definition of $F_{V_1\oplus V_2}$. This shows $f_a\in\clS$.
		\vspace{.5\baselineskip}

	\textbf{Step 2.}
	Recall that the set of $G$--Morse functions on $M$ is an open subset of $C^\infty_G(M)$. 
	We show that $\mathcal{S}$ is also open in the $C^\infty$ topology.
	
	Assume $f_0\in \mathcal{S}$, let $\hat {f}_0$ be a function that satisfies the stated conditions with respect to $f_0$. Let $\chi:B(1)\to\bR$ be a smooth $G$--invariant cut-off function that equals $1$ whenever $f_0\neq \hat f_0$ and equals $0$ near $\partial B(1)$. 
	Suppose $f_1$ is a $G$--Morse function such that $\|f_0-f_1\|_{C^\infty}<\epsilon$, where $\epsilon$ will be determined later. 
	
	Assume $\epsilon$ is sufficiently small so that $f_1$ is $G$--Morse whenever $\|f_0-f_1\|_{C^\infty}<\epsilon$. 
	Then the number of positive eigenvalues of $\Hess f_1|_{x=0}$ is the same as the number of positive eigenvalues of  $\Hess f_0|_{x=0}$. 	
	Let 
	$$\hat f_1 := (1-\chi) f_1 + \chi \hat f_0.$$ Then $\hat f_1$ is a $G$--invariant function such that the number of positive eigenvalues of $\Hess \hat f_1|_{x=0}$ is strictly greater than the number of positive eigenvalues of $\Hess f_1|_{x=0}$, and that $\hat f_1= f_1$ on a neighborhood of $\partial M$.
	
	Note that the properties of $\chi$ imply
	$$(1-\chi) f_0 +  \chi \hat f_0=\hat f_0,$$
	so $\hat f_1 - \hat f_0 = (1-\chi)(f_1-f_0)$, and we have 
	$$
		t \hat f_1 + (1-t) \hat f_0  = t(1-\chi) (f_1-f_0) + \hat f_0.
	$$
	Therefore, for $\epsilon$ sufficiently small, the function $t \hat f_1 + (1-t) \hat f_0$ is $G$--Morse for all $t\in [0,1]$. In particular, $\hat f_1$ is $G$--Morse. By Lemma \ref{lem_n_f_invariant_under_Morse_deformation}, we also have 
	$$\ind f_1 - \ind \hat f_1 = \ind f_0 - \ind \hat f_0.$$
 So $f_1\in \mathcal{S}$. 
	\vspace{.5\baselineskip}
		
	\textbf{Step 3.}
	We show that the property $f\in \clS$ can be verified locally near $0$, in the following sense:
	
	Assume $U$ is a subset of the interior of $M$ such that there exists a $G$--equivariant diffeomorphism $\varphi:M\to U$.
	Since $V$ has no trivial components as a $G$--representation, we know that $0$ is the only point in $M$ whose stabilizer is $G$, therefore we must have $\varphi(0) = 0$, and $0$ is in the interior of $U$. 
	
	Define $f^{(U)}(x) = f\circ \varphi(x)$ as a function on $M$. Assume $f^{(U)}$ is $G$--Morse and $f^{(U)}\in \clS$, we claim that $f\in \clS$. 
	
	To prove the claim, suppose $\hat f^{(U)}$ is a $G$--Morse function on $M$ that satisfies the desired conditions with respect to $f^{(U)}$. 
	Define
	$$
	\hat f(x) = 
	\begin{cases}
		f(x) & \text{ if } x\notin U \\
		\hat f^{(U)} \circ \varphi^{-1}(x) & \text{ if } x\in U.
	\end{cases}
	$$
	Since diffeomorphisms preserve critical points and preserve the number of positive eigenvalues of Hessians at critical points, we see that $\hat f$ satisfies the desired conditions with respect to $f$.
		\vspace{.5\baselineskip}
		
	\textbf{Step 4.}
	Recall that in our notation, $\Hess f|_{x=0}$ denotes a symmetric linear map on $TM|_{x=0} \cong V$. Define $h(v) = \langle \Hess f|_{x=0} v, v\rangle$, then $h$ is a $G$--invariant quadratic function on $M$ whose Hessian is equal to $\Hess f|_{x=0}$.
	In this step, we show that if $h\in \clS$, then $f\in \mathcal{S}$.
	
	Assume $h\in \clS$. 
	For $c>1$, let $f^{(c)} (x)= c^2 f(x/c)$, and view $f^{(c)}$ as a function on $M$. Since $f^{(c)}$ converges to $h$ as $c\to +\infty$ in the $C^\infty$ topology, Step 2 implies that there exists $c>1$ so that $f^{(c)}\in \clS$, which implies $f(x/c)\in \clS$. Therefore, by Step 3, we conclude that $f\in \clS$. 
		\vspace{.5\baselineskip}
		
	\textbf{Step 5.} Now we finish the proof of the lemma.
	By Step 4, we only need to show that $f\in \clS$ when $f$ is a $G$--invariant quadratic function whose Hessian has a full rank and is not positive definite. In this case, $V$ can be orthogonally decomposed into $V=V_1\oplus V_2$ as $G$--representations such that $V_1$ is irreducible, and $f$ has the form
	\begin{align*}
	f: V = V_1\oplus V_2 & \to \bR \\
	    x = (x_1, x_2) & \mapsto -\lambda^2 \|x_1\|^2 + h(x_2),
	\end{align*}
	where $h$ is a $G$--invariant quadratic function on $V_2$ such that $\Hess h_2$ has a full rank and $\lambda\neq 0$. By rescaling $f$, we may also assume without loss of generality that $\lambda \ge 2$. Define 
	\begin{align*}
		f_a : V_1 \oplus V_2 & \to \bR \\
		x = (x_1, x_2) & \mapsto -a \|x_1\|^2 - \|x_1\|^4 + \frac{h(x_2)}{4},
	\end{align*}
then Step 1 implies that there exists $a\in (0,1)$ such that $f_a\in \clS$. 

Consider the map 
\begin{align*}
\varphi: V_1\oplus V_2& \to V_1\oplus V_2\\
     (x_1,x_2) & \mapsto (\sqrt{a + \|x_1\|^2} \cdot (x_1/\lambda), x_2/2).
\end{align*}
Then $\varphi$ is a smooth $G$--equivariant self-diffeomorphism of $V=V_1\oplus V_2$, and we have $f\circ \varphi = f_a$. Since $\lambda\ge 2$ and $a\in (0,1)$, it is straightforward to verify that if $x=(x_1,x_2)$ satisfies $\|x\|\le 1$, then $\|\varphi(x)\|<1$. Therefore, applying Step 3 with $U=\varphi(M)$ and the assumption that $f_a\in \clS$, we conclude that $f\in \clS$. This finishes the proof of the lemma.
\end{proof}

The next few lemmas verify several special cases of Proposition \ref{prop_equivairant_cerf_finite_dim_induction}. We will later show that the general case can be deduced to these special cases by taking local slices and using the transversality results from Section \ref{subsec_transversality_finite_dim}.

The following lemma is the essence of the induction step.

\begin{Lemma}
\label{lem_special_case_linear_no_trivial_component}
Suppose $m(G)\preceq k$, and let $V$ be an orthogonal representation of $G$ without trivial component. Let $B(r)$ be the closed ball in $V$ centered at $0$ with radius $r$. Suppose $f_t$, $t\in[0,1]$ is a smooth 1-parameter family of $G$--invariant functions on $B(1)$, such that 
	\begin{enumerate}
		\item $f_0$ and $f_1$ are $G$--Morse,
		\item $\nabla f_t\neq 0$ on $\partial B(1)$ for all $t$.
	\end{enumerate}
	 Then $\ind f_0-\ind f_1\in \Bif_G= \Bif_G^{(k)}$.
\end{Lemma}

\begin{proof}
Applying Lemma \ref{lem_flip_sign_of_Hess} repeatedly, we know that there exist $\hat f_i$, $i=0,1$, such that
\begin{enumerate}
	\item $\hat f_i=f_i$ on a neighborhood of $\partial B(1)$,
	\item $\hat f_i$ is $G$--Morse,
	\item $\ind \hat f_i- \ind f_i$ is in the subgroup of $\Bif_G$ generated by elements of the form $\xi_G(V_1,V_2,g)$,
	\item $\Hess \hat f_i$ is positive definite at $0$.
\end{enumerate}
By Lemma \ref{lem_connect_local_positive}, $\hat f_0$ and $\hat f_1$ can be connected by a 1-parameter family of $G$--invariant functions $\hat f_t$, such that $\nabla \hat f_t\neq 0$ on $\partial B(1)$, and $\Hess \hat f_t$ is positive definite at $0$ for all $t$. Therefore, there exists $\epsilon$ sufficiently small, such that for all $t$, we have $\nabla \hat f_t \neq 0$ on $\partial B(\epsilon)$, and $\hat f_t|_{B(\epsilon)}$ is $G$--Morse with a unique critical point at $0$. 

Since $V$ has no trivial component, we have
$$
k':= \max_{p\in B(1)-B(\epsilon)}\, m\big(\Stab(p)\big)\prec k.
$$
Applying the induction hypothesis on the closure of $B(1)-B(\epsilon)$ yields
$$\ind \hat f_0-\ind \hat f_1\in \Bif_G^{(k')}\subset \Bif_G^{(k)},$$ therefore the lemma is proved.
\end{proof}

\begin{Lemma}
\label{lem_Hess_change_at_non_trivial_direction}
	Suppose $f_0$ and $f_1$ are connected by a smooth family $f_t$, $t\in[0,1]$ of $G$--invariant functions $M$, with the following properties:
	\begin{enumerate}
		\item there exists $t_0\in(0,1)$, such that  $f_t$ is Morse when $t\neq t_0$,
		\item 
	the function $f_{t_0}$ has exactly one degenerate critical orbit at $\Orb(p)$, and $\nabla f_{t_0}\neq 0$ on $\partial M$,
	\item $\ker \Hess_p f_{t_0}/T_p\Orb(p)$ contains no trivial component as a representation of $\Stab(p)$,
	\item $m\big(\Stab(p)\big)\preceq k$.
	\end{enumerate}
	Then $\ind f_0-\ind f_1\in\Bif_G^{(k)}$.
\end{Lemma}

\begin{proof}
Since the critical orbits of $f_{t_0}$ are discrete on $M-\Orb(p)$, there exists a slice $D$ of $p$, such that $\nabla f_{t_0}\neq 0$ on $\partial D$. By shrinking the interval of $t$ if necessary and invoking Lemma \ref{lem_n_f_invariant_under_Morse_deformation}, we may assume without loss of generality that $\nabla f_{t}\neq 0$ on $\partial D$ for all $t$. 

Recall that $U_p(D)$ is the neighborhood of $\Orb(p)$ defined by Definition \ref{def_U_p(D)}.
By Lemma \ref{lem_n_f_invariant_under_Morse_deformation} again,
we only need to compare the total indices of $f_i$ on $U_p(D)$ for $i=0,1$. By Lemma \ref{lem_index_on_slice}, we may further assume without loss of generality that $M=D$ and $\Stab(p)=G$.

Identify $D$ with the closed unit ball of $S_p$, and decompose $S_p$ as $S_p=S_p'\oplus S_p''$, where $S_p'=\ker \Hess_p f_{t_0}$, and $S_p''$ is the orthogonal complement of $S_p'$ in $S_p$.
By the assumptions, the Hessian of $f_{t_0}$ is non-degenerate on $S_p''$.

Let $B_{V'}(r)$ be the closed ball in $V'$ centered at $0$ with radius $r$. Let $\pi',\pi''$ be the orthogonal projections of $S_p$ onto $S_p',S_p''$ respectively.  Define
$$
M_{t,\epsilon}:=\{x\in D|\pi''\nabla f_t(x) = 0, \|\pi'x\|\le \epsilon\}.
$$
By the implicit function theorem, for $t$ sufficiently close to $t_0$ and $\epsilon$ sufficiently small, and by shrinking $D$ if necessary, $M_{t,\epsilon}$ is  $G$--equivariantly diffeomorphic to $B_{V'}(\epsilon)$ via $\pi'$. 

Let $\sigma\in \clR_G([G])$ be represented by the subspace of $S_p''$ generated by the eigenvectors of $\Hess_p f_{t_0}|_{S_p''}$ as a $G$--representation. Let $\hat f_t$  be the restrictions of $f_t$ on $M_{t,\epsilon}$. Under the assumption that $M=D$ and $(\epsilon,t)$ being sufficiently close to $(0,t_0)$ with $t\neq t_0$, we claim that $\hat f_t $ is $G$--Morse on $M_{\epsilon, t}$, and 
$$\ind f_t = \ind \hat f_t\oplus \sigma.$$ 
The desired result then follows from this claim by Lemma \ref{lem_direct_sum_Bif} and Lemma \ref{lem_special_case_linear_no_trivial_component}.

To prove the claim, suppose $(\epsilon,t)$ is sufficiently close to $(0,t_0)$ such that $M_{t,\epsilon}$ is a $G$--manifold, and suppose $q\in M_{t,\epsilon}$ is a critical point of $\hat f_t$ on $M_{t,\epsilon}$, then $\Orb(q)\subset M_{t,\epsilon}$. Recall that $S_q\subset T_qD$ is the orthogonal complement of $T_q\Orb(q)$ in $T_qD$. Let $S_q'$ be the orthogonal complement of $T_q\Orb(q)$ in $T_q M_{t,\epsilon}$, and let $S_q''$ be the orthogonal complement of $T_q M_{t,\epsilon}$ in $T_qD$. Then we have
$$
S_q=S_q'\oplus S_q''.
$$
Suppose $\Hess_q f_t : S_q\to S_q$ is given by the matrix
$\begin{pmatrix} N_{11} & N_{12}\\ N_{21} & N_{22} \end{pmatrix}$ under the decomposition above, then 
 $$N_{11}: S_q'\to S_q'$$
 equals $\Hess_q \hat f_t|_{S_q'}$. Recall that $M=D$, which is identified with a closed ball in a linear space. For each $v\in S_q'$, we have
$$
N_{11}(v) + N_{21}(v) = (\Hess_q  f_t) (v) = \frac{\partial}{\partial v} (\nabla f_t) \in \ker \pi'' = S_p'
$$
by the definition of $M_{t,\epsilon}$.
Therefore, there exist constants $z_1$ and $\epsilon_1$, such that for all $(\epsilon,t)\in(0,\epsilon_1)\times (t_0-\epsilon_1,t_0+\epsilon_1)$, we have
\begin{equation}
\label{eqn_N11_controls_N21_finite_dim}
	z_1\cdot \|N_{11}(v) \|\ge \|N_{21}(v)\|.
\end{equation}
In particular, we have $\ker N_{11}\subset \ker N_{21}$, thus 
$$
\ker N_{11} \subset \ker 
\begin{pmatrix}
N_{11} & N_{12}\\ N_{21} & N_{22}
\end{pmatrix}.
$$
By the assumptions, $f_t$ is $G$--Morse for all $t\neq t_0$, which implies $\ker 
\begin{pmatrix}
N_{11} & N_{12}\\ N_{21} & N_{22}
\end{pmatrix}=0$ for all $t\neq t_0$. Therefore $N_{11}$ is invertible, and hence $\hat f_t$ is $G$--Morse, when $\epsilon<\epsilon_1$ and $t\in  (t_0-\epsilon_1,t_0)\cup(t_0,t_0+\epsilon_1)$.

For $(\epsilon,t)$ sufficiently close to $(0,t_0)$, the map $$
N_{22} : S_q'' \to S_q''
$$
is an approximation of $\Hess_p f_{t_0}|_{S_p''}$, which is invertible. By the continuity of $\Hess_q f_t$, we have
$$
\lim_{(\epsilon,t)\to(0,t_0)} \big(\|N_{11}\| + \|N_{21}\| + \|N_{12}\| \big)=0. 
$$
Therefore, there exist constants $z_2,\epsilon_2$, such that for all 
$$(\epsilon,t)\in(0,\epsilon_2)\times (t_0-\epsilon_2,t_0+\epsilon_2),$$
 we have 
 $$
 \|N_{22}^{-1}\| \le z_{2}, \quad \|N_{12}\|\le \frac{1}{2z_1\,z_2}.
 $$
 Let $\epsilon_0=\min\{\epsilon_1,\epsilon_2\}$, and suppose $\epsilon<\epsilon_0$, $t\in(t_0-\epsilon_0,t)\cup (t,t_0+\epsilon_0)$.
 Let $s\in[0,1]$. Notice that
	\begin{align}
	&\begin{pmatrix}
		\id & -s N_{12} \circ N_{22}^{-1} \\
	0  & \id
	\end{pmatrix}
	\cdot 
	\begin{pmatrix}
		N_{11} & sN_{12} \\
		sN_{21} & N_{22}
	\end{pmatrix}
	\cdot
	\begin{pmatrix}
		\id & 0 \\
		-s N_{22}^{-1}\circ N_{21} & \id
	\end{pmatrix}
	\nonumber
	\\
	\label{eqn_elementary_matrix_transform_N22_finite_dim}
	= &
	\begin{pmatrix}
		N_{11}-s^2 N_{12}\circ N_{22}^{-1}\circ N_{21} & 0 \\
		0 & N_{22}
	\end{pmatrix}.
	\end{align}
	For every $v\in S_q'$, we have 
	\begin{align*}
		\big\|\big(N_{11}-s^2 N_{12}\circ N_{22}^{-1}\circ N_{21}\big)(v)\big\| 
		&\ge
		 \|N_{11}(v)\|- s^2\|N_{12}\|\cdot \| N_{22}^{-1}\|\cdot \|N_{21}(v)\|
		\\
		& \ge \|N_{11}(v)\|- s^2\|N_{12}\|\cdot \| N_{22}^{-1}\|\cdot (z_1 \|N_{11}(v)\|)
		\\
		& \ge \|N_{11}(v)\| - s^2 \cdot \frac{1}{2z_1\,z_2}\cdot z_2 \cdot  (z_1 \|N_{11}(v)\|)
		\\
		& \ge \frac12  \|N_{11}(v)\|.
	\end{align*}
Since $N_{11}$ is injective, the estimates above imply that $N_{11}-s^2 N_{12}\circ N_{22}^{-1}\circ N_{21} $ is injective, therefore it 
	is invertible for all $s\in[0,1]$. By \eqref{eqn_elementary_matrix_transform_N22_finite_dim}, the map
	\begin{equation}
	\label{eqn_matrix_Nij_homotopy_finite_dim}
		\begin{pmatrix}
		N_{11} & sN_{12} \\
		sN_{21} & N_{22}
	\end{pmatrix}
	\end{equation}
	 is invertible for all $s\in[0,1]$.	Moreover, the family \eqref{eqn_matrix_Nij_homotopy_finite_dim} is $\Stab(q)$--equivariant for all $s\in[0,1]$. Therefore the equivariant index of $f_t$ at $q$ is represented by the subspace of $S_q$ generated by the negative eigenvectors of $N_{11}$ and $N_{22}$, and hence the claim is proved.
\end{proof}

\begin{Lemma}
\label{lem_Hess_change_at_trivial_direction}
	Suppose $f_0$ and $f_1$ are connected by a smooth family $f_t$, $t\in[0,1]$ of $G$--invariant functions $M$, with the following properties:
	\begin{enumerate}
		\item $\nabla f_{t}\neq 0$ on $\partial M$ for all $t$,
	\item $\Hess_\sigma f_t$ (from Definition \ref{def_Hess_sigma_f}) is non-degenerate for all $t$ and all $\sigma$.
	\end{enumerate}
Suppose also that
	$$\max_{p\in M}\, m\big(\Stab(p)\big) = k.$$
	Then $\ind f_0-\ind f_1\in\Bif_G^{(k)}$.
\end{Lemma}

\begin{proof}
	Suppose $\sigma\in \clR_G$. By the assumptions, $\Hess_\sigma f$ is always non-degenerate, and $f_t$ induces a family of smooth functions on $M_\sigma/G$. Therefore, the result follows from the classical Cerf theory on $M_\sigma/G$. In fact, $\ind f_0-\ind f_1$ is contained in the subgroup of $\bZ\clR_G$ generated by 
	$$
	\sigma  +\sigma \oplus \bR
	$$
	for $\sigma\in \cup_{p\in M}\, \clR_G([\Stab(p)])$.
\end{proof}

We can now finish the proof of Proposition \ref{prop_equivairant_cerf_finite_dim_induction}.
\begin{proof}[Proof of Proposition \ref{prop_equivairant_cerf_finite_dim_induction}]
Recall that we defined $\frp_0,\frp_1\in\clF'$ such that $\iota(\frp_0)=f_0$, $\iota(\frp_1)=f_1$, and $\frp_0$ and $\frp_1$ are on the same connected component of $\clF'$.

Let $\frp_t$, $t\in[0,1]$ be a generic path from $\frp_0$ to $\frp_1$ that is transverse to all $\clF_\sigma$ and $\clF_{\sigma_1,\sigma_2}$. 
	Let $f_t=\iota(\frp_t)$. By Corollary \ref{cor_wall_cross_finite_dim}, for every $t_0$ such that $f_{t_0}$ is not not Morse, the function $f_{t_0}$ has exactly one degenerate critical orbit $\Orb(p)$ where $\Hess_p f/T_p\Orb(p)$ is an irreducible $\Stab(p)$--representation.

Let $\clS$ be the set of $t\in[0,1]$ such that $f_t$ is not $G$--Morse, then $\clS$ is a closed subset of $[0,1]$. 
For each $\sigma \in \clR_G$, let $\clS_\sigma$ be the subset of $t\in \clS$, such that 
 $f_{t}$ is degenerate at 
a critical orbit $\Orb(p)\subset M_\sigma$ where $\ker \Hess_p f_t/T_p\Orb(p)$ is in the isomorphism class $\sigma$. 

Suppose $\sigma$ is given by a non-trivial irreducible representation.
For each $t_0\in \clS_\sigma$, let $\Orb(p)$ be the critical orbit of $f_{t_0}$, and let $D$ be a slice of $p$, let $D^0\subset D$ be the fixed-point set of the $\Stab(p)$ action. Then $f_{t_0}$ induces a Morse function on $M_\sigma/G$, therefore for $\epsilon>0$ sufficiently small, there is a unique smooth map 
$$p(t):(t_0-\epsilon,t_0+\epsilon)\to D_0,$$
such that $p(t_0)=p$, and $p(t)$ is a critical point of $f_t$. Since the path $\frp_t$ intersects $\clF_\sigma$ transversely (or more precisely, it intersects the projection of $\widetilde{\clF}_\sigma$ in the proof of Lemma \ref{lem_one_degeneracy_codim} transversely),  we have 
$$
\frac{d}{dt}\Big|_{t=0}\Hess_{p(t)} f_t \neq 0 \mbox{ on }  \ker \Hess_p f \cap S_p.
$$
Therefore $t_0$ is an isolated point of $\clS$. As a result, $\clS_\sigma$ is a finite set and is isolated in $\clS$.

Let $\clS'\subset \clS$ be the union of $\clS_\sigma$ for all $\sigma \in \clR_G$ that are given by non-trivial representations. Then by the previous argument, we can divide the interval $[0,1]$ into finitely many sub-intervals, such that each interval is either disjoint from $\clS'$, or contains exactly one point of $\clS$ which is also in $\clS'$. In the former case, the sub-interval defines a family of $G$--invariant functions that satisfies the conditions of Lemma \ref{lem_Hess_change_at_trivial_direction}; in the latter case, the sub-interval defines a family that satisfies the conditions of Lemma \ref{lem_Hess_change_at_non_trivial_direction}. Therefore, Proposition \ref{prop_equivairant_cerf_finite_dim_induction} is proved.
\end{proof}

\subsection{Invariant counting of critical orbits}

This section studies the weighted counting of critical orbits of $G$--Morse functions and prove Theorem \ref{thm_finite_dim_Cerf_into}. By Theorem \ref{thm_equivairant_cerf_finite_dim}, suppose 
$$w:\bZ\clR_G\to \bZ$$
is a homomorphism such that $w=0$ on $\Bif_G$, then the value of $w(\ind f)$ does not depend on the choice of the $G$--Morse function $f$. We will classify all such functions $w$ by investigating the group structure of $\bZ\clR_G/\Bif_G$. We start with the following definition.

\begin{Definition}
	Let $\clR_G^{(0)}\subset \clR_G$ be the subset represented by zero representations.
\end{Definition}

By definition,
	$\clR_G^{(0)}$ is in one-to-one correspondence with $\Conj(G)$.

\begin{Lemma}
\label{lem_reduce_independent_of_g}
Let $H$ be a compact Lie group, let $k=m(H)\in \mathfrak{M}$, and assume $k\succ (0,1)$.
	Suppose $V$ is a non-trivial, irreducible, orthogonal representation of $H$, and let $V'$ be another representation of $H$. Let $g_1,g_2$ be two $H$--Morse functions on the unit sphere of $V$. Then there exists $k'\prec k$, such that
	$$
	\xi_H(V,V',g_1)-\xi_H(V,V',g_2) \in \Bif^{(k')}_H.
	$$ 
\end{Lemma}

\begin{proof}
By \eqref{eqn_dir_sum_xi}
 and Lemma \ref{lem_direct_sum_Bif}, we only need to prove the lemma for $V'=0$.
Let $S(V)$ be the unit sphere of $V$, and let
$$
k' = \max_{p\in S(V)} m\big(\Stab(p)\big).
$$
Since $V$ is a non-trivial irreducible representation, we have
$$
k'\prec m(H).
$$
By Proposition \ref{prop_equivairant_cerf_finite_dim_induction}, we have
$$
\ind g_1-\ind g_2\in \Bif^{(k')}_H.
$$
	Therefore the result follows from the definition of $\xi_H(V,V',g_i)$.
\end{proof}

\begin{Lemma}
\label{lem_reduce_independent_of_order}
Let $H$ be a compact Lie group, let $k=m(H)\in \mathfrak{M}$, and assume $k\succ (0,1)$.
	Suppose $V_1,V_2$ are two non-trivial, irreducible, orthogonal representations of $H$, and let $V'$ be another orthogonal representation of $H$. Let $g_1$, $g_2$ be $H$--Morse functions on the unit spheres of $V_1$, $V_2$ respectively. Then there exists $k'\prec k$, such that
	$$
	\xi_H(V_1,V_2\oplus V',g_1) + \xi_H(V_2,V',g_2) - \xi_H(V_2,V_1\oplus V',g_2) - \xi_H(V_1,V',g_1)\in \Bif_H^{(k')}.
	$$ 
\end{Lemma}

\begin{proof}
	By \eqref{eqn_dir_sum_xi}
 and Lemma \ref{lem_direct_sum_Bif}, we only need to prove the lemma for $V'=0$. Let $V=V_1\oplus V_2$. Let $B_V(r)$ be the closed ball in $V$ centered at $0$ with radius $r$. Let $f=-\|x\|^2$ be defined on $B_V(2)$. By the constructions of Example \ref{example_bifurcation_model_finite_dim_irred} and Example \ref{example_bifurcation_model_finite_dim_sum}, one can obtain an $H$--Morse function $f_1$ from $f$ by two irreducible bifurcations at $0$ such that
	\begin{enumerate}
		\item 	$\ind f_1 = - \xi_H(V_1,V_2,g_1) - \xi_H(V_2,0,g_2),$
		\item $\Hess f_1$ is positive definite at $0$. 
	\end{enumerate}
	Similarly, by switching the roles of $V_1$ and $V_2$, one obtains a function $f_2$ such that 
		\begin{enumerate}
		\item 	$\ind f_2 = - \xi_H(V_2,V_1,g_2) - \xi_H(V_1,0,g_1),$
		\item $\Hess f_2$ is positive definite at $0$. 
	\end{enumerate}
By definition, $f_1$ and $f_2$ can be connected by a smooth $G$--invariant functions $f_t$, $t\in[1,2]$, such that $\nabla f_t\neq 0$ on $\partial B(2)$. 

Let
$$
k' = \max_{p\in V-\{0\}} m\big(\Stab(p)\big).
$$
Since $V$ does not contain trivial component, we have
$$
k'\prec m(H).
$$
By Lemma \ref{lem_connect_local_positive} and Proposition \ref{prop_equivairant_cerf_finite_dim_induction}, we have 
	$$
	\xi_H(V_1,V_2,g_1) + \xi_H(V_2,0,g_2) - \xi_H(V_2,V_1,g_2) - \xi_H(V_1,0,g_1)= \ind f_2 -\ind f_1 \in \Bif_H^{(k')}.
	$$ 
Therefore the lemma is proved.
\end{proof}

\begin{Theorem}
\label{thm_equivariant_euler_finite_abstract}
The composition of the homomorphisms 
$$\Phi: \bZ\clR_G^{(0)}\xhookrightarrow{} \bZ\clR_G \to \bZ\clR_G /\Bif_G$$
is an isomorphism.
\end{Theorem}

\begin{proof}
	We first prove that $\Phi$ is a surjection. Let 
	$\pi:\bZ\clR_G \to \bZ\clR_G /\Bif_G$ be the projection map.
	If $H$ is a closed subgroup of $G$ and $\rho:H\to \Hom(V,V)$ is a representation of $H$,  let $[H,V,\rho]\in \clR_G$ be the element represented by $(H,V,\rho)$.
	
	We deduce a contradiction assuming $\Phi$ is not surjective. Let $(H,V,\rho)$ a representation such that 
	\begin{equation}
	\label{eqn_assume_Phi_not_surj}
		\pi([H,V,\rho])\notin\ima \Phi,
	\end{equation}
	and choose one with the minimum value of $m(H)$. If there are multiple such representations with the same value of $m(H)$, we choose one with the minimum value of $\dim V$.
	
	If $\dim V=0$, then $[H,V,\rho]\in \clR_G^{(0)}$ and hence \eqref{eqn_assume_Phi_not_surj} contradicts the definition of $\Phi$. If $\dim V>0$, decompose $V$ as $V=V_1\oplus V_2$, where $V_1$ is an irreducible orthogonal representation of $H$. If $V_1$ is trivial, then we have
	$$
	\pi([H,V,\rho])= -\pi([H,V_2,\rho|_{V_2}]),
	$$
	therefore $\pi([H,V_2,\rho|_{V_2}])\notin \ima \Phi$, contradicting the definition of $(H,V,\rho)$.
	If $V_1$ is non-trivial, let $g$ be an $H$--Morse function on the unit sphere of $V_1$. We have
	$$
	[H,V,\rho] = i^H_G\big(\xi_H^-(V_1,V_2,g)\big),
	$$
	therefore by the definition of $\Bif_G$, we have
	$$\pi([H,V,\rho]) = \pi\circ i^H_G\big(\xi_H^+(V_1,V_2,g)\big).
	$$
	Notice that $i^H_G\big(\xi_H^+(V_1,V_2,g)\big)$ is given by a linear combination of $[H,V_2,\rho|_{V_2}]$ and elements of $\clR_H$ given by the representations of groups $K$ with $m(K)\prec m(H)$, which yields a contradiction to the definition of $(H,V,\rho)$. In conclusion, the map $\Phi$ is surjective.
	
To prove the injectivity of $\Phi$, we construct a projection
$$
p: \bZ\clR_G \to \bZ\clR_G^{(0)},
$$ 
such that $\ker p\supset \Bif_G$, and $p$ restricts to the identity map on $\bZ\clR_G^{(0)}$. Suppose $[H,V,\rho]\in \clR_G$, we define $p([H,V,\rho])$ by induction on $m(H)$ and $\dim V$ as follows. 

If $\dim V=0$, then $[H,V,\rho]\in \clR_G^{(0)}$, and we define $p([H,V,\rho])=[H,V,\rho]$. 

If $\dim V>0$, suppose $p([H',V',\rho'])$ is already defined when $m(H')\prec m(H)$, and when $m(H')=m(H)$, $\dim V'<\dim V$, such that $p=0$ on $\Bif_G^{(k')}$ for all $k'\prec m(H)$. Decompose $V$ as $V=V_1\oplus V_2$, where $V_1$ is an irreducible orthogonal representation of $H$. If $V_1$ is trivial, then we define 
$$
p([H,V,\rho]):=-p([H,V_2,\rho]).
$$
If $V_1$ is non-trivial, let $g$ be an $H$--Morse function on the unit sphere of $V_1$. Define
$$
p([H,V,\rho]) := p\circ i^H_G\big(\xi_H^+(V_1,V_2,g)\big).
$$
By Lemma \ref{lem_reduce_independent_of_g}, Lemma \ref{lem_reduce_independent_of_order}, and the induction hypothesis, the definition of $p$ does not depend on the choice of the decomposition of $V$ or the choice of the function $g$, and $p$ restricts to zero on $\Bif_G^{(m(H))}$.
\end{proof}

\begin{proof}[Proof of Theorem \ref{thm_finite_dim_Cerf_into}]
	Recall that we view $\Conj(G)$ as a subset of $\clR_G$ by identifying $[H]\in \Conj(G)$ with the element in $\clR_G$ represented by the zero representation of $H$. 
	Also recall that in the statement of Theorem \ref{thm_finite_dim_Cerf_into}, $\eta$ denotes a (set-theoretic) map from $\clR_G$ to $\bZ\Conj(G)$ such that $\eta(\sigma) = \sigma$ for all $\sigma \in \Conj(G)$. We need to show that the map $\eta$ described by Theorem \ref{thm_finite_dim_Cerf_into} is existent and unique.
	
By Theorem \ref{thm_equivairant_cerf_finite_dim} and the constructions of Section \ref{subsec_local_bif_model_finite_dim}, 
the sum
$$
\sum_{[p]\in\Crit(f)}\eta(\ind p)
$$
is independent of the function $f$ for all closed $G$--manifolds $M$, if and only if $\eta$ is identically zero on $\Bif_G$. Therefore the desired result is an immediate consequence of Theorem \ref{thm_equivariant_euler_finite_abstract}.
\end{proof}

\begin{Corollary}
\label{cor_equivariant_euler_finite}
Every map $$w:\clR_G^{(0)}\to \bZ$$
can be uniquely extended to a homomorphism
 $$w:\bZ\clR_G\to \bZ,$$ such that 
 $w=0$ on $\Bif_G$. 
 \qed
\end{Corollary}

\begin{Addendum}\label{add:universal}
Our construction in Theorem \ref{thm_finite_dim_Cerf_into} can be viewed as a Morse-theoretic realization of the \emph{universal equivariant Euler characteristic} of $G$-manifolds defined in \cite[Chapter IV, 1]{tom-Dieck}. Such notion is more generally defined for pointed $G$-CW complexes. 

An additive invariant on the category of pointed $G$-CW complexes is defined to be a pair $(b, \mathbf{B})$ where $\mathbf{B}$ is an abelian group and $b$ assigns each pointed pointed $G$-CW complex with an element in $\mathbf{B}$ such that:
\begin{enumerate}
	\item if $X$ and $Y$ are pointed $G$-CW complexes which are pointed $G$-homotopy equivalent, then $b(X) = b(Y)$;
	\item if $A$ is a pointed $G$-CW subcomplex of $X$, then the relation $b(A) - b(X) + b(X/A) = 0$ holds in $\mathbf{B}$.
\end{enumerate}
An additive invariant $(u, \mathbf{U})$ is called \emph{universal} if for every additive invariant $(b, \mathbf{B})$, there is a group homomorphism $\phi: \mathbf{U} \to \mathbf{B}$ such that $b = \phi \circ u$. It turns out that for a compact Lie group $G$, the universal additive invariant always exists, for which $\mathbf{U} = \mathbf{U}(G)$ is exactly the freely abelian group generated by conjugacy classes of closed subgroups of $G$; for any closed subgroup $H \subset G$ and pointed $G$-CW complex $X$, the coefficient of $u$ on the coordinate given by $[H]$ is given by the ordinary Euler characteristic of $X^H / NH$ minus $1$. Corollary \ref{cor_compute_equivariant_Euler} below can be used to show that the invariant defined by Theorem \ref{thm_finite_dim_Cerf_into} is equal to the universal equivariant Euler characteristic defined in \cite{tom-Dieck} for smooth closed manifolds.

\end{Addendum}

\subsection{Computations}
\label{subset_finite_dim_Cerf_examples}
This subsection describes a method that one can use to compute
$$
\sum_{[p]\in\Crit(f)}\eta(\ind p) \in \bZ\Conj(G)
$$
from Theorem \ref{thm_finite_dim_Cerf_into}. By Theorem  \ref{thm_finite_dim_Cerf_into}, this sum does not depend on the choice of the function $f$, therefore it can be regarded as a topological invariant of the closed $G$--manifold $M$. Addendum \ref{add:universal} discussed the algebro-topological interpretation of this invariant.

Results proved in the current subsection (Section \ref{subset_finite_dim_Cerf_examples}) will not be used in the proof of Theorem \ref{thm_existence_casson_intro}.

We start with the following lemma.

\begin{Lemma}
	Suppose $\eta$ is given by Theorem \ref{thm_finite_dim_Cerf_into}, and let $\bR$ denote the trivial 1-dimensional representation of $G$. Then for each $\sigma \in \clR_G$, we have
	$$
	\eta(\sigma \oplus \bR)  = - \eta(\sigma).
	$$
\end{Lemma}

\begin{proof}
By the definition of $\Bif_G$ and the proof of Theorem \ref{thm_finite_dim_Cerf_into}, we have 
$$
(\sigma\oplus \bR + \sigma) \in \Bif_G = \ker(\eta),
$$
and hence the lemma is proved.
\end{proof}

\begin{Corollary}
	Suppose $\sigma\in \Conj(G)$, then 
	$$\eta(\sigma\oplus (\oplus^n\bR)) = (-1)^n \sigma.$$
\end{Corollary}

\begin{proof}
By the definition of $\eta$, we have $\eta(\sigma) = \sigma$. Hence the result follows from the previous lemma.
\end{proof}

\begin{remark}
	Recall that $\Conj(G)$ is identified with a subset of $\clR_G$ via the zero representations. Therefore, for $\sigma = [H] \in \Conj(G)$, the element $\sigma\oplus (\oplus^n\bR)$ is represented by the $n$ dimensional trivial representation of $H$.
\end{remark}

Recall that if $f$ is a $G$--Morse function on $M$ and $p$ is a critical point of $f$, we use $\Hess^-_p f$ to denote the span of the negative eigenspaces of $\Hess f|_p$ (see Definition \ref{def_equi_index_finite_dim}). 
\begin{Corollary}
	\label{cor_compute_equivariant_Euler}
	Suppose $f$ is a $G$--Morse function on $M$ such that for every critical point $p$ of $f$, the space $\Hess^-_p f$ is a trivial representation of $\Stab(p)$. Then 
	$$
	\sum_{[p]\in\Crit(f)}\eta(\ind p) = \sum_{[p]\in \Crit(f)} (-1)^{\dim \Hess^-_p f} \cdot [\Stab(p)].
	\phantom\qedhere\makeatletter\displaymath@qed
	$$
\end{Corollary}

\begin{remark}
	By the equivariant triangulation theorem for smooth $G$--manifolds \cite{illman1983equivariant} and the correspondence between $G$--equivariant cellular structures and $G$--Morse functions (see \cite[Section 4]{wasserman1969equivariant}), we know that
	the required function $f$ in Corollary \ref{cor_compute_equivariant_Euler}  exists for all closed $G$--manifolds $M$.
\end{remark}

\subsection{Extending $\clR_G$ over $\bR$}
\label{subsec_extend_ind_over_R}
In the gauge-theoretic setting, the space spanned by the negative eigenvalues of the Hessian is always infinite-dimensional, and one has to use the spectral flow to define an analogue of the equivariant index. However, the spectral flow is in general not gauge invariant, and one way to cancel the gauge ambiguity is to add the spectral flow by a term given by the Chern-Simons functional. Since the Chern-Simons functional takes values in $\bR$, we need to extend the set $\clR_G$ to a space over $\bR$.

Notice that the definition of $\clR_G$ can be reformulated as follows.
Suppose $H$ is a compact Lie group, let $\clR(H)$ be the representation ring of $H$. In the following, we will only use the abelian group structure of $\clR(H)$. Let $\clR^+(H)\subset \clR(H)$ be the subset of elements given by $[H,V,\rho]-[H,0]$, where $[H,0]$ is the isomorphism class of the zero representation of $H$. 

Consider the disjoint union  $\sqcup_H\, \big(\clR(H)\big)$, where $H$ runs over all closed subgroups of $G$. The group $G$ acts on $\sqcup_H\, \big(\clR(H)\big)$ by conjugation as follows. Suppose 
	$$\rho:H\to\Hom(V,V)$$ 
	is a representation of $H$. Let $g\in G$. Define the action of $g$ on $[H,V,\rho]$ to be the isomorphism class of $(gHg^{-1}, V, g(\rho))$, where $g(\rho)$ is given by
	\begin{align*}
	g(\rho): gHg^{-1} &\to \Hom(V,V) \\
			h &\mapsto \rho(g^{-1}hg).
	\end{align*}
	Then the conjugation action of $g$ preserves $\sqcup_H\, \big(\clR^+(H)\big)$, and it extends linearly as maps from $\clR(H)$ to $\clR(gHg^{-1})$. The set $\clR_G$ equals the quotient set of $\sqcup_H\, \big(\clR^+(H)\big)$ by the conjugation action of $G$.
\begin{Definition}
\label{def_tilde_R_G}
	Define $\widetilde{\clR}_G$ to be the quotient set of $\sqcup_H\clR(H)\otimes \bR$, where $H$ runs over all closed subgroups of $G$, by the conjugation action of $G$. 
	
	Let $\widetilde{\clR}_G^{(0)}$ be the subset of $\widetilde{\clR}_G$ consisting of the elements represented  by 
	$$a_1\cdot[H,V_1,\rho_1] +\cdots +a_k\cdot [H,V_k,\rho_k ]\in\clR(H)\otimes\bR,$$ where $H$ is a closed subgroup of $G$, and  $(V_i,\rho_i)$ are non-isomorphic irreducible representations of $H$, and $a_i\in[0,1)$. 
	
	Suppose $H$ is a closed subgroup of $G$, let $\widetilde{\clR}_G([H])$ be the image of  $\sqcup_H\clR(H)\otimes \bR$ in $\widetilde{\clR}_G$.
\end{Definition}

The conjugation action of $G$ is trivial on $\clR(G)\otimes \bR$. Therefore similar to Lemma \ref{lem_defn_direct_sum}, we have a well-defined direct sum operation
$$
\oplus : \widetilde{\clR}_G\times (\clR(G)\otimes \bR) \to \widetilde{\clR}_G
$$
given as follows. Suppose $\sigma\in \widetilde{\clR}_G$ is given by 
$$a_1\cdot[H,V_1,\rho_1] +\cdots +a_k\cdot [H,V_k,\rho_k ]\in\clR(H)\otimes\bR,$$
and suppose $\tau \in \clR(G)\otimes \bR$ be given by 
$$b_1\cdot[G,W_1,\eta_1] +\cdots +b_s\cdot [G,W_s,\eta_s ]\in\clR(G)\otimes\bR,$$
then $\sigma\oplus \tau$ is defined to be the element represented by 
$$
\sum_{i=1}^k\sum_{j=1}^s a_ib_j\cdot [H,V_i\oplus W_j,\rho_i\oplus (\eta_j|_H)]\in\clR(H)\otimes\bR.
$$

\begin{Definition}
 Suppose $H$ is a closed subgroup of $G$, and $(V,\rho_V)$ is a non-trivial irreducible orthogonal representation of $H$. Let $g$ be an $H$--Morse function on the unit sphere of $V$, and let $\widetilde{V}\in \clR(H)\otimes \bR$. 
 Define 
$$
\xi_H(V,\widetilde{V},g):=\xi_H(V,0,g) \oplus \widetilde{V}.
$$
\end{Definition}

\begin{Definition}
	Let $\widetilde{\Bif}_G$ be the subgroup of $\bZ\widetilde{\clR}_G$ generated by all the elements given by $i^H_G\big(\xi_H(V,\widetilde{V},g)\big)$ and $\sigma+\sigma\oplus \bR$, where $\widetilde{V}\in \clR(H)\otimes \bR$, and $\sigma\in \widetilde{\clR}_G$.
\end{Definition}

The proof of Theorem \ref{thm_equivariant_euler_finite_abstract} can be easily modified to prove the following result.
\begin{Proposition}
\label{prop_quotient_by_tilde_Bif}
The composition of the homomorphisms 
$$\widetilde \Phi: \bZ\widetilde\clR_G^{(0)}\xhookrightarrow{} \bZ\widetilde\clR_G \to \bZ\widetilde\clR_G /\widetilde\Bif_G$$
is an isomorphism.
\qed
\end{Proposition}

\section{Holonomy perturbations and transversality}
\label{sec_holonomy_perturbation}
The rest of this paper generalizes the results in Section \ref{sec_equivariant_Cerf} to the gauge-theoretic setting. This section establishes the transversality properties that are analogous to Lemma \ref{lem_one_degeneracy_codim} and  Lemma \ref{lem_two_degeneracy_codim}. 

\subsection{Preliminaries}
Let $Y$ be a smooth, oriented, closed $3$--manifold. Let $G$ be a compact, simply-connected simple Lie group, and let $\frg$ be the Lie algebra of $G$. By Cartan's theorem  \cite{cartan1952topologie,bott1956application}, we have 
\begin{equation}
\label{eqn_pi_i(G)}
\pi_1(G)=\pi_2(G)=0,~\pi_3(G)\cong \bZ,
\end{equation}
 therefore every principal $G$--bundle over $Y$ is trivial. Let 
\begin{equation}
\label{eqn_def_P_trivial}
P=Y\times G
\end{equation}
 be the trivial bundle. We will abuse notation and also use $\frg$ to denote the trivial $\frg$--bundle over $Y$ when there is no source of confusion.

Fix an integer $k\ge 2$. Let $\clC$ be the space of $L_k^2$--connections over $P$, then $\clC$ is an affine space over $L_k^2(T^*Y\otimes \frg).$ 
Let $\clG$ be the $L_{k+1}^2$--gauge group of $P$, then $\clG$ is identified with the set of $L_{k+1}^2$--maps from $Y$ to $G$.  

Let $\theta$ be the trivial connection associated to the product structure \eqref{eqn_def_P_trivial}, then
$$\clC=\theta+ L_k^2(T^*Y\otimes \frg),$$
and the action of $g\in \clG$ on $\theta+b\in \clC$ is given by 
$$
g(\theta+b) = \theta + \Ad_g(b) - g^{-1} dg.
$$
By the Sobolev multiplication theorem, $\clG$ is a Banach Lie group that acts smoothly on $\clC$. 

By \eqref{eqn_pi_i(G)}, we have 
\begin{equation}
\label{eqn_pi_0_clG}
	\pi_0(\clG)\cong \bZ,
\end{equation}
where the group structure on $\pi_0(\clG)$ is induced from the group structure of $\clG$. Since $Y$ is oriented, one can fix a canonical choice of the isomorphism \eqref{eqn_pi_0_clG}. For $g\in \clG$, we will use $\deg g$ to denote the image of $g$ in $\bZ$ under this isomorphism. 

Fix a Riemannian metric on $Y$, and define the inner product on $\frg$ by the Killing form.
The \emph{Chern-Simons functional} on $\clC$ is given by
\begin{equation}
\label{eqn_def_chern_simons}
	\CS(\theta+ b):= \frac12 \langle *d_{\theta}b,b\rangle_{L^2} + \frac13 \langle *[b\wedge b],b\rangle_{L^2}.
\end{equation} 
 $\CS(\theta+ b)$ is independent of the choice of the Riemannian metric. Let $\grad \CS$ be the formal gradient of $\CS$, then we have
\begin{align*}
\CS(\theta)&=0, \\
(\grad \CS)(B) &= *F_B\mbox{ for all }B\in\clC. 
\end{align*}

The Chern-Simons functional is only invariant under the identity component of $\clG$. In general, there exists a non-zero constant $c$ depending on $G$ such that 
$$
\CS(g(B))-\CS(B) = c\,\deg(g)
$$
for all $B\in \clC$ and $g\in \clG$.
Therefore, the Chern-Simons functional defines a $\clG$--invariant map from $\clC$ to $\bR/c\bZ$.
 
 We make the following definition analogous to Definition \ref{def_stab_orb_finite}.
 
 \begin{Definition}
 	Suppose $B\in \clC$. Define 
 	\begin{align*}
 	 	\Stab(B) &:= \{g\in\clG| g(B)=B \}\\
 	 	\Orb(B) &:= \{g(B) \in \clC | g\in\clG \}.
 	\end{align*}
 \end{Definition}
 
 Notice that although $\clG$ is an infinite dimensional group, the stabilizer $\Stab(B)$ is always finite dimensional for any $B \in \clC$. In fact, let $y_0\in Y$ be a fixed point, and let 
$$\clG_{y_0}:= \{g\in \clG | g = \id \mbox{ on } y_0 \},$$ 
then the action of $\clG_{y_0}$ on $\clC$ is free. Since  $\clG_{y_0}$ is a normal subgroup of $\clG$ and 
$$\clG/\clG_{y_0}\cong G,$$
the stabilizer group $\Stab(B)$ maps isomorphically to a closed subgroup of $G$ by restricting to $y_0$. 
 
 The following is a standard result in gauge theory and is analogous to Lemma \ref{lem_local_slice}. The reader may refer to, for example, \cite[Section 4.2.1]{donaldson1990geometry}, for more details.
 
 \begin{Proposition}[Slice theorem]
 \label{prop_slice_thm}
 	For $B\in \clC$, let 
 	$$
 	S_{B,\epsilon}:=\{B+b\,|\,d_B^*b=0,\|b\|_{L_{k,B}^2}< \epsilon\},
 	$$
 	where 
 	$$
 	\|b\|_{L_{k,B}^2}^2 := \sum_{i=0}^k\|\nabla^i_B b\|_{L^2}^2.
 	$$
 	Then for each $B$,  there exists $\epsilon>0$ depending on $B$, such that
 	\begin{align*}
 	 	\clG\times_{\Stab(B)} S_{B,\epsilon} & \to \clC \\
 	 	 [g,B+b] & \mapsto g(B+b)
 	\end{align*}
 	is a diffeomorphism onto an open neighborhood of $\Orb(B)$. 
 \end{Proposition}

\subsection{Holonomy perturbations}

Holonomy perturbations are widely used in gauge theory as a family of perturbations of the Chern-Simons functional. They were first used by Donaldson  \cite{donaldson1987orientation} in the study of $4$--dimensional Yang-Mills theory and Floer  \cite{floer1988instanton} in the construction of instanton Floer homology for $3$--manifolds. For our purpose, we briefly review the construction following the notation of \cite{kronheimer2011knot}.

\begin{Definition}
\label{def_cylinder_datum}
We regard the circle $S^1$ as $\bR/\bZ$, and let $D^2$ be the open unit disk in the plane.
A \emph{cylinder datum} is a tuple $(q_1,\cdots,q_m,\mu,h)$ with $m\in \bZ^+$ that satisfies the following conditions.
\begin{enumerate}
\item $q_i: S^1 \times D^2 \rightarrow Y$ is a smooth immersion for $i=1,\cdots,m$;
\item there exists $\epsilon > 0$ such that $q_1, \dots, q_m$ coincide on $(-\epsilon, \epsilon) \times D^2$,
\item $\mu$ is a non-negative, smooth, compactly supported 2-form on $D^2$, such that 
$$\int_{D^2} \mu = 1;$$
\item $h:G^m\to \bR$ is a smooth function that is invariant under the diagonal action of $G$ by conjugations. 
\end{enumerate}
\end{Definition}

Suppose $B\in \clC$, let ${\bf q} = (q_1,\cdots,q_m,\mu,h)$ be a cylinder datum, then for $z\in D^2$ and $i=1,\cdots,m$, the holonomy of $B$ at $q_i(0,z)$ along $q_i(S^1\times\{z\})$ defines a map on the fiber $P|_{q_i(0,z)}$.   Under the trivialization of $P$ given by \eqref{eqn_def_P_trivial}, this map is given by a left multiplication of an element in $G$, and we use $\hol_{q_i,z}(B)\in G$ to denote this element. Therefore we have a map
\begin{align*}
	\hol_{\bf q}(B) : D^2 &\to G^m \\
	z &\mapsto \big(\hol_{q_1,z}(B),\cdots,\hol_{q_m,z}(B)\big).
\end{align*}
The \emph{cylinder function} associated to $\bf q$ is defined to be
\begin{align*}
f_{{\bf q}}:\clC &\to \bR\\
B &\mapsto \int_{D^2} h\big(\hol_{\bf q}(B)\big) \,\mu. 
\end{align*}
By definition, $f_{\bf q}$ is a $\clG$--invariant function on $\clC$.

Let $\mathcal{T}$ be the tangent bundle of $\clC$. 
For $B\in \clC$, recall that the \emph{formal gradient} of $f_{{\bf q}}$ at $B$ is defined to be the unique vector $$b\in \mathcal{T}|_B = L_k^2(T^*Y\otimes\frg)$$
 such that for every $b'\in L_k^2(T^*Y\otimes\frg)$, we have
 $$
 \frac{d}{dt}f_{{\bf q}}(B+tb') = \langle b,b'\rangle _{L^2}.
 $$
 By \cite[Proposition 3.5 (i)]{kronheimer2011knot}, the formal gradient of $f_{{\bf q}}$ exists and is a smooth section of $\mathcal{T}$. We use $\grad f_{{\bf q}}$ to denote the formal gradient of $f_{{\bf q}}$.

Let $\{{\bf q}_i\}_{i\in\mathbb{N}}$ be a fixed sequence of  cylinder data, such that for every cylinder datum
$$
{\bf q} = (q_1,\cdots,q_m,\mu,h),
$$ 
there is a subsequence of $\{{\bf q}_i\}_{i\in\mathbb{N}}$ so that it consists of elements with the same value of $m$ and converges to $\bf q$ in $C^\infty$. By the discussion before \cite[Definition 3.6]{kronheimer2011knot}, there exists a sequence $\{C_i\}_{i\in\mathbb{N}}$ of positive real numbers, such that for every sequence $\{a_i\}_{i\in\mathbb{N}}\subset \bR$ satisfying
$$
\sum |a_i| \, C_i\ <+\infty,
$$
the series $\sum a_i f_{{\bf q}_i}$
converges to a smooth function on $\clC$, and the series
$$
\sum a_i \grad f_{{\bf q}_i}
$$
converges to a smooth section of $\mathcal{T}$ that equals the formal gradient of $\sum a_i f_{{\bf q}_i}$.

\begin{Definition}
\label{def_hol_perturbation_Banach}
	Let $\clP$ be the Banach space of sequences of real numbers $\{a_i\}_{i\in\mathbb{N}}$ such that $\sum |a_i| \, C_i\ <+\infty.$ For $\pi = \{a_i\}_{i\in\mathbb{N}}\in\clP$, define its norm in $\clP$  by
	$$
	\|\pi\|_{\clP} := \sum |a_i| \, C_i.
	$$
	Define 
	\begin{align*}
	f_\pi &:= \sum a_i f_{{\bf q}_i}, \\
	V_\pi &:= \sum a_i \grad f_{{\bf q}_i}.
	\end{align*}
\end{Definition}

For $\pi\in \clP$, let $DV_\pi$ be the derivative of $V_\pi$. Then for $B\in \clC$, the derivative $DV_\pi(B)$ defines a linear endomorphism on 
$$\mathcal{T}|_B=L_k^2(T^*Y\otimes \frg).$$
Let 
$$
\sym (\mathcal{T}|_B)
$$
be the Banach space of linear endomorphism on $\mathcal{T}|_B$ that are bounded with respect to the $L_k^2$--norm and symmetric with respect to the $L^2$--norm, and define the norm on $\sym (\mathcal{T}|_B)$ as the $L_k^2$--operator norm. Then $\sym (\mathcal{T}|_B)$ for $B\in \clC$ form a (trivial) Banach vector bundle $\sym(\mathcal{T})$ over $\clC$. By \cite[Proposition 3.7 (ii)]{kronheimer2011knot},
the map $B\mapsto DV_\pi(B)$ is a smooth section of  $\sym(\mathcal{T})$.

Fix a base point $y_0$ on $Y$, and let $\gamma$ be a smooth arc from $y_0$ to $y\in Y$. For  $B\in \clC$, the holonomy of $B$ along $\gamma$ is a map from $P|_{y_0}$ to $P|_{y}$. Since $P$ is the trivial bundle, the map is given by the left multiplication of an element in $G$. We  use
$$\hol_\gamma(B)\in G$$ 
to denote the corresponding element. Suppose $b\in L_k^2(T^*Y\otimes \frg)$, then 
$$
\frac{d}{dt}\Big|_{t=0} \hol_\gamma(B+tb)
$$
defines a tangent vector of $G$ at $\hol_\gamma(B)\in G$. 

The rest of this subsection proves that the holonomy perturbations are sufficiently flexible so that it can realize any $\clG$--invariant jet on any finite dimensional subspace of $\clC$. The precise statement is given by Proposition \ref{prop_abundance_2_pts}. This property will be used in the transversality argument in Section \ref{subsec_equi_trans_gauge}.

Recall that $y_0\in Y$ denotes a fixed base point for the rest of this subsection.

\begin{Lemma}
\label{lem_hol_distinguish_connection}
	Let $B_1,B_2\in\clC$. Suppose there exists $u\in G$, such that 
	\begin{equation}
	\label{eqn_hol_match}
	\hol_\gamma(B_1)=u\hol_\gamma(B_2)u^{-1}
	\end{equation}
	 for all immersed loops $\gamma$ that start and end at $y_0$. Then there exists $g\in \clG$, such that $g(B_1)=B_2$.
\end{Lemma}

\begin{proof}
	Take $g_0\in\clG$ such that $g_0|_{y_0}=u$, then  
	$$
	\hol_\gamma\big(g_0(B_1)\big) = u\hol_\gamma(B_1)u^{-1}
	$$
	for all loops $\gamma$ based at $y_0$.
	Therefore, by replacing $B_1$ with $g_0(B_1)$, we may assume without loss of generality that $u=\id\in G$. 
	
	Define a gauge transformation $g$ of $P$ as follows. For $y\in Y$, take an arc $\gamma$ from $y_0$ to $y$, and define the value of $g$ at $y$ to be 
	$$
	g|_y = \hol_\gamma(B_2)\cdot \hol_\gamma(B_1)^{-1}.
	$$
	Since \eqref{eqn_hol_match} holds with $u=\id$, the definition of $g$ is independent of the choice of $\gamma$, and we have $g(B_1)=B_2$. Since $B_1$ and $B_2$ are both $L_k^2$--connections, the  standard regularity argument implies that $g$ is a $L_{k+1}^2$--gauge transformation, and hence $g\in \clG$.
\end{proof}

The following two lemmas prove an infinitesimal version of Lemma \ref{lem_hol_distinguish_connection}.

\begin{Lemma}
\label{lem_infinitesimal_gauge_transform_base_id}
	Let $B\in \clC$. 
	Suppose $b\in L_k^2(T^*Y\otimes \frg)$ satisfies 
	\begin{equation}
	\label{eqn_hol_change_zero}
				\frac{d}{dt}\Big|_{t=0} \hol_\gamma(B+tb)=0  
	\end{equation}
	for all immersed loops $\gamma$ that start and end at $y_0$, then there exists $s\in L_{k+1}^2(\frg)$, such that $s(y_0)=0$ and $d_Bs=b$.
\end{Lemma}

\begin{proof}
	We define the section $s$ as follows. Let $y\in Y$, take a smooth arc $\gamma$ from $y_0$ to $y$, and let $s(y)\in\frg$ be such that
	\begin{equation}
	\label{eqn_def_xi_by_der_hol}
			\frac{d}{dt}\Big|_{t=0}\hol_\gamma(B+tb) = \frac{d}{dt}\Big|_{t=0}\Big(\exp\big(-t\, s(y)\big)\cdot \hol_\gamma(B) \Big).
	\end{equation}
	Note that by definition, $\hol_\gamma(-)$ is an element in $G$, and $s(y)$ is uniquely determined by the above equation because the right translation by $\hol_\gamma(B)\in G$ maps $\frg$ isomorphically to the tangent space of $G$ at $\hol_\gamma(B)$.
	By \eqref{eqn_hol_change_zero}, the value of $s(y)$ does not depend on the choice of the arc $\gamma$.
	 By the Sobolev embedding theorems, $B-\theta$ is continuous, therefore \eqref{eqn_def_xi_by_der_hol} implies that $s$ is in $C^1$. 
	 
	 To compute $d_Bs$, it is convenient to view $G$ as a closed subgroup of $U(N)$ for some $N$ and view elements of $G$ and $\frg$ as matrices. This allows us to simplify notation by carrying out the computation using matrix multiplications instead of the tangent maps of the translation operators.   For each $r\in[0,1]$, let $\gamma_r$ be the path from $\gamma(0)$ to $\gamma(r)$ along $\gamma$, and let $h(r,t)=\hol_{\gamma_r}(B+tb)$. With respect to the standard trivialization of the bundle, define $\beta(r,t)\in \frg$ to be the pairing of the $1$--form $B+tb-\theta$ with the tangent vector of $\gamma$ at $\gamma(r)$, and define $s(r,t)\in \frg$ to be the value of $s$ at $\gamma(r)$. 
	 By the definitions of $\beta$, $h$, and $g$, we have 
	 $$
	 \frac{\partial}{\partial r}h(r,t) + \beta(r,t) h(r,t) = 0,\quad s(r) = -\frac{\partial}{\partial t}h(r,t)\cdot h^{-1}(r,t)|_{t=0}.
	 $$
	 In the following, we will denote partial derivatives using the subscript notation and omit the variables from the notation of functions. So the two equations above are rewritten as $h_r +\beta h=0$ and $s=-h_t h^{-1}$ at $t=0$. Therefore, at $t=0$, we have
	 $$
	 (sh)_r = -h_{tr} = (\beta h)_t,
	 $$
	 which gives
	 $$
	 s_r = \beta_t + \beta h_t h^{-1} -s h_r h^{-1}.
	 $$
	 Since $h_t h^{-1} = -s$, $h_r h^{-1} = -\beta$, the above equation implies 
     $s_r + [\beta,s] = \beta_t$. Note that $\beta_t$ equals the paring of $b$ with the tangent vector of $\gamma$ at $\gamma(r)$. Since this computation holds for all smooth arcs starting at $y_0$, we have $d_Bs = b$. 
	 
	 In conclusion, we have $s\in C^1$ and $d_Bs = b$, and the standard bootstrapping argument implies that $s\in L_{k+1}^2(\frg)$. 
\end{proof}

\begin{Lemma}
\label{lem_infinitesimal_gauge_transform_detected_by_hol}
	Let $B\in \clC$, $\xi\in \frg$. 
	Suppose $b\in L_k^2(T^*Y\otimes \frg)$ satisfies 
	\begin{equation}
		\label{eqn_cond_on_b_infinitesimal_gauge_on_loops}
				\frac{d}{dt}\Big|_{t=0} \hol_\gamma(B+tb)=\frac{d}{dt}\Big|_{t=0} \Big(\exp(t\xi)\cdot \hol_\gamma(B) \cdot \exp(-t\xi)\Big)  	\end{equation}
	for all immersed loops $\gamma$ that start and end at $y_0$, then there exists $s\in L_{k+1}^2(\frg)$, such that $s(y_0)=\xi$, and $d_Bs=b$.
\end{Lemma}

\begin{proof}
	Let $g_t\in \clG$, $t\in(-1,1)$, be a smooth family of gauge transformations such that 
	\[
		g_0=\id  
		\quad \mbox{and} \quad  \frac{d}{dt}\Big|_{t=0}\, g_t(y_0) = \xi.	\]
	Since holonomy is equivariant with respect to gauge transformations, we have
	$$
	\frac{d}{dt} \hol_\gamma(g_t(B))=\frac{d}{dt}\Big(\exp(t\xi)\cdot \hol_\gamma(B) \cdot \exp(-t\xi)\Big).
	$$
	Let 
	$$s_0=\frac{d}{dt}\Big|_{t=0} \, g_t,$$ then $s_0$ is a smooth section of $\frg$. Let
	$$b'= b-d_Bs_0.$$
	 By \eqref{eqn_cond_on_b_infinitesimal_gauge_on_loops} and the fact that the differentials of the operator $\hol_\gamma(-)$ are linear, we have		
	 \[
	 \frac{d}{dt}\Big|_{t=0} \hol_\gamma(B+tb')=0.
	 \]
	  The desired result now follows from Lemma \ref{lem_infinitesimal_gauge_transform_base_id}.
\end{proof}

Lemma  \ref{lem_hol_distinguish_connection} and Lemma \ref{lem_infinitesimal_gauge_transform_detected_by_hol} have the following immediate consequence.

\begin{Lemma}
\label{lem_2_pt_injectivity_finite_loops}
	 Suppose $B_1, B_2\in\clC$ are not gauge equivalent.
	 For $i=1,2$, let 
	$$L_i\subset \mathcal{T}|_{B_i} = L_k^2(T^*Y\otimes\frg)$$
	be a finite dimensional linear space such that
	 $$L_i\cap \ima d_{B_i}=\{0\}.$$ 
	Then there exists $m\in \mathbb{Z}^+$ and 
	$$\vec \gamma = (\gamma_1,\cdots,\gamma_m),$$
	where each $\gamma_j$ is an immersed loop in $Y$ based at $y_0$, such that the map
	\begin{align*}
    \hol_{\vec \gamma}:	\clC  & \to G^m \\
						B &\mapsto \big(\hol_{\gamma_1}(B),\cdots,\hol_{\gamma_m}(B)\big)
	\end{align*}
	satisfies the following conditions:
	\begin{enumerate}
		\item $\hol_{\vec \gamma}(B_1)\neq \hol_{\vec \gamma}(B_2)$ in the quotient set $G^m/G$ (where the action of $G$ on $G^m$ is the diagonal action by conjugations),
		\item for $i=1,2$, let $p_i= \hol_{\vec \gamma}(B_i)$, then the tangent map of $\hol_{\vec \gamma}$ maps $L_i$ injectively into $T_{p_i}G^m/T_{p_i}\Orb(p_i)$.
	\end{enumerate}
\end{Lemma} 

\begin{proof}
	Since $B_1$ and $B_2$ are not gauge equivalent, by Lemma \ref{lem_hol_distinguish_connection}, there exists 
	$$\vec\gamma^{(0)}=(\gamma_1^{(0)},\cdots,\gamma_{m_0}^{(0)}),$$
	 where each $\gamma_j$ is an immersed loop in $Y$ based at $y_0$,  such that $
	 \hol_{\vec\gamma^{(0)}}(B_1)$ and $
	 \hol_{\vec\gamma^{(0)}}(B_2)$ are not conjugate in $G^{m_0}$.
	 
	 By Lemma \ref{lem_infinitesimal_gauge_transform_detected_by_hol}, for each $i=1,2$ and each vector $0\neq v\in L_i$, there exists 
	 $$
	 \vec\gamma^{(i,v)}=(\gamma^{(i,v)}_1,\cdots,\gamma^{(i,v)}_s),
	 $$
	 	where each $\gamma^{(i,v)}_{j}$ is an immersed loop in $Y$ based at $y_0$, 
	 	such that if we write 
	 	$$p_{i,v}:=\hol_{\vec\gamma^{(i,v)}}(B_i)\in G^{s},$$
	 	 then the tangent map of $\hol_{\vec\gamma^{(i,v)}}$ maps $v$ to a non-zero vector in 
	 	$T_{p_i'} G^{m_i}/T_{p_i'}\Orb(p_i)$.
	 	
	 	Since each $L_i$ is finite dimensional, the desired result is proved.
\end{proof}

Lemma \ref{lem_2_pt_injectivity_finite_loops} implies the following property of holonomy perturbations, which will be used in the transversality argument in Section \ref{subsec_equi_trans_gauge}.

\begin{Proposition}
\label{prop_abundance_2_pts}
	Suppose $B_1, B_2\in\clC$ are not gauge equivalent.
	 For $i=1,2$, let 
	$$L_i\subset \mathcal{T}|_{B_i} = L_k^2(T^*Y\otimes\frg)$$
	be finite dimensional linear spaces that are invariant under the action of $\Stab(B_i)$, such that 
	 $$L_i\cap \ima d_{B_i}=\{0\}.$$ 
	 We further assume that $B_1,B_2$ are smooth and that $L_1,L_2$ are spanned by smooth sections of $T^*Y\otimes\frg$.
	 
	 For $r\in \bZ^+$, let $\mathcal{J}_i(r)$ be the linear space of $r$--jets of $\Stab(B_i)$--invariant functions on the affine space $B_i+L_i$ at $B_i$, then every $\pi\in\clP$ defines an element in $\mathcal{J}_i(r)$ by restricting $f_{\pi}$ to $B_i+L_i$. Let
	\begin{equation}
	\label{eqn_restriction_map_jets}
			\Phi: \clP \to \mathcal{J}_1(r_1)\oplus \mathcal{J}_2(r_2) 
	\end{equation}
	be given by the restriction maps. Then $\Phi$ is surjective for all $r_1,r_2$.
\end{Proposition}

\begin{proof}
	Since $\mathcal{J}_1(r_1)\oplus \mathcal{J}_2(r_2) $ is a finite dimensional linear space and \eqref{eqn_restriction_map_jets} is a linear map, we only need to show that the image is dense. 
	
	Suppose $(j_1,j_2)\in \mathcal{J}_1(r_1)\oplus \mathcal{J}_2(r_2) $.
	As in Lemma \ref{lem_2_pt_injectivity_finite_loops}, for $m\in \bZ^+$, let $G$ act on $G^m$ diagonally by taking conjugation at each component. 
	By Lemma \ref{lem_2_pt_injectivity_finite_loops}, there exists $$\vec \gamma = (\gamma_1,\cdots,\gamma_m),$$
	where each $\gamma_j$ is an immersed loop in $Y$ based at $y_0$, such that the map
	\begin{align*}
\hol_{\vec \gamma}: \clC &\to G^m \\
		       B &\mapsto \big(\hol_{\gamma_1}(B),\cdots,\hol_{\gamma_m}(B)\big)
	\end{align*}
satisfies the statement of  Lemma \ref{lem_2_pt_injectivity_finite_loops} with respect to $B_1$, $B_2$, $L_1$, $L_2$. Since $B_1,B_2$ are smooth and that $L_1,L_2$ are spanned by smooth sections of $T^*Y\otimes\frg$, the map $\hol_{\vec \gamma}$ is smooth on the affine spaces $B_i+ L_i$ ($i=1,2$). Therefore, there exists a smooth $G$--invariant function 
	\begin{equation}
	\label{eqn_def_h_in_proof_of_abundance}
		h:G^m\to \bR,
	\end{equation}
 	such that the map
	\begin{align*}
				f_{\vec\gamma,h}:	\clC  & \to \bR \\
						B &\mapsto h\big(\hol_{\gamma_1}(B),\cdots,\hol_{\gamma_m}(B)\big)
	\end{align*}
	restricts to $(j_1,j_2)$ in $\mathcal{J}_1(r_1)\oplus \mathcal{J}_2(r_2) $. For $i=1,\cdots,m$, let the map $q_i$ be given by
	\begin{align*}
	q_i:S^1\times D^2 &\to Y \\
	      (s,z) &\mapsto \gamma_i(s),
	\end{align*}
	 let $\mu$ be an arbitrary smooth, non-negative, compactly supported 2-form on $D^2$ that integrates to $1$, and let $h$ be given by \eqref{eqn_def_h_in_proof_of_abundance}. Recall that  $\{{\bf q}_i\}_{i\in\mathbb{N}}$ is the sequence of  cylinder data used in Definition \ref{def_hol_perturbation_Banach}. Then by definition, there exists a subsequence of $\{{\bf q}_i\}_{i\in\mathbb{N}}$ that converges to $(q_1,\cdots,q_m,\mu,h)$ in $C^\infty$. Let $\{{\bf q}_{n_i}\}_{i\in\mathbb{N}}$ be such a subsequence, then we have
	 $$
	 \lim_{i\to\infty} \Phi(f_{{\bf q}_{n_i}}) = (j_1,j_2),
	 $$
	 and hence the result is proved.
\end{proof}

\subsection{Hessians of perturbed flat connections}
This subsection defines the Hessians of perturbed Chern-Simons functionals at critical points.

Let $\pi\in\clP$, let $\grad (\CS+f_\pi) (B)$ be the formal gradient of $\CS+f_\pi$ at $B$. Then
$$\grad (\CS+f_\pi) (B)= *F_B+V_\pi(B),$$
and the derivative of $\grad (\CS+f_\pi) $ at $B$ is given by the operator $*d_B + DV_\pi(B)$. 

\begin{Definition}
	A connection $B$ is called $\pi$--flat, if 
	\begin{equation}
	\label{eqn_pi_flatness}
			*F_B+V_\pi(B)=0
	\end{equation}
\end{Definition}

\begin{Lemma}
\label{lem_decomp_K_when_flat}
	If $B$ is $\pi$--flat, then the operator $*d_B + DV_\pi(B)$ is identically zero on $\ima d_B\cap L_k^2(T^*Y\otimes \frg)$.
\end{Lemma}
\begin{proof}
	Notice that $\ima d_B\cap L^2_k(T^*Y\otimes \frg)$ is the tangent space of $\Orb(B)$. Suppose $s\in \ima d_B\cap L_k^2(T^*Y\otimes \frg)$, let $B(t)$ be a smooth curve in $\Orb(B)$ such that 
	$$
	\frac{d}{dt} B(t) = s \quad \mbox{at } t=0.
	$$
	Then by the gauge invariance of Equation \eqref{eqn_pi_flatness}, we have
	$$
	*F_{B(t)} + V_\pi\big(B(t)\big)=0 \quad \mbox{for all } t,
	$$
	therefore 
	$$
	d_B s + DV_\pi(B)(s) = \frac{d}{dt}\Big|_{t=0} \Big(*F_{B(t)} + V_\pi\big(B(t)\big)\Big)=0,
	$$
	and hence the lemma is proved.
\end{proof}

Suppose $B\in \clC$ and $\pi\in\clP$.
Define the operator
$$
K_{B,\pi}: L_{k}^2(\frg)\oplus L_k^2(T^*Y\otimes \frg) \to L_{k-1}^2(\frg)\oplus L_{k-1}^2(T^*Y\otimes \frg) 
$$
by 
\begin{equation}
\label{eqn_def_K}
	K_{B,\pi}(\xi,b):=( d_B^*b, d_B\xi + *d_B b+ DV_\pi(B) (b) ).
\end{equation}
Then $K_{B,\pi}$ is self-adjoint and elliptic, therefore it is Fredholm with index zero, and its spectrum is discrete and is contained in $\bR$.

The domain of $K_{B,\pi}$ can be orthogonally decomposed as 
\begin{equation}
\label{eqn_decomp_domain_K}
	L_{k}^2(\frg)\oplus\big( \ima d_B \cap L_k^2(T^*Y\otimes \frg)\big)  \oplus \big(\ker d_B^* \cap L_k^2(T^*Y\otimes \frg)\big),
\end{equation}
and range of $K_{B,\pi}$ can be orthogonally decomposed as
\begin{equation}
\label{eqn_decomp_range_K}
	L_{k-1}^2(\frg)\oplus\big( \ima d_B \cap L_{k-1}^2(T^*Y\otimes \frg)\big)  \oplus \big(\ker d_B^* \cap L_{k-1}^2(T^*Y\otimes \frg)\big).
\end{equation}

If $B$ is $\pi$-flat, then by Lemma \ref{lem_decomp_K_when_flat} and the fact that $K_{B,\pi}$ is self-adjoint, the operator $K_{B,\pi}$ is given by 
\begin{equation}
\label{eqn_decomp_K_when_flat}
	\begin{pmatrix}
0 & d_B^* & 0 \\
d_B & 0 & 0 \\
0 & 0 & *d_B+DV_\pi(B)	
\end{pmatrix}
\end{equation}
under the decompositions \eqref{eqn_decomp_domain_K} and \eqref{eqn_decomp_range_K}.

\begin{Definition}
Suppose $B$ is $\pi$--flat,
	define $\Hess_{B,\pi}$ to be the operator
$$
\Hess_{B,\pi}: \ker d^*_B\cap L_k^2(T^*Y\otimes \frg) \to \ker d^*_B\cap L_{k-1}^2(T^*Y\otimes \frg)
$$
given by 
$$\Hess_{B,\pi} := *d_B+DV_\pi(B).$$ 
\end{Definition}

By definition, $\Hess_{B,\pi}$ is a self-adjoint Fredholm operator with index zero.

\begin{Definition}
\label{def_non-deg_perturbed_connection}
	We say that a $\pi$--flat connection $B$ \emph{non-degenerate}, if $\Hess_{B,\pi}$ is an isomorphism. We say that $\pi\in\clP$ is  \emph{non-degenerate}, if all critical points of $\CS+f_{\pi}$ are non-degenerate as $\pi$--flat connections.
\end{Definition}

\subsection{Equivariant transversality}
\label{subsec_equi_trans_gauge}
This subsection establishes the transversality properties of holonomy perturbations that are analogous to Lemma \ref{lem_one_degeneracy_codim} and Lemma \ref{lem_two_degeneracy_codim}.

Recall that for each $y\in Y$, the restriction to $y$ gives a map
$$
r_y: \Stab(B)\to G
$$
that sends $\Stab(B)$ to a closed subgroup of $G$. Suppose $y_1,y_2\in Y$, let $\gamma$ be an arc from $y_1$ to $y_2$. Then 
$$
r_{y_2} (g) = \hol_{\gamma}(B) \cdot r_{y_1}(g)\cdot \hol_{\gamma}(B)^{-1} \mbox{ for all } g\in\Stab(B).
$$
As a consequence, every finite dimensional representation of $\Stab(B)$ defines an element of $\clR_G$. 

Suppose $\pi\in\clP$ and let $\Orb(B)$ be a critical orbit of $\CS+f_\pi$, then $\ker \Hess_{B,\pi}$ defines an element in $\clR_G$  as a $\Stab(B)$--representation. It is straightforward to verify that this element is invariant under gauge transformations of $B$.

\begin{Definition}
	A linear map $P$ defined on a linear subspace of $L^2(T^*Y\otimes \frg)$ is called \emph{symmetric}, if $\langle Px,y\rangle = \langle x,Py\rangle $ for all $x,y$ in the domain of $P$.
\end{Definition}

Let $\mathcal{H}(B)$ be the Banach space of $\Stab(B)$--invariant, symmetric, bounded, linear operators from  the Banach space
$$\ker d_B^*\cap L_k^2(T^*Y\otimes\frg)$$
to the Banach space
$$\ker d_B^*\cap L_{k-1}^2(T^*Y\otimes\frg).$$
Define the norm on $\clH(B)$ as the operator norm, then $\clH(B)$ becomes a Banach space.  
 For $\sigma\in \clR_G$, let $\mathcal{H}_\sigma (B)\subset \mathcal{H}(B)$ consist of the elements that are Fredholm with index $0$ such that their kernels represent $\sigma$. 
 
Recall that the dimension $d(\sigma)$ is defined by Definition \ref{def_d(sigma)}. Suppose $s\in \mathcal{H}_\sigma (B)$, recall that we use $\sym_{\Stab(B)}\ker s$ to denote the space of $\Stab(B)$--equivariant symmetric maps on $\ker s$.

\begin{Lemma}
\label{lem_clH_sigma_submanifold}
	$\mathcal{H}_\sigma (B)$ is a submanifold of $\mathcal{H}(B)$ with codimension $d(\sigma)$. 
Moreover,  suppose $s\in \mathcal{H}_\sigma (B)$, let $\Pi$ be the $L^2$--orthogonal projection from 
  $$\ker d_B^*\cap L_{k-1}^2(T^*Y\otimes\frg)$$ to $\ker s$. Suppose  $L\subset \clH(B)$ is a linear subspace. Then $s+L$ is transverse to $\mathcal{H}_\sigma (B)$ if and only if the linear map
	\begin{align*}
	 L &\to \sym_{\Stab(B)}\ker s \\
	 l &\mapsto \Pi\circ (l|_{\ker s})
	\end{align*}
	is surjective.
\end{Lemma}

\begin{proof}
The proof is similar to the argument of Lemma \ref{lem_sym_sigma_submanifold_finite_dim}.
Suppose $s\in \mathcal{H}_\sigma (B)$. For $i=k,k-1$, let $V^{(i)}$ be the $L^2$--orthogonal complement of $\ker s$ in $\ker d_B^*\cap L_i^2(T^*Y\otimes\frg)$. Then the domain and range of $s$ decompose as $\ker s\oplus V^{(k)}$ and $\ker s\oplus V^{(k-1)}$ respectively. By the assumptions, the map $s$ restricts to an isomorphism from $V^{(k)}$ to $V^{(k-1)}$.

 There exists an open neighborhood $U$ of $s$ in $\clH(B)$ such that all $s'\in U$ are Fredholm and have index zero. Suppose $s'$ decomposes as a map from $\ker s\oplus V^{(k)}$ to $\ker s\oplus V^{(k-1)}$ as
 $$
 s' = \begin{pmatrix}
 	S_{11} & S_{12} \\
 	S_{21} & S_{22}
 \end{pmatrix},
 $$
 then after shrinking $U$ if necessary, the map
 $$
 S_{22} : V^{(k)} \to V^{(k-1)}
 $$
 is always invertible. By the same computation as \eqref{eqn_elementary_matrix_trans_S22_finite_dim}, we have $s'\in \mathcal{H}_\sigma (B)$ if and only if 
 $$
S_{11}-S_{12}\circ S_{22}^{-1}\circ S_{21} = 0.
$$
Therefore $\mathcal{H}_\sigma(B)$ is a submanifold of $\mathcal{H}(B)$ near $s$, and its tangent space at $s$ is given by $S_{11}=0$. Hence the lemma is proved.
\end{proof}

\begin{Definition}
Suppose $\sigma\in \clR_G$
	Let 
	\begin{equation*}
		 \clP_\sigma:=\{\pi\in\clP\,|\,\exists B\in \clC \mbox{ such that } *F_B+V_\pi=0, \mbox{ and }
	 \ker \Hess_{B,\pi} \mbox{ represents }  \sigma\}.
	\end{equation*}
\end{Definition}

\begin{Lemma}
\label{lem_gauge_one_degeneracy_codim}
	$\clP_\sigma$ is a $C^\infty$--subvariety of $\clP$ with codimension at least $d(\sigma)$.
\end{Lemma}

Recall that the concept of $C^\infty$--subvariety was introduced by Definition \ref{def_C_infty_subvariety}.

\begin{proof}
	Let 
	\begin{equation*}
		\widetilde \clP_\sigma := \{(\pi,B)\in \clP\times \clC \, | *F_B+V_\pi=0, 
		\mbox{ and }
	 \ker \Hess_{B,\pi}  \mbox{ represents }  \sigma\}.
	\end{equation*}
	Suppose $(\pi,B)\in \widetilde \clP_\sigma$. By elliptic regularity, $B$ is smooth and $\ker \Hess_{B,\pi}$ is spanned by smooth sections of $T^*Y\otimes \frg$ after a gauge transformation. 
	
	Let $S_{B,\epsilon}$ be a slice of $B$ given by Proposition \ref{prop_slice_thm}. 
	We claim that there exists an open neighborhood $U$ of $(\pi,B)$ in $\clP\times S_{B,\epsilon}$, such that 
	\begin{enumerate}
		\item $\widetilde \clP_\sigma\cap U$ is a Banach manifold,
		\item The projection of $\widetilde \clP_\sigma\cap U$ to $\clP$ is Fredholm and has index $-d(\sigma)$.
	\end{enumerate}
	The result then follows from the above claim and the separability of $\clP\times \clC$.
	
	To prove the claim, let $S_{B,\epsilon}^0$ be the fixed-point subspace of $S_{B,\epsilon}$ under the action of $\Stab(B)$, and let $\hat S_{B,\epsilon}^0$ be the closure of $S_{B,\epsilon}^0$ in $L_{k-1}^2(T^*Y\otimes \frg)$. Then
\begin{equation}
\label{eqn_intersect_clP_tilde_sigma}
\widetilde \clP_\sigma\cap \clP\times S_{B,\epsilon}=\widetilde \clP_\sigma\cap \clP\times S_{B,\epsilon}^0.
\end{equation}	
For each $B'\in S_{B,\epsilon}^0$, we have $\Stab(B')=\Stab(B)$. Suppose $i=k$ or $k-1$, then the spaces 
$$\ker d^*_{B'}\cap L_i^2(T^*Y\otimes \frg)$$ form a (trivial) $\Stab(B)$--equivariant Banach vector bundle over $S_{B,\epsilon}^0$. We fix a $\Stab(B)$--equivariant trivialization of this vector bundle near $B$. One way of constructing the trivialization is to take the $L^2$--orthogonal projection onto $\ker  d^*_{B}\cap L_i^2(T^*Y\otimes \frg)$.

Under the trivialization above, the set \eqref{eqn_intersect_clP_tilde_sigma} is given by the pre-image of $\{0\}\times\clH_\sigma(B)$ of the map
\begin{align*}
\varphi:\clP\times S_{B,\epsilon}^0 &\to \hat{S}_{B,\epsilon}^0 \times \mathcal{H}(B) \\
(\pi',B')&\mapsto (*F_{B'}+V_{\pi'}, \Hess_{B',\pi'}).
\end{align*}

We claim that $\varphi$ is transverse to $\{0\}\times \clH_\sigma (B)$ at $(\pi,B)$. Since transversality is an open condition, this implies that $\varphi^{-1}\big(\{0\}\times \mathcal{H}_{\sigma}(B)\big)$
is a Banach manifold near $(\pi,B)$. To prove the claim, note that by standard properties of elliptic operators, the differentials of the map 
\begin{align*}
\psi: S^0_{B,\epsilon}&\to \hat{S}_{B,\epsilon}^0\\
B' &\mapsto *F_{B'}
\end{align*}
are Fredholm. Let $L\subset \hat{S}_{B,\epsilon}^0$ be the $L^2$--orthogonal complement of the image of $d\psi$ at $B$. Then by elliptic regularity, $L$ is spanned by smooth sections, and it is a $\Stab(B)$--invariant finite dimensional subspace of $\hat{S}_{B,\epsilon}^0$. Let $\Pi_L$ and $\Pi$ denote the $L^2$--orthogonal projections to $L$ and $\ker \Hess_{B,\pi}$ respectively. 
By Proposition \ref{prop_abundance_2_pts}, the linear map 
\begin{align*}
\xi_1: \clP &\to L \\
\pi' &\mapsto \Pi_{L}(V_{\pi'})
\end{align*}
is surjective.  In fact, since $V_{\pi'} = \grad f_{\pi'}$, the value of $\Pi_{L}(V_{\pi'})$ is determined by the $1$--jet of the restriction of $f_{\pi'}$ to the finite-dimensional affine space $B+L$ by taking the gradient on this subspace. Similarly, we show that the linear map 
\begin{align*}
\xi_2: \ker \xi_1 &\to \sym_{\Stab(B)}(\ker \Hess_{B,\pi}) \\
\pi' & \mapsto \Pi\circ DV_{\pi'}|_{\ker \Hess_{B,\pi}}
\end{align*}
is surjective.
We apply Proposition \ref{prop_abundance_2_pts} with $B_1=B$ and $L_1=L+\ker \Hess_{B,\pi}$ (and an arbitrary choice of $B_2$ and $L_2$).  Since the Hessians of $2$-jets on $B_1+L_1$ at $B_1$, such that their projections to $1$-jets are zero, can realize every possible equivariant symmetric map (see Lemma \ref{lem_surj_jet_M_sigma}), Proposition \ref{prop_abundance_2_pts} implies that $\xi_2$ is surjective.
The desired claim then follows from the above surjectivity statements and Lemma \ref{lem_clH_sigma_submanifold}. In conclusion, we see that
$\varphi^{-1}\big(\{0\}\times \mathcal{H}_{\sigma}(B)\big)$
is a Banach manifold near $(\pi,B)$. 

We now compute the tangent map of the projection of $\varphi^{-1}\big(\{0\}\times \mathcal{H}_{\sigma}(B)\big)$ to $\clP$. 
Let $\Pi$ be the $L^2$--orthogonal projection to $\ker \Hess_{B,\pi}$, and let $\Pi^\perp = \id - \Pi$.
For each $\eta\in\clP$, let 
$j_1(\eta)=\Pi(V_\eta),$
 let 
 $$j_2(\eta)= \Pi\circ (DV_\eta|_{\ker \Hess_{B,\pi}}),$$
then $j_2(\eta)$ is a symmetric endomorphism on $\ker \Hess_{B,\pi}$.
 Let $T^0$, $\hat T^0$ be the tangent spaces of $S_{B,\epsilon}^0$, $\hat{S}_{B,\epsilon}^0$ at $B$ respectively. 
By Lemma \ref{lem_clH_sigma_submanifold}, the tangent space of $\varphi^{-1}\big(\{0\}\times \mathcal{H}_{\sigma}(B)\big)$ at $(\pi,B)$ is given by
$$
\widetilde{T} = \{ (\eta,b)\in \clP\times T^0 |\Hess_{B,\pi}(b) + V_\eta=0, \Pi\circ *ad(b)+ j_2(\eta)=0\}.
$$
Decompose $b$ as $b_1+b_2$ where $b_1=\Pi(b)$ and $b_2=\Pi^\perp(b)$, and let $\Hess_{B,\pi}^{-1}$ be the inverse of $\Hess_{B,\pi}$ on $(\ker \Hess_{B,\pi})^\perp$. Then $(\eta,b)\in \widetilde{T}$ if and only if 
\begin{enumerate}
	\item $j_2(\eta) = -\Pi\circ *ad(b)$,
	\item $j_1(\eta) = 0$,
	\item $b_2 = - \Hess_{B,\pi}^{-1} \Pi^\perp(V_\eta)$.
\end{enumerate}
Therefore by Proposition \ref{prop_abundance_2_pts}, the linear map
\begin{align*}
 \widetilde{T}&\to \clP\times \ker \Hess_{B,\pi} \\
 (\eta,b) &\mapsto (\eta,b_1)
\end{align*}
 is injective with closed image, and the codimension of its image equals 
 $$\dim \ker \Hess_{B,\pi} + d(\sigma).$$ 
 Since the projection from $\clP\times \ker \Hess_{B,\pi}$ to $\clP$ is Fredholm and with index equal to $\dim \ker \Hess_{B,\pi}$, we conclude that the projection of $\widetilde{T}$ to $\clP$ is Fredholm with index $-d(\sigma)$, and hence the result is proved.
\end{proof}

\begin{Definition}
Suppose $\sigma_1,\sigma_2\in \clR_G$.
	Let $\clP_{\sigma_1,\sigma_2}$ to be the set of $\pi\in\clP$, such that there exist $B_1,B_2\in \clC$ with the following properties:
	\begin{enumerate}
		\item $\Orb(B_1)\neq \Orb(B_2),$
		\item $*F_{B_1}+V_{\pi_1}=0$, 
		$*F_{B_2}+V_{\pi_2}=0, $
		\item  $\ker \Hess_{B_1,\pi}$   represents  $\sigma_1$,  
		\item 
		 $\ker \Hess_{B_2,\pi}$ represents  $\sigma_2$.
	\end{enumerate}
\end{Definition}

\begin{Lemma}
\label{lem_gauge_two_degeneracy_codim}
Suppose $\sigma_1,\sigma_2\in \clR_G$. Then
	$\clP_{\sigma_1,\sigma_2}$ is a $C^\infty$--subvariety of $\clP$ with codimension at least  $d(\sigma_1)+d(\sigma_2)$.
\end{Lemma}

\begin{proof}
The proof is essentially the same as Lemma \ref{lem_gauge_one_degeneracy_codim}.
Let 
	\begin{multline*}
		\widetilde \clP_{\sigma_1,\sigma_2} := \{(\pi,B_1,B_2)\in \clP\times \clC \times \clC\, | \Orb(B_1)\neq \Orb(B_2),
		\\
		*F_{B_1}+V_{\pi_1}=0,\,
		\ker \Hess_{B_1,\pi}   \mbox{ represents } \sigma_1,
		\\
		*F_{B_2}+V_{\pi_2}=0,\,
		\ker \Hess_{B_1,\pi}   \mbox{ represents } \sigma_2
		\}. 
	\end{multline*}

	Suppose $(\pi,B_1,B_2)\in \widetilde \clP_{\sigma_1,\sigma_2}$. 
	By elliptic regularity, $B_i$ ($i=1,2$) are smooth and $\ker \Hess_{B_i,\pi}$ are spanned by smooth sections of $T^*Y\otimes \frg$ after gauge transformations.
	 
	For $i=1,2$, let $S_{B_i,\epsilon_i}$ be a slice of $B_i$ given by Proposition \ref{prop_slice_thm}. We claim that there exists an open neighborhood $U$ of $(\pi,B_1,B_2)$ in $\clP\times S_{B_1,\epsilon_1}\times S_{B_2,\epsilon_2}$, such that 
	\begin{enumerate}
		\item $\widetilde \clP_{\sigma_1,\sigma_2}\cap U$ is a Banach manifold,
		\item The projection of $\widetilde \clP_{\sigma_1,\sigma_2}\cap U$ to $\clP$ is Fredholm and has index $-d(\sigma_1)-d(\sigma_2)$.
	\end{enumerate}
	The result then follows from the above claim and the separability of $\clP\times \clC\times \clC$.

	To prove the claim, let $S_{B_i,\epsilon_i}^0$, $\hat{S}_{B_i,\epsilon_i}^0$ be as in the proof of  Lemma \ref{lem_gauge_one_degeneracy_codim}. Then
$$\widetilde \clP_{\sigma_1,\sigma_2}\cap \clP\times S_{B_1,\epsilon_1}\times S_{B_2,\epsilon_2} = \widetilde \clP_{\sigma_1,\sigma_2}\cap \clP\times S_{B_1,\epsilon_1}^0\times S_{B_2,\epsilon_2}^0,$$
and it is given by the pre-image of $\{0\}\times \mathcal{H}_{\sigma_1}(B_1)\times\{0\}\times  \mathcal{H}_{\sigma_2}(B_2)$ of the map
\begin{align*}
\varphi:\clP\times S_{B_1,\epsilon_1}^0\times S_{B_2,\epsilon_2}^0 &\to \hat{S}_{B_1,\epsilon_1}^0 \times \mathcal{H}(B_1) \times \hat{S}_{B_2,\epsilon_2}^0 \times \mathcal{H}(B_2)\\
(\pi',B_1',B_2')&\mapsto (*F_{B_1'}+V_{\pi'}, \Hess_{B_1',\pi'},*F_{B_2'}+V_{\pi'},\Hess_{B_2',\pi'}).
\end{align*}
Similar to the proof of Lemma \ref{lem_gauge_one_degeneracy_codim}, by Lemma \ref{lem_clH_sigma_submanifold} and  Proposition \ref{prop_abundance_2_pts}, we have
\begin{equation}
\label{eqn_varphi_inverse_2_pts_gauge}
\varphi^{-1}\big(\{0\}\times \mathcal{H}_{\sigma_1}(B_1)\times \{0\}\times \mathcal{H}_{\sigma_2}(B_2)\big)
\end{equation}	
is a Banach manifold near $(\pi,B_1,B_2)$. By the same argument as the proof of Lemma \ref{lem_gauge_one_degeneracy_codim}, the projection of \eqref{eqn_varphi_inverse_2_pts_gauge} to $\clP$ is Fredholm with index $-d(\sigma_1)-d(\sigma_2)$ near $(\pi,B_1,B_2)$.
\end{proof}

Lemma \ref{lem_gauge_one_degeneracy_codim} and Lemma \ref{lem_gauge_two_degeneracy_codim} have the following immediate corollaries. Recall that by Definition \ref{def_non-deg_perturbed_connection}, $\pi\in \clP$ is called non-degenerate if all the critical points of $\CS+f_\pi$ are non-degenerate.

\begin{Corollary}
The set $\clP^{reg}\subset \clP$ of non-degenerate holonomy perturbations is of Baire second category.
\end{Corollary}

\begin{Corollary}
\label{generic_path}
For any pair $\pi_0, \pi_1 \in \clP^{reg}$, one can find a generic smooth path $\pi_t: [0,1] \to \clP$ from $\pi_0$ to $\pi_1$, such that there are only countably many $t$ where $\pi_t$ is degenerate. Moreover, for every such $t$ there is exact one degenerate critical orbit $\Orb(B)$ of $\CS+f_{\pi_t}$, and the kernel of $\Hess_{B,\pi_t}$ is an irreducible representation of $\Stab(B)$.
\end{Corollary}

\section{$\SU(n)$ Casson invariants for integer homology spheres}
\label{sec_su(n)_casson}
This section proves Theorem \ref{thm_existence_casson_intro}.
From now on, we assume that $G=\SU(n)$ and $Y$ is an integer homology sphere.
\subsection{Classification of stabilizers}
\label{sebsec_classify_stab}
When $G=\SU(n)$, there is a combinatorial classification of all possible stabilizers on $\clC$. 

Recall that $P$ is the trivial $\SU(n)$--bundle given by \eqref{eqn_def_P_trivial}.
Let $B$ be a connection on $P$.  Let $y_0\in Y$, then $\Stab(B)$ embeds as a closed subgroup of $\SU(n)$ by restricting to $y_0$. Let $\hol_{y_0}(B)\subset G$ be the holonomy group of $B$ at $y_0$, then $\hol_{y_0}(B)$ is a closed subgroup of $\SU(n)$, and the restriction of $\Stab(B)$ to $y_0$ equals the commutator subgroup of $\hol_{y_0}(B)$.

View $\bC^n$ is a representation of $\hol_{y_0}(B)\subset \SU(n)$, then an element $g\in \SU(n)$ is in the commutator group of $\hol_{y_0}(B)$ if and only if it gives a $\hol_{y_0}(B)$--module homomorphism on $\bC^n$. Suppose the isotypic decomposition of $\bC^n$ as a representation of $\hol_{y_0}(B)$ is given by 
\begin{equation}
\label{eqn_Hol(B)_isotypic_decomp}
	\bC^n\cong V(n_1)^{\oplus m_1} \oplus\cdots \oplus V(n_r)^{\oplus m_r},
\end{equation}
where $V(n_i), 1 \leq i \leq r$ are $\bC$--vector spaces of dimension $n_i$. Then by Schur's lemma, every unitary $\hol_{y_0}(B)$--homomorphism $g$ can be decomposed as 
$$
g=\diag ( g_1, \cdots, g_r ),
$$
where 
$$
g_i\in \U(m_i)\otimes \id_{V(n_i)}.
$$
Hence the
commutator subgroup of $\hol_{y_0}(B)$ is given by the kernel of
\begin{align*}
\U(m_1)\times\cdots\times \U(m_r) &\to \U(1)\\
(u_1,\cdots,u_r)&\mapsto \det(u_1)^{n_1}\cdots \det(u_r)^{n_r},
\end{align*}
which will be denoted by 
\begin{equation}
\label{eqn_def_S(U)}
	\s (\U(m_1)^{n_1}\times\cdots\times \U(m_r)^{n_r}).
\end{equation}

Let $E$ be the unitary $\bC^n$--bundle associated to $P$. 
By taking parallel translations with respect to $B$, the decomposition \eqref{eqn_Hol(B)_isotypic_decomp} at $y_0$ gives a decomposition of $E$. Since $Y$ is an integer homology sphere, every unitary complex vector bundle over $Y$ is trivial, therefore we may write $E$ as 
$$
E = E(n_1)^{\oplus m_1} \oplus \cdots \oplus E(n_r)^{\oplus m_r},
$$
where every $E(n_i)$ is a trivial $\bC$--vector bundle of rank $n_i$, and the connection $B$ decomposes as the direct sum of irreducible unitary connections on $E(n_i)$. 

 We summarize the previous discussions as follows.

\begin{Definition}
\label{def_sigma_n}
	Let $\Sigma_n$ be the set of tuples of positive integers
\[ \big((n_1, m_1), \dots, (n_r, m_r)\big), \]
such that 
\begin{enumerate}
	\item $n = \sum_{i=1}^{r} m_i \, n_i$,
	\item $n_1 \leq n_2 \leq \cdots \leq n_r$,
	\item the $m_i$'s are in non-decreasing order if the corresponding $n_i$'s are the same.
\end{enumerate}
\end{Definition}

\begin{Definition}
\label{def_partition_clC}
Suppose $\sigma=\big((n_1, m_1), \dots, (n_r, m_r)\big)\in \Sigma_n$. Define $\clC_\sigma$ to be the subset of $\clC$ consisting of $B\in \clC$, such that $P\times_{\SU(n)} \bC^n$ decomposes as $$
E = E(n_1)^{\oplus m_1} \oplus \cdots \oplus E(n_r)^{\oplus m_r},
$$
where every $E(n_i)$ is a trivial $\bC$--vector bundle of rank $n_i$, and the connection $B$ decomposes as the direct sum of irreducible unitary connections on $E(n_i)$.
\end{Definition}

By the previous discussions, $\clC$ is stratified by the union of $\clC_\sigma$ for $\sigma\in\Sigma_n$.

\begin{Definition}
\label{def_partial_order_Sigma_n}
	Suppose $\sigma_1,\sigma_2\in\Sigma_n$. We write $\sigma_1\prec\sigma_2$ if and only if $\clC_{\sigma_1}$ is in the closure of $\clC_{\sigma_2}$. Then $\prec$ defines a partial order on $\Sigma_n$.
\end{Definition}

\begin{Definition}
\label{def_H_sigma}
	Let $\sigma =\big((n_1, m_1), \dots, (n_r, m_r)\big)\in \Sigma_n$. Define
	$$
		H_\sigma:=	\s (\U(m_1)^{n_1}\times\cdots\times \U(m_r)^{n_r}).
	$$
	to be the subgroup of $\SU(n)$ associated to $\sigma$.
\end{Definition}

\begin{Definition}
\label{def_clC^H}
	Suppose $H$ is a closed subgroup of $\SU(n)$. Recall that $\clG$ is identified with the set of $L_{k+1}^2$--maps from $Y$ to $\SU(n)$. Let $\clG_H$ be the subgroup of $\clG$ consisting of constant maps to $H$. Define $\clC^{H}\subset \clC$ to be the fixed point set of $\clG_H$.
\end{Definition}

By definition, $\clC^{H}$ is an affine space, and $\theta\in \clC^{H}$ for all $H$. Suppose $B\in \clC_\sigma$, then there exists a gauge transformation $g\in\clG$ such that $g(B)\in \clC^{H_\sigma}$, and the linear homotopy from $g(B)$ to $\theta$ remains in $\clC^{H_\sigma}$.

Suppose $\sigma_1\prec \sigma_2$, then there exists $g\in G$ such that $ H_{\sigma_2} \subset g\cdot H_{\sigma_1}\cdot g^{-1}$. We abuse the notation and also use $g\in\clG$ to denote the constant map from $Y$ to $g\in G$, then
$$
g(\clC^{H_{\sigma_1}})\subset \clC^{H_{\sigma_2}}.
$$

\subsection{Equivariant spectral flow}
\label{subsec_equiv_spec_flow}
Let $\mathcal{V}$ be a Hilbert space, and let $\mathcal{D}\subset \mathcal{V}$ be a dense subspace. Suppose $f_t:\mathcal{D}\to \mathcal{V}$, $t\in[0,1]$, is a smooth family of self-adjoint operators on $\mathcal{V}$, such that the spectra of $f_t$ is discrete on $\bR$ for all $t$, and that $0$ is not in the spectra of $f_0$ and $f_1$. Let $\lambda_0>0$ be the minimum of absolute values of the eigenvalues of $f_0$ and $f_1$. For a generic $c\in(-\lambda_0,\lambda_0)$, the eigenvalues of the family $f_t+c\cdot \id$ cross zero transversely. The \emph{spectral flow} of $f_t$, $t\in[0,1]$, is defined to be the number of times where a negative eigenvalue of $f_t+c\cdot \id$ crosses zero and becomes a positive eigenvalue, minus the number of times where a positive eigenvalue of $f_t+c\cdot \id$ crosses zero and becomes a negative eigenvalue, as $t$ goes from $0$ to $1$.

Suppose $H$ is a compact Lie group that acts on $\mathcal{V}$, and suppose the family $f_t$ is $H$--equivariant, then we can refine the definition of spectral flow and obtain an element in the representation ring of $H$ as follows. Recall that the representation ring of $H$ is denoted by $\clR(H)$. Let $\clR^{irr}(H)$ be the set of isomorphism classes of irreducible representations of $H$. Then for each $W\in \clR^{irr}(H)$, $f_t$ defines a family of self-adjoint operators on $\Hom_{H}(W,\mathcal{V})$. Let $n_W\in \bZ$ be the spectral flow of the induced operators on $\Hom_{H}(W,\mathcal{V})$ by $f_t$, then we define the \emph{equivariant spectral flow} of the family $f_t$ to be
$$
\sum_{W\in\clR^{irr}(H)} n_W \cdot[W] \in \clR(H).
$$

Alternatively, the equivariant spectral flow can be described as follows. As before, let $\lambda_0>0$ be the minimum of absolute values of the eigenvalues of $f_0$ and $f_1$, and take $c\in(-\lambda_0,\lambda_0)$ such that the eigenvalues of the family $f_t+c \cdot \id$ cross zero transversely. Suppose  $f_t+c\cdot \id$ has eigenvalue zero for $t=t_1,\cdots,t_r$. We may further perturb $c$ such that at each $t_i$, the eigenvalues either cross zero from the negative side to the positive side, or from positive the positive side to the negative side, but not in both directions. Let $\eta_i=1$ if the eigenvalues cross zero from the negative side to the positive side at $t_i$, and let $\eta_i=-1$ if the eigenvalues cross from the positive side to the negative side at $t_i$.
At each $t_i$, the kernel of $f_{t_i}+c\cdot \id$ is finite-dimensional and $H$--invariant, and hence it defines an element $[W_i]\in \clR(H)$. Then the equivariant spectral flow of $f_t$ is given by 
$$
\sum_{i=1}^r \eta_i \cdot [W_i].
$$

If $f_0$ or $f_1$ have non-trivial kernel, we define the \emph{equivariant spectral flow} to be the spectral flow from $f_0+\epsilon\cdot \id$ to $f_1+\epsilon\cdot \id$, for $\epsilon$ positive and sufficiently small.

\begin{remark}
	When $f_0$ or $f_1$ have non-trivial kernel,
	our convention of the spectral flow is different from \cite[Definition 4.1]{boden1998the}. The current convention is slightly more convenient for the later discussions.
\end{remark}

\begin{Definition}
\label{def_equivariant_spectral_flow}
	Suppose $H$ is a closed subgroup of $\SU(n)$, let $B\in \clC^{H}$, and let $\pi\in \clP$. Recall that the operator $K_{B,\pi}$ is defined by Equation \eqref{eqn_def_K}. Define
	$$
	Sf_{H} (B,\pi) \in \clR(H)
	$$
	to be the $H$-equivariant spectral flow from $K_{B,\pi}$ to $K_{\theta,0}$,  by the linear homotopy from $(B,\pi)$ to $(\theta,0)$. 
\end{Definition}

Suppose $H_1\subset H_2$, let 
$$
r^{H_2}_{H_1}:\clR(H_2)\to \clR(H_1)
$$
be the homomorphism given by restrictions of representations of $H_2$ to representations of $H_1$. 
Suppose $B\in \clC^{H_2}$, then we have
$$
Sf_{H_1}(B,\pi) = r^{H_2}_{H_1} \big(Sf_{H_2}(B,\pi)\big).
$$

If $\sigma_1,\sigma_2\in\Sigma_n$ satisfies $\sigma_1\prec \sigma_2$ (see Definition \ref{def_partial_order_Sigma_n}), let $g$ be an element of $G$ such that $ H_{\sigma_2} \subset g\cdot H_{\sigma_1}\cdot g^{-1}$. Then there is a homomorphism
\begin{equation}
\label{eqn_def_r_g}
	r_g: \clR(H_{\sigma_1}) \to \clR(H_{\sigma_2})
\end{equation}
that takes the isomorphism class of a  representation
$$
\rho:H_{\sigma_1}\to \Hom( V, V)
$$
to the isomorphism class of
\begin{align*}
 H_{\sigma_2} &\to \Hom(V, V) \\
							h &\mapsto \rho(g^{-1}hg).
\end{align*}
Suppose $B\in \clC^{H_{\sigma_1}}$. We abuse the notation and let $g\in\clG$ denote the constant map from $Y$ to $g$, then $g(B)\in \clC^{H_{\sigma_2}}$, and we have
\begin{equation}
\label{eqn_change_Sf_H_under_gauge}
Sf_{H_{\sigma_2}}\big(g(B)\big) = r_g \big(Sf_{H_{\sigma_1}} (B)\big).
\end{equation}

Notice that $Sf_H$ is in general not gauge invariant: if $B\in\clC^H$, and $g\in \clG$ is a gauge transformation such that $g(B)\in \clC^{H}$, then $Sf_H (B)$ may \emph{not} be equal to $Sf_H (g(B))$. This issue will be discussed in Section \ref{subsec_gauge_inv_equi_ind}.

\subsection{Index of non-degenerate perturbed flat connections}
\label{subsec_gauge_inv_equi_ind}
This subsection constructs a correction term that cancels the gauge ambiguity of the equivariant spectral flow. 

We start with the following technical lemma. Let $\clA\subset \clC$ be the space of flat connections on $P$. 

\begin{Lemma}
	\label{lem_clA_locally_connected}
	For each $\sigma$, the space $\clA\cap \clC^{H_\sigma}$ is locally connected.
\end{Lemma}

\begin{proof}
Take $B\in \clA\cap \clC^{H_\sigma}$, and let $S_{B,\epsilon}$ be a slice given by Proposition \ref{prop_slice_thm}. Since $\clC^{H_\sigma}$ is an affine subspace of $\clC$, the intersection of $S_{B,\epsilon}$ and $\clC^{H_\sigma}$ is a linearly embedded disk in $S_{B,\epsilon}$. The operator $B'\mapsto *F_{B'}$ on $S_{B,\epsilon}\cap \clC^{H_\sigma}$ is a nonlinear operator whose linearization at $B$ has a finite-dimensional kernel. If we restrict the image to $\ker d^*_B\cap \clC^{H_\sigma}$, then the linearization is a Fredholm operator. By Kuranishi reduction, the space of solutions to $*F_{B'}=0$ on $S_{B,\epsilon}\cap \clC^{H_\sigma}$ is homeomorphic to the zero set of a map between finite dimensional linear spaces.  It is straightforward to verify that the Kuranishi model is, in fact, given by analytic maps. Therefore, the space of solutions to $*F_{B'}=0$ on $S_{B,\epsilon}\cap \clC^{H_\sigma}$ is homeomorphic to an analytic variety. By \cite{whitney1965local}, every analytic variety is locally connected. Therefore, by the properties of slices in Proposition \ref{prop_slice_thm}, we conclude that $\clA\cap \clC^{H_\sigma}$ is locally connected.
\end{proof}

By the Uhlenbeck compactness theorem, the quotient space $\clA/\clG$ is compact.
By Lemma \ref{lem_clA_locally_connected}, there exists a $\clG$--invariant open neighborhood $\mathcal{U}$ of $\clA$, such that for each $\sigma\in \Sigma_n$, the inclusion 
$$
\clA\cap \clC^{H_\sigma} \hookrightarrow \mathcal{U}\cap \clC^{H_\sigma}
$$
induces a one-to-one correspondence on the set of connected components. Since $\clA/\clG$ is compact, there exists $r_0>0$ depending on $Y$, such that if $\|\pi\|_{\clP} < r_0$, then all critical points of $\CS + f_\pi$ lie in $\mathcal{U}$.

Take 
$$\sigma=\big((n_1, m_1), \dots, (n_r, m_r)\big)\in\Sigma_n,$$
 and suppose $B\in \mathcal{U}\cap \clC^{H_\sigma}$.
Then $E=P\times_{\SU(n)}\bC^n$  is decomposed as 
$$
E = E(n_1)^{\oplus m_1} \oplus \cdots \oplus E(n_r)^{\oplus m_r},
$$
 where $E(n_i)$ have rank $n_i$, and $E(n_i)$ are constant subbundles of $E$ with respect to the trivialization \eqref{eqn_def_P_trivial}. The connection $B$ is given by the direct sum of irreducible connections on each $E(n_i)$. 
 
 Since $B\in \mathcal{U}\cap \clC^{H_\sigma}$, there exists $\hat B\in \clA\cap \clC^{H_\sigma}$, such that $\hat B$ and $B$ are in the same connected component of $\mathcal{U}\cap \clC^{H_\sigma}$. The connection $\hat B$ is also given by the direct sum of flat connections on each $E(n_i)$. Let $\hat B_i$ be the restriction of $\hat B$ to $E(n_i)$. 
 Since the Chern-Simons functional is constant on the connected components of $\clA$, the value of $\CS(\hat B_i)$ is independent of the choice of $\hat B$. 
 
 Suppose $g:E\to E$ is a gauge transformation that decomposes as the direct sum of $g_1,\cdots,g_r$, where $g_i:E(n_i)\to E(n_i)$ is a unitary bundle map. Then each $g_i$ is given by a map from $Y$ to $U(n_i)$. Since $Y$ is an integer homology sphere, $g_i$ is homotopic to a constant map if $n_i=1$, and when $n_i\ge 2$, then the homotopy class of $g_i$ is classified by the induced map 
 \begin{equation}
 \label{eqn_H3_Y_to_H3_U}
 	  H_3(Y;\bZ)\to H_3(U(n_i))\cong \bZ.
 \end{equation}
Recall that $Y$ is oriented and fix an isomorphism from $H_3(U(n_i))$ to $\bZ$, the map \eqref{eqn_H3_Y_to_H3_U} identifies the homotopy classes of $g_i$ with $\bZ$. We use $\deg g_i\in \bZ$ to denote the image of the homotopy class of $g_i$ in $\bZ$, and we call it the \emph{degree} of $g_i$.
 
Fix an arbitrary $\pi_0\in\clP$ and $B_0\in \clC^{H_\sigma}$. For each $i$, let $g^{(i)}: E\to E$ be a fixed unitary bundle automorphism which commutes with $H_\sigma$ such that its component on $E(n_i)$ has degree $1$ and its components on $E(n_j)$ are the identity for all $j\neq 0$. Define $\tau_i\in \clR(H_\sigma)$ to be the $H_\sigma$--equivariant spectral flow from $K_{B_0,\pi_0}$ to $K_{g^{(i)}(B_0),\pi_0}$.
 Because the homotopy class of a unitary automorphism on $E(n_i)$ is determined by its degree, and because of the excision property of the index, we have for all $B\in\clC^{H_\sigma}$ and $\pi\in \clP$, the $H_\sigma$--equivariant spectral flow from $K_{B,\pi}$ to $K_{g(B),\pi}$ is equal to
 $$
 \sum_{n_i\ge 2}  \deg (g_i)\cdot \tau_i.
 $$

On the other hand, let $B_i$ be the restriction of $B$ to $E(n_i)$, then for each $n_i\ge 2$, we have
 $$
 \CS(g_i(B_i))-\CS(B_i) = 4\pi^2 n_i\cdot \deg (g_i).
 $$
 
 \begin{Definition}
	Suppose $B\in \mathcal{U}\cap \clC^{H_\sigma}$. Define 
	$$
	\CS_\sigma(B):= \sum_{n_i\ge 2} \frac{\CS(\hat B_i)}{4\pi^2 n_i} \cdot \tau_i \in \clR(H_\sigma)\otimes \bR,
	$$
	where $\tau_i$ are given as above. 
\end{Definition}
 
Now suppose $\pi\in \clP$ satisfies $\|\pi\|_{\clP} < r_0$ so that all the critical points of $\CS+f_\pi$ are contained in $\mathcal{U}$, and suppose $B$ is $\pi$--flat. We define a gauge-invariant equivariant index of $B$ that takes value in $\widetilde{R}_{\SU(n)}$ (see Definition \ref{def_tilde_R_G}).

 \begin{Definition}
 \label{def_equivariant_ind_gauge}
Let $\pi,B$ be as above. Take $\sigma\in\Sigma_n$ such that $B\in\clC_\sigma$, and take $g\in \clG$ such that $g(B)\in \clC^{H_\sigma}$.
 	Define $\ind(B,\pi)\in \widetilde\clR_{\SU(n)}([H_\sigma])$ to be the element represented by
 	\begin{equation}
 	\label{eqn_def_equiv_index_gauge}
 		 	Sf_{H_\sigma}(g(B),\pi) -[\ker d_{g(B)}] - \CS_\sigma(g(B)) \in \clR(H_\sigma)\otimes \bR,
 	\end{equation}
 	where $[\ker d_{g(B)}]\in\clR([H_\sigma])$ is given by $\ker d_{g(B)}\subset L_k^2(\frg)$ as an $H_\sigma$--representation. 
 \end{Definition}
 
 \begin{remark}
 	The extra term $[\ker d_{g(B)}]$ is necessary for the proof of Lemma \ref{lem_compare_indices_on_slice}.
 \end{remark}

The definition of $\ind(B)$ is independent of the choice of $g$ and therefore is gauge-invariant.

Notice that since $Y$ is an integer homology sphere, the Chern-Simons functional of any flat connection on a line bundle over $Y$ equals zero. Therefore we have
\begin{equation}
\label{eqn_split_of_CS}
\frac{\CS(\hat B)}{4\pi^2 n} = \sum_{n_i\ge 2} \frac{\CS(\hat B_i)}{4\pi^2 n_i} \cdot m_i,
\end{equation}
where the normalizing constants on the denominators come from the convention in the definition of the Chern-Simons functional \eqref{eqn_def_chern_simons}.

If $\sigma_1,\sigma_2\in\Sigma_n$ satisfies $\sigma_1\prec \sigma_2$ (see Definition \ref{def_partial_order_Sigma_n}), let $g$ be an element of $G$ such that $ H_{\sigma_2} \subset g\cdot H_{\sigma_1}\cdot g^{-1}$. The homomorphism $r_g$ given by \eqref{eqn_def_r_g} extends linearly to a homomorphism
$$
	r_g: \clR(H_{\sigma_1})\otimes \bR \to \clR(H_{\sigma_2})\otimes \bR.
$$
Suppose $B\in \clC^{H_{\sigma_1}}$. We abuse the notation and let $g\in\clG$ be the constant map from $Y$ to $g$, then $g(B)\in \clC^{H_{\sigma_2}}$, and \eqref{eqn_split_of_CS} implies that
\begin{equation}
\label{eqn_change_Cs_sigma_under_gauge}
\CS_{\sigma_2}\big(g(B)\big) = r_g \big(\CS_{\sigma_1} (B)\big).
\end{equation}

Recall that by Uhlenbeck's compactness theorem, the moduli space of $\pi$--flat connections is compact for all $\pi\in \clP$. If $\pi$ is non-degenerate, then the critical set is finite.

\begin{Definition}
\label{def_total_index_gauge}
	Suppose $\pi\in \clP$ satisfies $\|\pi\|_{\clP} < r_0$ so that all the critical points of $\CS+f_\pi$ are contained in $\mathcal{U}$, and suppose that $\pi$ is non-degenerate. Define the \emph{total index} of $\pi$ by
	$$
	\ind(\pi):=\sum_{\Orb(B)\textrm{ is }\pi-\textrm{flat}}\ind(B,\pi)\in \bZ\widetilde\clR_{\SU(n)}
	$$
\end{Definition}

We also introduce the following refinement of Definition \ref{def_total_index_gauge}.

\begin{Definition}
	Suppose $\pi\in \clP$ satisfies $\|\pi\|_{\clP} < r_0$ so that all the critical points of $\CS+f_\pi$ are contained in $\mathcal{U}$, and suppose that $\pi$ is non-degenerate. Let $\eta$ be a connected component of $\mathcal{U}$, define
	$$
	\ind_\eta(\pi):=\sum_{\substack{\Orb(B)\textrm{ is }\pi-\textrm{flat}\\ \Orb(B) \cap \eta\neq \emptyset}}\ind(B,\pi)\in \bZ\widetilde\clR_{\SU(n)}.
	$$
\end{Definition}

\subsection{Comparison of the total index}
\label{sebsec_change_total_index_gauge}
Suppose $\pi_0,\pi_1\in \clP$ are non-degenerate and sufficiently small such that $\ind(\pi_0), \ind(\pi_1)$ are defined. Let $\clA$, $\mathcal{U}$ be as in Section \ref{subsec_gauge_inv_equi_ind}. Suppose $\eta$ is a connected component of $\mathcal{U}$.

The main result of this subsection is the following theorem.

\begin{Theorem}
\label{thm_change_total_index_gauge}
	$\ind_\eta (\pi_0)- \ind_\eta (\pi_1) \in\widetilde{\Bif}_{\SU(n)}.$
\end{Theorem}

And we have the following immediate corollary.
\begin{Corollary}
\label{cor_change_total_index_gauge_all}
	$\ind (\pi_0)- \ind (\pi_1) \in\widetilde{\Bif}_{\SU(n)}.$
\end{Corollary}

\begin{proof}[Proof of Theorem \ref{thm_change_total_index_gauge}]
	We use the Kuranishi reduction argument to reduce to the finite-dimensional case so that we can invoke Theorem \ref{thm_equivairant_cerf_finite_dim}.
	
	Take a generic smooth path $\pi_t$ ($t\in[0,1]$) from $\pi_0$ to $\pi_1$ in the sense of Corollary \ref{generic_path}. To simplify the notation, we will denote $\CS + f_{\pi_t}$ by $\CS_t$.
	
	Suppose $\CS_{t_0}$ is degenerate at $B_0$ where $B_0\in \eta$. Let $V_0\subset \ker d_{B_0}^*$ be the kernel of $\Hess_{B_0,\pi_{t_0}}$. For $m=k,k-1$,  let $V_0^{\perp(m)}$ be the orthogonal complement of $V_0$ in $\ker d_{B_0}^*\cap  L_{m}^2(T^*Y\otimes \frg)$. Then $\Hess_{B_0,\pi_t}$ restricts to an isomorphism from $V_0^{\perp(k)}$ to $V_0^{\perp(k-1)}$.
	
	Let 
	$$\Pi_0:L^2_k(T^*Y\otimes \frg)\to V_0$$
	 be the $L^2$--orthogonal projection onto $V_0$,
	 and let 
	 $$\Pi_0^{\perp(m)} :L^2_k(T^*Y\otimes \frg)\to V_0^{\perp(m)}$$
	 be the $L^2$--orthogonal projection onto $V_0^{\perp(m)}$ for $m=k,k-1$. 
	
	For $\epsilon>0$ sufficiently small, let $S_{B_0,\epsilon}$ be the slice of $B_0$ given by Proposition \ref{prop_slice_thm}, and define
	\begin{equation}
	\label{eqn_def_M_epsilon,t}
		M_{t,\epsilon}:=\{B\in S_{B_0,\epsilon}| \Pi_0^{\perp(k-1)} (*F_B+V_{\pi_t}) =0\}.
	\end{equation}
	Then by the implicit function theorem, there exists $\epsilon_0>0$, such that for all  $\epsilon\in(0,\epsilon_0)$ and all $t\in(t_0-\epsilon_0,t_0+\epsilon_0)$, the set $M_{t,\epsilon}$ is an embedded manifold with dimension $\dim V_0$, and the $L^2$--orthogonal projection of $M_{t,\epsilon}$ to $V_0$ is a smooth embedding with open image.
	
	Let $H_0$ be the stabilizer of $B_0$, then $H_0$ acts on $M_{t,\epsilon}$.
	Let $U_\epsilon(B_0)$ be the image of $S_{B_0,\epsilon}$ under the action of $\clG$. By Proposition  \ref{prop_slice_thm}, $U_\epsilon(B_0)$ is a $\clG$--invariant open neighborhood of $B_0$. For all  $\epsilon\in(0,\epsilon_0)$ and all $t\in(t_0-\epsilon_0,t_0+\epsilon_0)$, the critical orbit of $\CS_t$ on $U_\epsilon(B_0)$ is in one-to-one correspondence with the critical orbits of the restriction of $\CS_t$ to $M_{t,\epsilon}$. Since $\CS_{t_0}$ has exactly one degenerate critical orbit $B_0$, it has at most countably many critical orbits, and thus we may choose $\epsilon$ such that  $\partial U_\epsilon(B_0)$ contains no critical orbit of $\CS_{t_0}$. By Uhlenbeck's compactness theorem, there are only finitely many critical points of $\CS_{t_0}$ on $\clC-U_\epsilon(B_0)$. Also recall that $\pi_t$ is non-degenerate except for countably many values of $t$. Therefore there exist $t_+\in (t_0,t_0+\epsilon_0)$ and $t_-\in (t_0-\epsilon_0,t_0)$, such that 
	\begin{enumerate}
		\item for all $t\in (t_-,t_+)$, the boundary $\partial U_\epsilon(B_0)$ contains no critical orbit of $\CS_t$;
		\item for all  $t\in (t_-,t_+)$, all the critical points of $\CS_t$ on $\clC-U_\epsilon(B_0)$ are  non-degenerate.
	\end{enumerate}
	
	As a consequence, take $t_-'\in (t_-,t_0)$ and $t_+'\in (t_0,t_+)$ such that $\pi_{t_-'}$ and $\pi_{t_+'}$ are non-degenerate, then the difference 
	$$
	\ind \pi_{t_+'}- \ind \pi_{t_-'}
	$$
	is given by the difference of the total indices of $\CS_{t_+'}$ and $\CS_{t_-'}$ on $U_\epsilon(B_0)$. We claim that:
		\begin{Lemma}
	\label{lem_compare_indices_on_slice}
		Let $H_0$, $\CS_t$ and $M_{t,\epsilon}$ be as above, and suppose $B\in M_{t,\epsilon}$ is a critical point of the restriction of $\CS_{t}$ to $M_{t,\epsilon}$. Suppose $(\epsilon,t)$ is sufficiently close to $(0,t_0)$. Then $B$ is non-degenerate as a critical point of  $M_{t,\epsilon}$ when regarding $M_{t,\epsilon}$ as a finite dimensional $H_0$--manifold, if and only if $B$ is non-degenerate as a $\pi_t$--flat connection. If $B$ is non-degenerate, let $\ind_{t,\epsilon} B\in \clR_{H_0}$ be the index of the critical orbit of $B$ as a point on $M_{t,\epsilon}$, then $\ind (B,\pi_t)\in \widetilde\clR_{\SU(n)}$ is given by
		$$
		\ind (B,\pi_t) = i^{H_0}_{\SU(n)}\big( \ind B_0 \oplus \ind_{t,\epsilon} B).
		$$
	\end{Lemma}
	
	We will postpone the proof of Lemma \ref{lem_compare_indices_on_slice} to the next subsection.
	Lemma \ref{lem_compare_indices_on_slice} and Theorem \ref{thm_equivairant_cerf_finite_dim} imply that
	$$
		\ind \pi_{t_+'}- \ind \pi_{t_-'}\in \widetilde\Bif_{\SU(n)}.
	$$

	Now consider
	$$\clS:=\{t\in(0,1)|\pi_t \mbox{ is degenerate} \}.$$
	The value of $\ind \pi_t$ is constant on any open interval in $[0,1]-\clS$. 
	For each $t_0\in \clS$, let $I_{t_0}:=(t_-,t_+)$ be the open interval given as above, then by the previous argument, the image of $\ind \pi_t$ in 
	$$\widetilde \clR_{\SU(n)}/\widetilde \Bif_{\SU(n)}$$
	 is constant on $I_{t_0}-\clS$. Since $\clS$ is countable, $I_{t_0}-\clS$ is a dense subset of $I_{t_0}$.
	Since $\clS$ is compact, there exists a finite subset of $\clS$ such that the corresponding open intervals $I_{t_0}$ cover $\clS$, therefore the theorem is proved.
	\end{proof}

 We can now prove Theorem \ref{thm_existence_casson_intro}  as a straightforward consequence of Theorem \ref{thm_change_total_index_gauge} and Proposition \ref{prop_quotient_by_tilde_Bif}. We first repeat the statement of the theorem using the notation defined in Section \ref{sec_su(n)_casson}.

 \begin{reptheorem}{thm_existence_casson_intro}
 	For every $n\ge 3$, there exists a function 
 	$$
 	w: \widetilde{\clR}_{\SU(n)} \to \bC
 	$$
 	with the following property. 
 	Suppose $Y$ is an integer homology sphere, let 
 	$$P=\SU(n)\times Y$$ 
 	be the trivial $\SU(n)$--bundle over $Y$, let $\theta$ be the trivial connection of $P$.
 	Then for a generic holonomy perturbation $\pi$, the critical set of the perturbed Chern-Simons functional consists of finitely many non-degenerate orbits. 
 	Let $\clM_\pi$ be the moduli space of critical points of the Chern-Simons functional perturbed by $\pi$, and decompose $\clM_\pi$ as 
 	$$\clM_\pi=\clM_\pi^*\sqcup \clM_\pi^r,$$
 	where $\clM_\pi^*$ consists of irreducible critical orbits, and $\clM_\pi^r$ consists of reducible critical orbits. Then for $\pi$ sufficiently small, the sum
 	$$
 		\lambda_{w}:= \sum_{[B]\in \clM^*} (-1)^{Sf(B,\pi)} + \sum_{[B]\in \clM^r} e^{\pi i\cdot \CS(\hat B)/(\pi^2)}\cdot w(\ind B)
 	$$
 	is independent of $\pi$, where $Sf(B,\pi)\in \bZ$ is the (classical) spectral flow from $K_{B,\pi}$ to $K_{\theta,0}$ via the linear homotopy, and $\hat B$ is a flat connection close to $B$. 
 \end{reptheorem}
 
\begin{proof}[Proof of Theorem \ref{thm_existence_casson_intro}]
Let $\{1\}$ be the trivial group, then $\clR(\{1\})\otimes\bR \cong \bR$, therefore $\bR$ canonically embeds in $\widetilde{\clR}_{\SU(n)}$ as the image of $\clR(\{1\})\otimes\bR$. Under this identification, we have $\bR\cap \widetilde{\clR}_{\SU(n)}^{(0)}=[0,1)$. Let $w$ be an arbitrary function from $\widetilde{\clR}_{\SU(n)}^{(0)}$ to $\bC$ such that $w(s) = e^{\pi i s}$ on $[0,1)$. By Proposition \ref{prop_quotient_by_tilde_Bif}, the function $w$ can be uniquely extended to a homomorphism
$$w:\bZ\widetilde{\clR}_{\SU(n)}\to \bC.$$ such that $w=0$ on $\widetilde \Bif_{\SU(n)}$. 

We claim that $w$ satisfies the desired condition. Since $w=0$ on $\widetilde \Bif_{\SU(n)}$, we have $w(s+1)=-w(s)$ on the image of $\bR$ in $\widetilde{\clR}_{\SU(n)}$, therefore $w(s) = e^{\pi i s}$ on the image of $\bR$.

 Let $\clA$, $\mathcal{U}$ be as in Section \ref{subsec_gauge_inv_equi_ind}, and suppose $\eta$ is a connected component of $\mathcal{U}$. By Theorem \ref{thm_change_total_index_gauge}, the sum
	$$
			\lambda_{\eta,w}:= \sum_{\substack{[B]\in \clM^*\\ [B]\cap \eta\neq \emptyset }} w(\ind B) + \sum_{\substack{[B]\in \clM^r\\ [B]\cap \eta\neq \emptyset}}  w(\ind B)
	$$
	is independent of the perturbation $\pi$. Notice that for $[B]\in \clM^*$, the equivariant index of $B$ is an element of $\clR(\{1\})\otimes\bR\cong \bR$ given by 
	$$
	Sf(B,\pi) -\frac{\CS(B_\eta)}{\pi^2},
	$$
	where $B_\eta$ is a flat connection in the connected component $\eta$, and $Sf(B,\pi)\in \bZ$ is the (classical) spectral flow from $K_{B,\pi}$ to $K_{\theta,0}$ via the linear homotopy. By the definition of $\mathcal{U}$, the value of $\CS(B_\eta)$ is independent of the choice of $B_\eta$. Therefore
	$$
		e^{\pi i \CS(B_\eta)/(4\pi^2 n)} \cdot 	\lambda_{\eta,w}= \sum_{\substack{[B]\in \clM^*\\ [B]\cap \eta\neq \emptyset }} (-1)^{Sf(B,\pi)} + \sum_{\substack{[B]\in \clM^r\\ [B]\cap \eta\neq \emptyset }} e^{\pi i\cdot \CS(B_\eta)/(\pi^2)}\cdot w(\ind B)
	$$
	is independent of the choice of $\pi$, and hence the theorem is proved.
\end{proof}

\subsection{Proof of Lemma \ref{lem_compare_indices_on_slice}}
This subsection is devoted to the proof of Lemma \ref{lem_compare_indices_on_slice}. 	

The analogous statement for the finite-dimensional case is clear: the change of equivariant index in Example 3.3 is the same as the change of equivariant index in Example 3.2, because in Example 3.3, the contribution of the equivariant index from $V'$ are the same for $t>0$ and $t<0$ and hence can be canceled. This turns out to be less obvious in the infinite-dimensional case, because one cannot cancel the contribution from an infinite-dimensional subspace and has to work with spectral flows instead.

The idea of the proof is to compute the equivariant spectral flow by reducing it to a finite-dimensional problem. We modify the relevant 1-parameter families of self-adjoint operators by small perturbations, so that the resulting spectral flow is only contributed by a finite-dimensional subspace and can be computed directly.  The technical difficulty comes from the fact that we need to find a reduction simultaneously compatible with $d_B$ and $d_{B_0}$.

As in the proof of Theorem \ref{thm_change_total_index_gauge}, let $V_0$ be the kernel of $\Hess_{B_0,\pi_{t_0}}$. For $m=k,k-1$, let $V_0'^{(m)}$ be the $L^2$--orthogonal complement of $V_0$ in $\ker d_{B_0}^*\cap L_m^2(T^*Y\otimes \frg)$. Let $H_0=\Stab(B_0)$, $H=\Stab(B)$.

	Notice that $\ker d_{B_0}$ is the tangent space of $\Stab(B_0)$, and $\ker d_{B}$ is the tangent space of $\Stab(B)$. Since $B\in S_{B_0,\epsilon}$, by Proposition \ref{prop_slice_thm}, we have $\Stab(B)\subset \Stab(B_0)$, therefore $\ker d_B\subset \ker d_{B_0}$.
	
	Let $V_1(B)$ be the $L^2$--orthogonal complement of $\ker d_B$ in $\ker d_{B_0}$. Then $V_1(B)$ is a finite dimensional subspace of $\ker d_{B_0}$, and we have
	$$T_{\id} H_0 = T_{\id} H\oplus V_1(B).$$
	Since $B_0$ is in $L_{k}^2$, it follows from the standard bootstrapping argument that 
	$$\ker d_{B_0}\subset L_{k+1}^2(\frg).$$

	Let 
	$$
	V_2(B):=d_{B}(\ker d_{B_0})\subset L_{k}^2(T^*Y\otimes \frg),
	$$
	then $V_2(B)$ is a finite dimensional subspace of $\ima d_B$. Notice that $V_2(B)$ is the tangent space of the $H_0$--orbit of $B$. Therefore $V_2(B)\subset \ker d_{B_0}^*$,
	and the map
	$$
	d_B: V_1(B)\to V_2(B)
	$$
	is an isomorphism. 
	
	For $m=k,k-1$, let $V_2'(B)^{(m)}$ be the $L^2$--orthogonal complement of $V_2(B)$ in $\ima d_B\cap L_{m}^2(T^*Y\otimes \frg)$.
	
	 Let $\Pi_B$ be the $L^2$--orthogonal projection onto $\ker d_B^*$, and define
	$$V_3(B):= \Pi_B(V_2(B)^\perp),$$
	where $V_2(B)^\perp$ is the $L^2$--orthogonal complement of $V_2(B)$ in $T_B M_{\epsilon, t}$. 
	Then $V_3(B)$ is finite dimensional, and we have
	$$\dim V_0 = \dim T_BM_{\epsilon, t} = \dim V_2(B) +\dim  V_3(B).$$ 
	
	For $m=k,k-1$, let $V_3'(B)^{(m)}$ be the orthogonal complement of $V_3(B)$ in 
	$$ \ker d_B^*\cap L_{m}^2(T^*Y\otimes\frg).$$ 

To simplify the notation, we make the following definitions.
\begin{Definition}
	We say that a function $f(B)$ of $B$ \emph{converges} to $c$ as $(\epsilon,t)\to (0,t_0)$, if for every $\delta>0$, there exists $\epsilon_1>0$ depending on $B_0$, such that whenever $\epsilon<\epsilon_1$ and $t\in(t_0-\epsilon_1,t_0 +\epsilon_1)$, we have $|f(B)-c|<\delta$. 
\end{Definition}
\begin{Definition}
	Suppose $W,W'$ are Banach spaces, and $\iota_{V}: V \hookrightarrow W$, $\iota_{V'}: V' \hookrightarrow W'$ are embeddings of fixed closed subspaces. Suppose $V(B)\subset W$, $V'(B)\subset W'$ are closed subspaces depending on $B$. 
	
	\begin{enumerate}
		\item We say that $V(B)$ \emph{converges} to $V$ as $(\epsilon,t)\to (0,t_0)$, if for every $\delta>0$, there exists $\epsilon_1>0$ depending on $B_0$, such that whenever $\epsilon<\epsilon_1$ and $t\in(t_0-\epsilon_1,t_0 +\epsilon_1)$, there exists a bounded linear operator $\varphi:V\to W$, such that $\varphi(V)=V(B)$, and 
	$$\|\varphi-\iota_{V}\|<\delta,$$
	 where $\|\cdot\|$ denotes the operator norm.
	 \item Suppose $H:V\to V'$ is a bounded linear operator, and $H(B):V(B)\to V'(B)$ is a bounded linear operator that depends on $B$. We say that $H(B)$ converges to $H$ as $(\epsilon,t)\to (0,t_0)$, if for every $\delta>0$, there exists $\epsilon_1>0$ depending on $B_0$, such that whenever $\epsilon<\epsilon_1$ and $t\in(t_0-\epsilon_1,t_0 +\epsilon_1)$, there exist bounded linear operators $\varphi:V\to W$, $\varphi':V'\to W'$, such that 
	 $$\varphi(V)=V(B), \, \varphi'(V')=V'(B),$$
	$$\|\varphi-\iota_{V}\|<\delta,\, \|\varphi'-\iota_{V'}\|<\delta,$$
	$$ \|H - \varphi^{-1}\circ H(B)\circ \varphi \|<\delta.$$
	\end{enumerate}
\end{Definition}

\begin{Lemma}
\label{lem_subspaces_converge}
Suppose $m= k$ or $k-1$. Then 
	$V_2'(B)^{(m)}$ converges to $\ima d_{B_0}\cap L_m^2(T^*Y\otimes \frg),$ and $V_3'(B)^{(m)}$ converges to $V_0'^{(m)}$, as $(\epsilon,t)\to (0,t_0)$.
\end{Lemma} 
\begin{proof}
Let $(\ker d_{B_0})^\perp\subset L^2(\frg)$ be the $L^2$--orthogonal complement of $\ker d_{B_0}$.  Then 
$$
d_{B_0}:(\ker d_{B_0})^\perp \cap L_{m+1}^2 (\frg)\to \ima d_{B_0}\cap L_m^2(T^*Y\otimes \frg)
$$ 
is an isomorphism. Let
$$
d_{B_0}^{-1}: \ima d_{B_0}\cap L_m^2(T^*Y\otimes \frg)\to (\ker d_{B_0})^\perp \cap L_{m+1}^2 (\frg)
$$
be its inverse map. 
Let $\Pi_2$ be the $L^2$--orthogonal projection to $V_2(B)$, and
let $\Pi_2^\perp=\id - \Pi_2$.  Then 
 \begin{align*}
 V_2'(B)^{(m)} &=\Pi_2^\perp\circ d_B \big((\ker d_{B_0})^\perp \cap L_{m+1}^2 (\frg)\big)
 \\
 &= \Pi_2^\perp\circ d_B\circ d_{B_0}^{-1}\big(\ima d_{B_0}\cap L_m^2(T^*Y\otimes \frg)\big) .
\end{align*}

  Let $\iota: \ima d_{B_0}\cap L_m^2(T^*Y\otimes \frg) \hookrightarrow L_m^2(T^*Y\otimes \frg)$ be the inclusion map. Since $V_2(B)\subset \ker d_{B_0}^*$, we have $\Pi_2=0$ on $\ima d_{B_0}\cap L_m^2(T^*Y\otimes \frg) $. Therefore there exists a constant $z_1$, such that 
	\begin{align*}
		&\|\Pi_2^\perp\circ d_B\circ d_{B_0}^{-1} - \iota \| \\
	\le & \|\Pi_2^\perp\circ d_B\circ d_{B_0}^{-1} - \Pi_2^\perp\circ d_{B_0}\circ d_{B_0}^{-1}  \| + \|\Pi_2^\perp -\iota\| 
	\\
	= & \|\Pi_2^\perp\circ d_B\circ d_{B_0}^{-1} - \Pi_2^\perp\circ d_{B_0}\circ d_{B_0}^{-1}  \|
	\\
	\le & z_1 \cdot \|B-B_0\|_{L_{k,B_0}^2}\|d_{B_0}^{-1}\|\le \epsilon\, z_1\, \|d_{B_0}^{-1}\|.
	\end{align*}
	Hence  $V_2'(B)^{(m)}$ converges to $\ima d_{B_0}\cap L_m^2(T^*Y\otimes \frg)$ as $(\epsilon,t)\to (0,t_0)$.
	
 Recall that $V_2(B)^\perp$ denotes the $L^2$--orthogonal complement of $V_2(B)$ in $T_BM_{\epsilon,t}$, and $\Pi_B$ denotes the $L^2$--orthogonal projection onto $\ker d_B^*$. 
	 By definition,
	$$
	V_3(B) =\Pi_B( V_2(B)^\perp).
	$$
	
	Let $\Pi_B^\perp = \id - \Pi_B$, then $\Pi_B^\perp$ is  the orthogonal projection onto $\ima d_B\cap L_k^2(T^*Y\otimes \frg)$, which is spanned by $V_2(B)$ and the space
	\begin{equation}
	\label{eqn_space_ker_dB_perp}
		 d_B\big((\ker d_{B_0})^\perp\cap L_{k+1}^2(T^*Y\otimes \frg)\big).
	\end{equation}
	Moreover, as $(\epsilon,t)\to (0,t_0)$, the space \eqref{eqn_space_ker_dB_perp} converges to 
	$$
	d_{B_0}\big((\ker d_{B_0})^\perp\cap L_{k+1}^2(T^*Y\otimes \frg)\big) = V_0'^{(k)}.
	$$
	Since $V_2(B)^\perp\subset T_B M_{t,\epsilon}$ which converges to $V_0$, and $V_0$ is orthogonal to $V_0'^{(k)}$, we conclude that 
	\begin{equation}
	\label{eqn_lim_Pi_B_perp_on_V2_perp}
	\lim_{(\epsilon,t)\to(0,t_0)} \, \| \Pi_B^{\perp}\,|_{V_2(B)^\perp}\| = 0. 
	\end{equation}
	and 
	\begin{equation}
	\label{eqn_lim_id-Pi_B_perp_on_V2_perp}
	\lim_{(\epsilon,t)\to(0,t_0)} \, \| \Pi_B\,|_{V_0'}\| = 0,	
	\end{equation}

	By \eqref{eqn_lim_Pi_B_perp_on_V2_perp}, $V_3(B)$ gets arbitrarily close to $V_2(B)^\perp$ as $(\epsilon,t)\to (0,t_0)$. Therefore $V_3(B)$ is transverse to $V_0'^{(m)}$. 	Let $\Pi_3$ be the $L^2$--orthogonal projection to $V_3(B)$, and let $\Pi_3^\perp=\id -\Pi_3$, then we have 
	 \begin{equation}
	 \label{eqn_V3'(B)_as_maps_from_V0'}
	 	V_3'(B)^{(m)}=\Pi_3^\perp\circ \Pi_B^\perp(V_0'^{(m)}).
	 \end{equation}

	 By \eqref{eqn_lim_id-Pi_B_perp_on_V2_perp}, the space $\Pi_B^\perp(V_0'^{(m)})$ converges to $V_0'^{(m)}$ as $(\epsilon,t)\to(0,t_0)$. 
	 Since $V_2(B)^\perp$ is tangent to $TM_{t,\epsilon}$, which is orthogonal to $V_0'^{(m)}$, we have that $\Pi_3^\perp(V_0'^{(m)})$ converges to $V_0'^{(m)}$ as $(\epsilon,t)\to(0,t_0)$. Therefore the desired result follows from \eqref{eqn_V3'(B)_as_maps_from_V0'}. 
\end{proof}
	
	Now we return to the proof Lemma \ref{lem_compare_indices_on_slice}. 
	For $m=k,k-1$, let $(\ker d_{B_0})^\perp_{(m)}$ be the orthogonal complement of $\ker d_{B_0}$ in $L_m^2(\frg)$. Then the domain of the operator $K_{B_0,\pi_{t_0}}$ is orthogonally decomposed as 
	$$
	\ker d_{B_0}\oplus (\ker d_{B_0})^\perp_{(k+1)} \oplus (\ima d_{B_0}\cap L_k^2(T^*Y\otimes\frg) ) \oplus V_0\oplus V_0'^{(k)},
	$$
	and the range of $K_{B_0,\pi_{t_0}}$ is orthogonally decomposed as 
	$$
	\ker d_{B_0}\oplus (\ker d_{B_0})^\perp_{(k)} \oplus (\ima d_{B_0}\cap L_{k-1}^2(T^*Y\otimes\frg) )  \oplus V_0\oplus V_0'^{(k-1)}.
	$$
	Under this decomposition, the operator $K_{B_0,\pi_{t_0}}$ is given by the matrix
	\begin{equation}
	\label{eqn_K_at_B0_decomp}
		\begin{pmatrix}
		0 & 0& 0 & 0 & 0\\
		0 & 0 & d_{B_0}^* & 0 & 0\\
		0 & d_{B_0} & 0 & 0 & 0\\
		0 & 0 & 0 & 0 & 0 \\
		0 & 0 & 0 & 0 & \Hess_{B_0,\pi_{t_0}}
	\end{pmatrix}.
	\end{equation}
Moreover, the restricted maps
	\begin{equation}
	\label{eqn_isom_dB0_on_ker_perp}
		d_{B_0}
		: (\ker d_{B_0})^\perp_{(k)} \to \ima d_{B_0}\cap L_k^2(T^*Y\otimes\frg) 
	\end{equation}
	and 
	\begin{equation}
	\label{eqn_hess_isom_on_V0'}
		\Hess_{B_0,\pi_{t_0}}: V_0'^{(k)} \to V_0'^{(k-1)}
	\end{equation}
	are isomorphisms. 
	
	Now we study the operator $K_{B,\pi_t}$. Decompose the domain of $K_{B,\pi_t}$ as 
	\begin{equation*}
		\ker d_{B} \oplus V_1(B)   \oplus (\ker d_{B_0})^\perp_{(k+1)}  \oplus V_2(B) \oplus V_2'(B)^{(k)}  \oplus V_3(B)\oplus V_3'(B)^{(k)},
	\end{equation*}
	and decompose the range of  $K_{B,\pi_t}$ as 
	\begin{equation*}
	\ker d_{B} \oplus V_1(B)  \oplus (\ker d_{B_0})^\perp_{(k)} \oplus V_2(B) \oplus V_2'(B)^{(k-1)}   \oplus V_3(B)\oplus V_3'(B)^{(k-1)}.
	\end{equation*}
	Then by Equation \eqref{eqn_decomp_K_when_flat}, the operator is given by a matrix of the form
	\begin{equation}
	\label{eqn_K_at_B_decomp}
		\begin{pmatrix}
		0 & 0 & 0 & 0 & 0 & 0 & 0 \\
		0 & 0 & 0 & M_{11}^* & M_{21}^* & 0 & 0 \\
		0 & 0 & 0 & M_{12}^* & M_{22}^* & 0 & 0 \\
		0 & M_{11} & M_{12} & 0 & 0 & 0 & 0 \\
		0 & M_{21} & M_{22} & 0 & 0 & 0 & 0 \\	
		0 & 0      & 0      & 0      & 0      & N_{11} & N_{12} \\
		0 & 0      & 0      & 0      & 0      & N_{21} & N_{22} 
		\end{pmatrix}.
	\end{equation}

	By the previous arguments, the operator $K_{B,\pi_t}$ maps $V_1(B)$ isomorphically to $V_2(B)$. Therefore $M_{11}$ is invertible, and $M_{21}=0$. By Lemma \ref{lem_subspaces_converge}, $M_{22}$ converges to
	 the isomorphism \eqref{eqn_isom_dB0_on_ker_perp} as $(\epsilon,t)\to (0,t_0)$, therefore $M_{22}$ is an isomorphism when $|\epsilon|$ and $|t-t_0|$ are sufficiently small.
	
	We now study the matrix 
	$$
	\begin{pmatrix}
		N_{11} & N_{12} \\
		N_{21} & N_{22}
	\end{pmatrix}.
	$$
\begin{Lemma}
\label{lem_property_of_2by2_matrix_N}
		There exists a constant $\epsilon_1>0$ depending only on $B_0$, such that the following statements hold when $|\epsilon|, |t-t_0|<\epsilon_1$:
		\begin{enumerate}
			\item The operator	$$
	\begin{pmatrix}
		N_{11} & N_{12} \\
		N_{21} & N_{22}
	\end{pmatrix}
	$$
		is invertible if and only if $N_{11}$ is invertible.
		\item Suppose $N_{11}$ is invertible, then 
	$$
	\begin{pmatrix}
		N_{11} & sN_{12} \\
		sN_{21} & N_{22}
	\end{pmatrix}
	$$
	is invertible for all $s\in[0,1]$. 
	\end{enumerate}
	\end{Lemma}
		
\begin{proof}
	By Lemma \ref{lem_decomp_K_when_flat} and the defintions, the image of $K_{B,\pi_t}$ on $V_3(B)$ is given by the directional derivatives of $\grad \CS_t$ on $M_{t,\epsilon}$. By the definition of $M_{t,\epsilon}$ from \eqref{eqn_def_M_epsilon,t}, we have
	$$
	\Pi_0^{\perp(k-1)} \big(K_{B,\pi_t} (V_3(B)) \big)= 0.
	$$
	By Lemma \ref{lem_subspaces_converge}, the space $V_3'(B)^{(k-1)}$ converges to $V_0'^{(k-1)}$ as $(\epsilon,t)\to (0,t_0)$. Therefore, there exists a constant $z_3$ indepedent of $B$, such that
	\begin{equation}
	\label{eqn_estimate_N21_from_above}
			z_1\|N_{11}(v)\| \ge \|N_{21}(v)\| 
	\end{equation}
	for all $v\in V_3(B)$ when $(\epsilon,t)$ is sufficiently close to $(0,t_0)$.
	
	Recall that $N_{11}$ is a linear endomorphism on the finite dimensional space $V_3(B)$.
	If $N_{11}$ is non-invertible, then $\ker N_{11}\neq 0$. Let $v\in V_3(B)$ be a non-zero vector in $\ker N_{11}$, then by \eqref{eqn_estimate_N21_from_above}, we have $N_{21}(v)= 0$, thus 
	$\begin{pmatrix}
		N_{11} & N_{12} \\
		N_{21} & N_{22}
	\end{pmatrix}$
	is non-invertible.
	
	We now assume $N_{11}$ is invertible and prove Part (2). By taking $s=1$, Part (2) implies Part (1) of the lemma in the case that $N_{11}$ is invertible.
	
	 Since the operator $K_{B,\pi}$ is self-adjoint and depends continuously on $(B,\pi)$, we have
	$$\lim_{(\epsilon,t)\to (0,t_0)}\|N_{12}\| = \lim_{(\epsilon,t)\to (0,t_0)}\|N_{21}\|=0.
	$$
	By Lemma \ref{lem_subspaces_converge},  the operator $N_{22}$ converges to the isomorphism \eqref{eqn_hess_isom_on_V0'} as $(\epsilon,t)\to (0,t_0)$.
	Therefore, there exist constants $z_4$, $\epsilon_2$ only depending on $B_0$, such that when $\epsilon<\epsilon_2$, $t\in(t_0-\epsilon_2,t_0+\epsilon_2)$, we have
	$$
	\|N_{22}^{-1}\|\le z_2, \quad \mbox{ and } \quad 
	\|N_{12}\|, \, \|N_{21}\| \le \frac{1}{2z_1\,z_2}.
	$$
	Notice that
	\begin{align}
	&\begin{pmatrix}
		\id & -s N_{12} \circ N_{22}^{-1} \\
	0  & \id
	\end{pmatrix}
	\cdot 
	\begin{pmatrix}
		N_{11} & sN_{12} \\
		sN_{21} & N_{22}
	\end{pmatrix}
	\cdot
	\begin{pmatrix}
		\id & 0 \\
		-s N_{22}^{-1}\circ N_{21} & \id
	\end{pmatrix}
	\nonumber
	\\
	\label{eqn_elementary_matrix_transform_N22}
	= &
	\begin{pmatrix}
		N_{11}-s^2 N_{12}\circ N_{22}^{-1}\circ N_{21} & 0 \\
		0 & N_{22}
	\end{pmatrix}.
	\end{align}
	For every $v\in V_3(B)$ and $s\in[0,1]$, we have 
	\begin{align*}
		\big\|\big(N_{11}-s^2 N_{12}\circ N_{22}^{-1}\circ N_{21}\big)(v)\big\| 
		&\ge
		 \|N_{11}(v)\|- s^2\|N_{12}\|\cdot \| N_{22}^{-1}\|\cdot \|N_{21}(v)\|
		\\
		& \ge \|N_{11}(v)\|- s^2\|N_{12}\|\cdot \| N_{22}^{-1}\|\cdot (z_1 \|N_{11}(v)\|)
		\\
		& \ge \|N_{11}(v)\| - s^2 \cdot \frac{1}{2z_1\,z_2}\cdot z_2 \cdot  (z_1 \|N_{11}(v)\|)
		\\
		& \ge \frac12  \|N_{11}(v)\|.
	\end{align*}
Since $N_{11}$ is injective, the estimates above imply that $N_{11}-s^2 N_{12}\circ N_{22}^{-1}\circ N_{21} $ is injective, therefore it 
	is invertible for all $tsin[0,1]$. By \eqref{eqn_elementary_matrix_transform_N22}, the operator 
	$$\begin{pmatrix}
		N_{11} & sN_{12} \\
		sN_{21} & N_{22}
	\end{pmatrix}$$
	 is invertible for all $s\in[0,1]$.		
\end{proof}

We can now finish the proof of Lemma \ref{lem_compare_indices_on_slice}. By definition, $B$ is non-degenerate as a $\pi_t$--flat connection if and only if $\begin{pmatrix}
		N_{11} & N_{12} \\
		N_{21} & N_{22}
	\end{pmatrix}$ is invertible. By Part (1) of Lemma \ref{lem_property_of_2by2_matrix_N}, this is equivalent to $N_{11}$ being invertible. On the other hand, $N_{11}$ is conjugate to the Hessian of $\CS_t$ as a function on $M_{t,\epsilon}$ restricted to the normal direction of the $H_0$--orbit of $B$. Therefore $B$ is non-degenerate as a $\pi_t$--flat connection if and only if it is non-degenerate as a critical point of $\CS_t$ on the $H_0$--manifold $M_{t,\epsilon}$. 
	
To compare the indices of $B_0$ and $B$ as perturbed flat connections when $B$ is non-degenerate, we need to compute the $H$--equivariant spectral flow from the operator \eqref{eqn_K_at_B_decomp} to the operator \eqref{eqn_K_at_B0_decomp}. 

Recall that $M_{11}$ is invertible and $M_{12}=0$. Therefore by Lemma \ref{lem_property_of_2by2_matrix_N}, the linear deformation from \eqref{eqn_K_at_B0_decomp} to 
\begin{equation}
\label{eqn_spectral_flow_step_1}
\begin{pmatrix}
		0 & 0 & 0 & 0 & 0 & 0 & 0 \\
		0 & 0 & 0 & M_{11}^* & 0 & 0 & 0 \\
		0 & 0 & 0 & 0 & M_{22}^* & 0 & 0 \\
		0 & M_{11} & 0 & 0 & 0 & 0 & 0 \\
		0 & 0 & M_{22} & 0 & 0 & 0 & 0 \\	
		0 & 0      & 0      & 0      & 0      & N_{11} & 0 \\
		0 & 0      & 0      & 0      & 0      & 0 & N_{22} 
		\end{pmatrix}
\end{equation}
has zero spectral flow.

We then deform \eqref{eqn_spectral_flow_step_1} to the following operator via a linear deformation:
\begin{equation}
\label{eqn_spectral_flow_step_2}
	\begin{pmatrix}
		0 & 0 & 0 & 0 & 0 & 0 & 0 \\
		0 & 0 & 0 & 0 & 0 & 0 & 0 \\
		0 & 0 & 0 & 0 & M_{22}^* & 0 & 0 \\
		0 & 0 & 0  & 0 & 0 & 0 & 0 \\
		0 & 0 & M_{22} & 0 & 0 & 0 & 0 \\	
		0 & 0      & 0      & 0      & 0      & 0 & 0 \\
		0 & 0      & 0      & 0      & 0      & 0 & N_{22} 
\end{pmatrix}.
\end{equation}
Let $V_3^-(B)\subset V_3(B)$ be the subspace generated by the negative eigenvectors of $N_{11}$, then the $H$--equivariant spectral flow from \eqref{eqn_spectral_flow_step_1} to \eqref{eqn_spectral_flow_step_2} is given by 
$$
[V_2(B)] + [V_3^-(B)] \in \clR(H).
$$

Finally, notice that all the maps constructed in Lemma \ref{lem_subspaces_converge} are $H$--equivariant, therefore when $(\epsilon,t)$ is sufficiently close to $(0,t_0)$, the linear homotopy from \eqref{eqn_spectral_flow_step_2} to \eqref{eqn_K_at_B0_decomp} has zero spectral flow. In conclusion, the $H$--equivariant spectral flow from $K_{B,\pi_t}$ to $K_{B_0,\pi_{t_0}}$ is represented by the $H$--representation $V_2(B)\oplus V_3^-(B)$. Since $V_3^-(B)$ also represents the equivariant index of $B$ as a critical point on $M_{\epsilon,t}$, and 
$$
[V_2(B)] = [\ker d_{B_0}] - [\ker d_B]
$$
in the representation ring of $H$, the desired result follows from the definition of $\ind(B,\pi)$ and Equations \eqref{eqn_change_Sf_H_under_gauge} and \eqref{eqn_change_Cs_sigma_under_gauge}.

\section{Computations and examples}
\label{sec_examples}
\subsection{Characterization of irreducible bifurcations on $\clC$}
\label{subsec_cla}

Notice that although the definition of $\clR_{G}$ is given by the representations of \emph{all} closed subgroups of $G$, for a fixed $G$--manifold $M$, there are only finitely many possible conjugation classes of $\Stab(p)$ for $p\in M$, and there are only finitely many representations (up to conjugations by $G$) that can represent the equivariant index of critical points. 

Similarly, for the perturbed Chern-Simons functionals on $\clC$, only finitely many stabilizer groups (up to conjugations) and irreducible representations arise in the description of bifurcations. We already classified all the possible stabilizer groups in Section \ref{sebsec_classify_stab}, this subsection classify all the possible irreducible representations, and characterize the corresponding bifurcations. 

Recall that the set $\Sigma_n$ is defined by Definition \ref{def_sigma_n}. Let 
$$\sigma=((n_1, m_1), \dots, (n_r, m_r))\in\Sigma_n,$$
and let $B\in \clC_\sigma$. After a gauge transformation, we may assume that $B\in \clC^{H_\sigma}$. Then $E=P\times_{\SU(n)}\bC^n$ decomposes as 
\begin{equation}
\label{eqn_decompse_E_compute}
	E = E(n_1)^{\oplus m_1} \oplus \cdots \oplus E(n_r)^{\oplus m_r},
\end{equation}
where $E(n_i)$ are constant subbundles of $E$ with rank $n_i$, 
and $B$ is given by the direct sum of irreducible connections on each $E(n_i)$. 

Recall that $\clT$ denotes the tangent bundle of $\clC$, and we have
$$
\clT|_{B} = L_k^2(T^*Y\otimes \frg ).
$$
The action of $\Stab(B)\cong H_\sigma$ on $\clT|_{B}$ is given pointwise on $\frg$ by the adjoint action of $H_\sigma$. Therefore, we only need to find all the irreducible components of $\frg$ as $H_\sigma$--representations. 

Decompose $\frg = \mathfrak{su}(n)$ as
$$
\frg = \frg_\sigma \oplus \frg_\sigma^\perp,
$$
where $\frg_\sigma$ is the Lie algebra of the subgroup of $\SU(n)$ that preserves the decomposition \eqref{eqn_decompse_E_compute}, and $\frg_\sigma^\perp$ is the orthogonal complement of $\frg_\sigma$. Then the action of $H_\sigma$ on $\frg_\sigma$ is trivial, and one only needs to compute the irreducible components of $\frg_\sigma^\perp$ as a representation of $H_\sigma$. 
The space $\mathfrak{g}_\sigma^{\perp}\subset \mathfrak{su}(n)$ consists of the matrices 
\begin{equation}
\label{eqn_matrix_decompose_Wpq}
\begin{pmatrix}
W_{11} & W_{12} & \dots & W_{1r}  \\
W_{21} & W_{22} & \dots & W_{2r}  \\
\vdots & \vdots & \ddots & \vdots \\
W_{r1} & W_{r2} & \dots & W_{rr}
\end{pmatrix},
\end{equation}
such that for $1 \leq p, q \leq r$:
\begin{enumerate}
\item if $p \neq q$, $W_{pq}$ is a $\bC-$valued $m_p n_p \times m_q n_q$ matrix, and $W_{pq} = -W_{qp}^*$;
\item if $p = q$, $W_{pp}$ is given by
\[
\begin{pmatrix}
W^{(p)}_{11} & W^{(p)}_{12} & \dots & W^{(p)}_{1m_{p}} \\
W^{(p)}_{21} & W^{(p)}_{22} & \dots & W^{(p)}_{2m_{p}} \\
\vdots & \vdots & \ddots & \vdots \\
W^{(p)}_{m_p 1} & W^{(p)}_{m_p 2} & \dots & W^{(p)}_{m_p m_{p}}
\end{pmatrix}, \]
where $$W^{(p)}_{11}, \dots, W^{(p)}_{m_p m_{p}} \in \mathfrak{su}(n_p),$$
 $$W^{(p)}_{11} + \dots + W^{(p)}_{m_p m_{p}} = 0,$$
 and  $$W^{(p)}_{i, j} = -(W^{(p)}_{j, i})^{*} \mbox{ for all }i, j.$$
\end{enumerate}.

For $1\le p\le q\le r$, let $\frg_\sigma^\perp(p,q)$ be the subspace of $\mathfrak{g}_\sigma^{\perp}$ consisting of matrices in the form \eqref{eqn_matrix_decompose_Wpq} such that $W_{ij}=0$ unless $(i,j)=(p,q)$ or $(q,p)$. Then $\frg_\sigma^\perp(p,q)$ is invariant under the action of $H_\sigma$. 

For $p=1,\cdots,r$, let
$$
\varphi_p: H_\sigma \to \U(m_p)
$$
be given by the restriction of $H_\sigma$ to $E(n_p)^{m_p}$. Then 
\begin{equation}
\label{eqn_ima_varphi_p_from_H_sigma}
\ima \varphi_p = 
\begin{cases}
\U(m_p) & \textrm{ if } r\ge 2,\\
\SU(m_p) &  \textrm{ if } r=1.
\end{cases}  
\end{equation}
Moreover, for $1\le p<q\le r$, the image of 
$$\varphi_p\times \varphi_q:H_\sigma\to \U(m_p)\times \U(m_q)$$
 is given by 
\begin{equation}
\label{eqn_ima_varphi_p_times_varphi_q_from_H_sigma}
\ima (\varphi_p \times \varphi_q)= 
\begin{cases}
\U(m_p) \times \U(m_q) & \textrm{ if } r\ge 3,\\
	\s \big(\U(m_p)^{n_p}\times \U(m_q)^{n_q}\big).
&  \textrm{ if } r=2,
\end{cases}  
\end{equation}
where the group $\s \big(\U(m_p)^{n_p}\times \U(m_q)^{n_q}\big)$ is defined by \eqref{eqn_def_S(U)}.

For $1\le p\le r$, let $V_p$ be the representation of $H_\sigma$ on $\mathfrak{su}(m_p)$  given by the composition of $\varphi_p$ and the adjoint action of $\U(m_p)$.
For $1\le p<q\le r$, let $V_{p,q}$ be the representation of $H_\sigma$ on  $\mat_{m_p\times m_q}(\bC)$, where the action of $h\in H_\sigma$ on $x\in \mat_{m_p\times m_q}(\bC)$ is given by $\varphi_p(h)\cdot  x\cdot \varphi_q(h)^{-1}$. 

Then the following lemma gives a complete description of the isotypic decomposition of $\frg_\sigma^\perp$ as an $H_\sigma$--representation.

\begin{Lemma}
	Suppose $1\le p< q\le r$. 
	\begin{enumerate}
		\item  $V_p$ and $V_{p,q}$ are irreducible representations of $H_\sigma$. 
		\item The representation of $H_\sigma$ on $\frg_\sigma^\perp(p,q)$ is given by the direct sum of $n_p\cdot n_q$ copies of  $V_{p,q}$.
		\item The representation of $H_\sigma$ on $\frg_\sigma^\perp(p,p)$ is given by the direct sum of $n_p^2$ copies of  $V_{p}$.
	\end{enumerate}
\end{Lemma}
\begin{proof}
The irreducibility of $V_p$ and $V_{p,q}$ follows from \eqref{eqn_ima_varphi_p_from_H_sigma} and \eqref{eqn_ima_varphi_p_times_varphi_q_from_H_sigma}. 
	The rest of the lemma is a straightforward consequence of the definition of $\frg_\sigma^\perp(p,q)$.
\end{proof}

\begin{Definition}
Suppose 
$$\sigma=((n_1, m_1), \dots, (n_r, m_r))\in\Sigma_n,$$
and let $V_p$ and $V_{p,q}$ be given as above.
	We say that $\sigma'\in\Sigma_n$ \emph{bifurcates from} $\sigma$, if at least one of the following holds:
	\begin{enumerate}
		\item There exist $p$ and $0\neq x\in V_p$, such that $\Stab(x)\subset H_\sigma$ is conjugate to $H_{\sigma'}$ in $\SU(n)$.
		\item There exist $p<q$ and $0\neq x\in V_{p,q}$, such that $\Stab(x)\subset H_\sigma$ is conjugate to $H_{\sigma'}$ in $\SU(n)$.
	\end{enumerate}
\end{Definition}

Notice that for every $x\in V_p$, there exists $h\in H_\sigma$ such that $h(x)$ is given by a diagonal matrix. Similarly, for every $x\in V_{p,q}$, there exists $h\in H_\sigma$, such that $h(x)\in \mat_{m_p\times m_q}(\bC)$ has the form 
$$\left(\begin{array}{cccc|ccc}  
\lambda_{1} & 0 & \dots & 0  & 0 & \dots & 0 \\   
0 & \lambda_{2} & \dots & 0  & 0 & \dots & 0 \\
\dots & \dots & \dots & \dots  & 0 & \dots & 0 \\
0 & 0 & \dots & \lambda_{m_p}  & 0 & \dots & 0 \\
\end{array}\right).
$$
Therefore the following lemma follows from a straightforward computation in linear algebra.
\begin{Lemma}
Suppose $\sigma = ((n_1, m_1), \dots, (n_r, m_r))\in\Sigma_n$. Then $\sigma'$ bifurcates from $\sigma$ if and only if $\sigma'$ is given by one of the following, after a permutation of the entries of $\sigma'$:
\begin{enumerate}
\item replacing a pair $(n_p, m_p)$ in $\sigma$ by a sequence $(n_p, m_{1}'), \dots, (n_p, m_{j}')$, such that $m_p = m_{1}' + \cdots m_{j}'$,
\item replacing two pairs $(n_p, m_p), (n_q, m_q)$ in $\sigma$, where $m_p\le m_q$, by a sequence 
$$(n_q, m_q - m_p), (n_p + n_q, m_{1}'), \dots, (n_p + n_q, m_{j}'),$$
 or 
 $$(n_q, m_q - m_p),(n_p + n_q, m_{1}'), \dots, (n_p + n_q, m_{j-1}'), (n_p, m_{j}'),$$
 such that $m_p = m_{1}' + \cdots m_{j}'$.
\end{enumerate}
\qed
\end{Lemma}

\subsection{A closed formula of $\SU(4)$ Casson invariant}
\label{subsec_SU(4)_computation}
We write down an explicit closed formula of $\SU(4)$ Casson invariant using the previous computations. 

The set $\Sigma_4$ has 5 elements. For each $\sigma\in \Sigma_4$, we list the irreducible components that may appear in the isotypic decomposition of the equivariant spectral flows on $\clC^{H_\sigma}$, and introduce a notation for each component of the isotypic decomposition.

\begin{enumerate}
\item If $\sigma=((4,1))$, the equivariant spectral flow is given by trivial representations of $\bZ / 4$. Denote the equivariant spectral flow and the equivariant index by $Sf_{(4,1)}$ and $\ind_{(4,1)}$ respectively.

\item If $\sigma= ((1,1), (3,1))$, then $H_\sigma\cong \U(1)$, and it is give by
$$\diag ( e^{i \alpha}, e^{i \beta}, e^{i \beta}, e^{i \beta} )$$
where $\alpha + 3 \beta \equiv 0 \mod 2\pi.$
 Taking $\beta$ to be the coordinate on $\U(1)$, the Lie algebra $\mathfrak{su}(4)$ is decomposed as $\mathfrak{s}(\mathfrak{u}(1) \oplus \mathfrak{u}(3)) \oplus \bC^{3}$ such that $\U(1)$ acts trivially on the first component and acts by $e^{4 i \beta}$ on each $\bC$ component. Write the latter irreducible representation as $(\U(1)_{((1,1),(3,1)},  \bC, \rho_{4})$. 
 
 Decompose the spectral flow $Sf_{H_\sigma}$ as
 $$Sf_{(1,1), (3,1)} \oplus  Sf_{(1,1), (3,1)^{\perp}},$$
 where $Sf_{(1,1), (3,1)}$ is given by the trivial components, and $Sf_{(1,1), (3,1)^{\perp}}$ is given by the components that are isomorphic to $(\U(1)_{((1,1),(3,1)},  \bC, \rho_{4})$. 
  Decompose the equivariant index \eqref{eqn_def_equiv_index_gauge} similarly as $\ind_{(1,1), (3,1)} \oplus \ind_{(1,1), (3,1)^{\perp}}$.

\item If $\sigma = ((2,1), (2,1))$, then $H_\sigma\cong \U(1)$, and it is given by 
$$\diag ( e^{i \alpha}, e^{i \alpha}, e^{i \beta}, e^{i \beta} )$$ with $\alpha + \beta \equiv 0 \mod 2\pi$. Taking $\beta$ to be the coordinate on $\U(1)$, the Lie algebra $\mathfrak{su}(4)$ is decomposed as $\mathfrak{s}(\mathfrak{u}(2) \oplus \mathfrak{u}(2)) \oplus \bC^{4}$ such that $\U(1)$ acts trivially on the first component and acts by $e^{2 i \beta}$ on each $\bC$ component. Write the latter irreducible representation as $(\U(1)_{((2,1), (2,1))},  \bC, \rho_{2})$. 

Decompose the spectral flow $Sf_{H_\sigma}$ as 
$$Sf_{(2,1), (2,1)} \oplus Sf_{(2,1), (2,1)^{\perp}},$$
 where $Sf_{(2,1), (2,1)}$ is given by the trivial components, and $Sf_{(2,1), (2,1)^{\perp}}$ is given by the components that are isomorphic to $(\U(1)_{((2,1), (2,1))},  \bC, \rho_{2})$. 
Decompose the equivariant index \eqref{eqn_def_equiv_index_gauge} similarly as $\ind_{(2,1), (2,1)} \oplus \ind_{(2,1), (2,1)^{\perp}}$.

\item If $\sigma=(2,2)$, then $H_\sigma\cong \SU(2)$, and  it consists of the matrices of determinant $1$ with the form
\[ \begin{pmatrix}
u_{11} \id & u_{12} \id \\
u_{21} \id & u_{22} \id \\
\end{pmatrix}, \]
where $\id$ is the $2 \times 2$ identity matrix and
\[ \begin{pmatrix}
u_{11} & u_{12} \\
u_{21} & u_{22} \\
\end{pmatrix} \]
is a unitary $2 \times 2$ matrix. The Lie algebra $\mathfrak{su}(4)$ is decomposed as $\mathfrak{l}_{(2,2)} \oplus \mathfrak{su}(2)^{\oplus 4}$ such that $H_\sigma$ acts on the first component trivially and acts on each $\mathfrak{u}(2)$ component by restricting the adjoint action of $\U(2)$. Write the latter irreducible representation as $(\SU(2),  \mathfrak{su}(2), \Ad)$. 

Decompose the equivariant spectral flow $Sf_{H_\sigma}$ as $Sf_{(2,2)} \oplus Sf_{(2,2)^{\perp}}$, 
where $Sf_{(2,2)}$ is given by the trivial components, and $Sf_{(2,2)^{\perp}}$ is given by the components that are isomorphic to $(\SU(2),  \mathfrak{su}(2), \Ad)$. 
Decompose the equivariant index \eqref{eqn_def_equiv_index_gauge} similarly as  $\ind_{(2,2)} \oplus \ind_{(2,2)^{\perp}}$.

\item If $\sigma=((1, 2), (2,1))$, then $H_\sigma$   consists of matrices of the form
\[ \begin{pmatrix}
u_{11}  & u_{12} & 0 & 0 \\
u_{21}  & u_{22} & 0 & 0 \\
0 & 0 & e^{i \alpha} & 0 \\
0 & 0 & 0 & e^{i \alpha} \\
\end{pmatrix} \in\SU(4).
\]
 The Lie algebra $\mathfrak{su}(4)$ is decomposed as $\mathfrak{l}_{((1, 2), (2,1))} \oplus \mathfrak{su}(2) \oplus \mat_{2 \times 1}(\bC)^{\oplus 2}$. Write the last irreducible representation as $(\s(\U(2) \times \U(1)),  \mat(\bC)_{2 \times 1}, \mult)$. 
 
Decompose the equivariant spectral flow $Sf_{H_\sigma}$ as 
$$Sf_{((1, 2), (2,1))} \oplus Sf_{((1, 2), (2,1))^{\perp}}^{(1)} \oplus Sf_{((1, 2), (2,1))^{\perp}}^{(2)},$$
where $Sf_{((1, 2), (2,1))}$ is given by the trivial components, $Sf_{((1, 2), (2,1))^{\perp}}^{(1)}$ is given by the components that are isomorphic to the action of $H_\sigma$ on $\mathfrak{su}(2)$, and $ Sf_{((1, 2), (2,1))^{\perp}}^{(2)}$ is given by the components that are isomorphic to 
$$(\s(\U(2) \times \U(1)),  \mat(\bC)_{2 \times 1}, \mult).$$
 Decompose the equivariant index \eqref{eqn_def_equiv_index_gauge}  as
 $$\ind_{((1, 2), (2,1))} \oplus  \ind_{((1, 2), (2,1))^{\perp}}^{(1)} \oplus  \ind_{((1, 2), (2,1))^{\perp}}^{(2)}.$$
\end{enumerate}

The possible irreducible bifurcations on $\clC$ are given by
$$ ((1,1), (3,1)) \rightarrow ((4,1)), \quad ((2,1), (2,1)) \rightarrow ((4,1)), $$
$$ ((2,2)) \rightarrow ((2,1), (2,1)),  \quad ((1, 2), (2,1)) \rightarrow ((1,1), (3,1)). $$
Note that the isotypical piece $\mathfrak{su}(2)$ corresponding to type $((1, 2), (2,1))$ would induce the bifurcation $((1, 2), (2,1)) \to ((1, 1), (1, 1), (2,1))$, but the latter does not exist for perturbed-flat $\SU(4)$--connections over an integer homology sphere for small perturbations.
In the case of $\SU(4)$, the stabilizers act transitively on the unit spheres of the relevant irreducible representations. Therefore, the only equivariant Morse functions on the unit spheres are the constant functions.

Let $\rho_0$ denote the trivial representations, and let $0$ denote the zero representations. Then the following are the corresponding of values of $\xi_H(V,0,g)$ for the bifurcations above: 
$$[\U(1)_{((1,1),(3,1))},  \bC, \rho_{4}] - [\U(1)_{((1,1),(3,1))},  0]- [\bZ/4, \bR, \rho_0],$$
 $$ [\U(1)_{((2,1),(2,1))},  \bC, \rho_{2}]- [\U(1)_{((2,1),(2,1))},  0] - [\bZ/4, \bR, \rho_0], $$
$$[\SU(2),  \mathfrak{su}(2), \Ad] -[\SU(2),  0]- [\U(1)_{((2,1),(2,1))}, \bR, \rho_0], $$
$$[\s(\U(2) \times \U(1)),  \mat(\bC)_{2 \times 1}, \mult] - [\s(\U(2) \times \U(1)), 0] - [\U(1)_{((1,1),(3,1))}, \bR, \rho_0].$$

Suppose $H$ is a compact Lie group, let
$$
\dim: \clR(H)\to \bZ
$$
be the map given by taking formal dimensions. Then the map
extends linearly to 
$$\dim:\clR(H)\otimes \bR\to\bR,$$
and hence it defines a map
$$
\dim: \widetilde{\clR}_G \to \bR
$$
for any compact Lie group $G$. 
 
 Choose $\pi \in \clP$ to be a generic small perturbation and let $\mathcal{M}_{\pi}$ be the moduli space of $\pi$--flat connections. Then the following formula defines an $\SU(4)$--Casson invariant:
\begin{align*}
&\sum_{[B] \in \mathcal{M}_{\pi}(4,1)} (-1)^{\dim Sf_{(4,1)}(B)} \\
+ &\sum_{[B] \in \mathcal{M}_{\pi}((1,1), (3,1))} (-1)^{\dim Sf_{(1,1), (3,1)}(B)-1} \cdot\frac12  \dim \ind_{(1,1), (3,1)^{\perp}}(B) 
\\
+ &\sum_{[B] \in \mathcal{M}_{\pi}((2,1), (2,1))} (-1)^{\dim Sf_{(2,1), (2,1)}(B)-1} \cdot \frac12 \dim \ind_{(2,1), (2,1)^{\perp}}(B) 
\\
+ & \quad \frac12\,\sum_{[B] \in \mathcal{M}_{\pi}(2,2)} (-1)^{\dim  Sf_{(2,2)}(B)} \cdot \frac13 \dim  \ind_{(2,2)^{\perp}}(B) \cdot \big(\frac13 \dim  \ind_{(2,2)^{\perp}}(B) +1\big)
\\
+ & \quad \frac12\,\sum_{[B] \in \mathcal{M}_{\pi}((1, 2), (2,1))} (-1)^{ \dim Sf_{((1, 2), (2,1))}(B)-1} \cdot \frac14 \dim  \ind_{((1, 2), (2,1))^{\perp}}^{(2)}(B)
\\
& \qquad \qquad \cdot \big(\frac14 \dim  \ind_{((1, 2), (2,1))^{\perp}}^{(2)}(B)+1\big).
\end{align*}

\begin{remark}
	The extra $(-1)$'s on the exponents of $(-1)$ in the above formula comes from the trivial components of the term $-[\ker d_{g(B)}]$ in Definition \ref{def_equivariant_ind_gauge}
\end{remark}

\bibliographystyle{amsalpha}
\bibliography{references}

\begin{bibdiv}
\begin{biblist}

\bib{akbulut1990casson}{article}{
      author={Akbulut, Selman},
      author={McCarthy, John~D},
       title={Casson's invariant for oriented homology 3-spheres: an
  exposition},
        date={1990},
}

\bib{bai2021casson}{article}{
      author={Bai, Shaoyun},
       title={A symplectic formula of generalized {C}asson invariants},
        date={2021},
     journal={arXiv preprint arXiv:2102.03665},
}

\bib{boden1998the}{article}{
      author={Boden, Hans~U},
      author={Herald, Christopher~M},
       title={The {$\SU(3)$} {C}asson invariant for integral homology
  3-spheres},
        date={1998},
     journal={Journal of Differential Geometry},
      volume={50},
      number={1},
       pages={147\ndash 206},
}

\bib{BHK2001}{article}{
      author={Boden, Hans~U},
      author={Herald, Christopher~M},
      author={Kirk, Paul},
       title={An integer valued {$\SU(3)$} casson invariant},
        date={2001},
     journal={Mathematical Research Letters},
      volume={8},
      number={5},
       pages={589\ndash 603},
}

\bib{boyer1990varieties}{article}{
      author={Boyer, Steven},
      author={Nicas, Andrew},
       title={Varieties of group representations and casson’s invariant for
  rational homology 3-spheres},
        date={1990},
     journal={Transactions of the American Mathematical Society},
      volume={322},
      number={2},
       pages={507\ndash 522},
}

\bib{bott1956application}{article}{
      author={Bott, Raoul},
       title={An application of the {M}orse theory to the topology of
  {L}ie-groups},
        date={1956},
     journal={Bulletin de la Soci{\'e}t{\'e} Math{\'e}matique de France},
      volume={84},
       pages={251\ndash 281},
}

\bib{cartan1952topologie}{article}{
      author={Cartan, E},
       title={La topologie des espaces repr{\'e}sentatifs des groupes de {L}ie
  (1936)},
        date={1952},
     journal={Paris, Hermann},
       pages={1307\ndash 1330},
}

\bib{cerf1970stratification}{article}{
      author={Cerf, Jean},
       title={La stratification naturelle des espaces de fonctions
  diff{\'e}rentiables r{\'e}elles et le th{\'e}oreme de la pseudo-isotopie},
        date={1970},
     journal={Publications Math{\'e}matiques de l'IH{\'E}S},
      volume={39},
       pages={5\ndash 173},
}

\bib{cappell2002perturbative}{article}{
      author={Cappell, SE},
      author={Lee, R},
      author={Miller, EY},
       title={A perturbative {$\SU (3)$} casson invariant},
        date={2002},
     journal={Commentarii Mathematici Helvetici},
      volume={77},
      number={3},
       pages={491\ndash 523},
}

\bib{cappell1990symplectic}{article}{
      author={Cappell, Sylvain~E},
      author={Lee, Ronnie},
      author={Miller, Edward~Y},
       title={A symplectic geometry approach to generalized casson’s
  invariants of 3-manifolds},
        date={1990},
     journal={Bulletin of the American Mathematical Society},
      volume={22},
      number={2},
       pages={269\ndash 275},
}

\bib{curtis1994generalized}{article}{
      author={Curtis, Cynthia~L},
       title={Generalized casson invariants for {$\SO (3)$, $\U (2)$, $Spin
  (4)$, and $\SO (4)$}},
        date={1994},
     journal={Transactions of the American Mathematical Society},
      volume={343},
      number={1},
       pages={49\ndash 86},
}

\bib{duistermaat2012lie}{book}{
      author={Duistermaat, Johannes~Jisse},
      author={Kolk, Johan~AC},
       title={Lie groups},
   publisher={Springer Science \& Business Media},
        date={2012},
}

\bib{donaldson1990geometry}{book}{
      author={Donaldson, Simon~K},
      author={Kronheimer, Peter~B},
       title={The geometry of four-manifolds},
   publisher={Oxford University Press},
        date={1990},
}

\bib{donaldson1987orientation}{article}{
      author={Donaldson, Simon~K},
       title={The orientation of {Y}ang-{M}ills moduli spaces and 4-manifold
  topology},
        date={1987},
     journal={Journal of Differential Geometry},
      volume={26},
      number={3},
       pages={397\ndash 428},
}

\bib{doan2017counting}{article}{
      author={Doan, Aleksander},
      author={Walpuski, Thomas},
       title={On counting associative submanifolds and {S}eiberg-{W}itten
  monopoles},
        date={2017},
     journal={arXiv preprint arXiv:1712.08383},
}

\bib{floer1988instanton}{article}{
      author={Floer, Andreas},
       title={An instanton-invariant for 3-manifolds},
        date={1988},
     journal={Communications in mathematical physics},
      volume={118},
      number={2},
       pages={215\ndash 240},
}

\bib{herald2006transversality}{article}{
      author={Herald, Christopher~M},
       title={Transversality for equivariant exact 1-forms and gauge theory on
  3-manifolds},
        date={2006},
     journal={Advances in Mathematics},
      volume={200},
      number={1},
       pages={245\ndash 302},
}

\bib{illman1983equivariant}{article}{
      author={Illman, S{\"o}ren},
       title={The equivariant triangulation theorem for actions of compact
  {L}ie groups},
        date={1983},
     journal={Mathematische Annalen},
      volume={262},
       pages={487\ndash 501},
}

\bib{kronheimer2011knot}{article}{
      author={Kronheimer, Peter~B},
      author={Mrowka, Tomasz~S},
       title={Knot homology groups from instantons},
        date={2011},
     journal={Journal of Topology},
      volume={4},
      number={4},
       pages={835\ndash 918},
}

\bib{marin1988nouvel}{article}{
      author={Marin, Alexis},
       title={Un nouvel invariant pour les spheres d'homologie de dimension
  trois},
        date={1988},
     journal={Seminare Bourbaki, 1987--88, Asterisque N},
      volume={693},
       pages={151},
}

\bib{mrowka2011seiberg}{article}{
      author={Mrowka, Tomasz},
      author={Ruberman, Daniel},
      author={Saveliev, Nikolai},
       title={Seiberg-{W}itten equations, end-periodic dirac operators, and a
  lift of {R}ohlin's invariant},
        date={2011},
     journal={Journal of Differential Geometry},
      volume={88},
      number={2},
       pages={333\ndash 377},
}

\bib{nakajima2016towards}{article}{
      author={Nakajima, Hiraku},
       title={Towards a mathematical definition of {C}oulomb branches of
  {$3$}-dimensional {$\mathcal {N} = 4$} gauge theories, {I}},
        date={2016},
     journal={Advances in Theoretical and Mathematical Physics},
      volume={20},
      number={3},
       pages={595\ndash 669},
}

\bib{taubes1990casson}{article}{
      author={Taubes, Clifford~H},
       title={Casson's invariant and gauge theory},
        date={1990},
     journal={Journal of Differential Geometry},
      volume={31},
      number={2},
       pages={547\ndash 599},
}

\bib{taubes1996counting}{article}{
      author={Taubes, Clifford~H},
       title={Counting pseudo-holomorphic submanifolds in dimension 4},
        date={1996},
     journal={Journal of Differential Geometry},
      volume={44},
      number={4},
       pages={818\ndash 893},
}

\bib{tom-Dieck}{book}{
      author={tom Dieck, Tammo},
       title={Transformation groups},
      series={De Gruyter Studies in Mathematics},
   publisher={Walter de Gruyter \& Co., Berlin},
        date={1987},
      volume={8},
}

\bib{walker1992extension}{book}{
      author={Walker, Kevin},
       title={An extension of casson's invariant},
   publisher={Princeton University Press},
        date={1992},
      number={126},
}

\bib{wasserman1969equivariant}{article}{
      author={Wasserman, Arthur~G},
       title={Equivariant differential topology},
        date={1969},
     journal={Topology},
      volume={8},
      number={2},
       pages={127\ndash 150},
}

\bib{wendl2016transversality}{article}{
      author={Wendl, Chris},
       title={Transversality and super-rigidity for multiply covered
  holomorphic curves},
        date={2016},
     journal={arXiv preprint arXiv:1609.09867},
}

\bib{whitney1965local}{incollection}{
      author={Whitney, Hassler},
       title={Local properties of analytic varieties},
        date={1965},
   booktitle={Hassler whitney collected papers},
   publisher={Springer},
       pages={497\ndash 536},
}

\end{biblist}
\end{bibdiv}

\end{document}